\documentclass[10pt,a4paper]{article}
\usepackage[left=1.75cm,right=1.75cm,top=2cm,bottom=2cm]{geometry} 

\usepackage{siunitx}
\usepackage{subcaption}
\usepackage{amssymb}
\usepackage{amsmath}
\usepackage{amsthm}
\usepackage{xcolor}
\usepackage{graphicx}
\usepackage{epstopdf}
\usepackage{enumerate}
\usepackage{comment}
\usepackage{psfrag}
\usepackage[breaklinks,bookmarks=false]{hyperref}
\usepackage[textsize=tiny,color=yellow!60!green]{todonotes}
\usepackage{graphicx}
\usepackage{color}
\usepackage{mathbbol}
\usepackage{amssymb}   
\DeclareSymbolFontAlphabet{\amsmathbb}{AMSb}%
\usepackage{float}
\usepackage{array}
\usepackage{multirow}
\usepackage{dsfont}
\usepackage{mathrsfs}
\usepackage{amsfonts}
\usepackage{amssymb}

\usepackage{algorithm}
\usepackage{algpseudocode}

\let\epsilon\varepsilon
\let\phi\varphi

\newcommand{\bse}{\begin{subequations}}
\newcommand{\ese}{\end{subequations}}

\newcommand{\lmin}{\lambda_{\textnormal{min}}}
\newcommand{\lmax}{\lambda_{\textnormal{max}}}

\newcommand{\pt}{\partial_t}

\newcommand{\dx}{\textnormal{d}x}

\newcommand{\bP}{\mathbf{P}}

\newcommand{\bid}{\mathbf{I}}

\newcommand{\bu}{\mathbf{u}}

\newcommand{\bv}{\mathbf{v}}

\newcommand{\be}{\mathbf{e}}

\newcommand{\bsig}{\boldsymbol{\sigma}}
\newcommand{\blambda}{\boldsymbol{\Lambda}}

\DeclareMathOperator{\diam}{diam}

\DeclareMathOperator{\ess}{ess}

\newcommand{\Tcal}{\mathcal{T}}

\newcommand{\Ccal}{\mathcal{C}}

\newcommand{\real}{{\amsmathbb R}}
\newcommand{\nat}{{\amsmathbb N}}
\newcommand{\bC}{\amsmathbb{C}}

\def\eps{\varepsilon}

\newcommand{\norm}[1]{\lVert#1\rVert}

\newtheorem{theorem}{Theorem}[section]

\newtheorem{remark}{Remark}[section]

\newtheorem{assum}{Assumption}
\newtheorem{Algo}{Algorithm}

\makeatletter
\newcommand{\pushright}[1]{\ifmeasuring@#1\else\omit\hfill$\displaystyle#1$\fi\ignorespaces}
\newcommand{\pushleft}[1]{\ifmeasuring@#1\else\omit$\displaystyle#1$\hfill\fi\ignorespaces}
\makeatother

\begin{document}

\title{An iterative staggered scheme for phase field brittle fracture propagation with stabilizing parameters}
\author{Mats Kirkes\ae{}ther Brun\footnotemark[2]\and
Thomas Wick\footnotemark[3]
\and Inga Berre\footnotemark[2] \footnotemark[4]
\and Jan Martin Nordbotten\footnotemark[2]\ \footnotemark[5]
\and Florin Adrian Radu\footnotemark[2] 
}
\date{\today}
\maketitle

\renewcommand{\thefootnote}{\fnsymbol{footnote}}

\footnotetext[2]{Department of Mathematics, University of Bergen, P. O. Box 7800, N-5020 Bergen, Norway.
\href{mailto:mats.brun@uib.no}{mats.brun@uib.no},
\href{mailto:inga.berre@uib.no}{inga.berre@uib.no},
\href{mailto:jan.nordbotten@uib.no}{jan.nordbotten@uib.no},
\href{mailto:florin.florin.radu@uib.no}{florin.radu@uib.no}
}
\footnotetext[3]{Leibniz University Hannover, Institute of Applied
  Mathematics, AG Wissenschaftliches Rechnen, Welfengarten 1, 30167 Hannover.
\href{mailto:thomas.wick@ifam.uni-hannover.de}{thomas.wick@ifam.uni-hannover.de}
}
\footnotetext[4]{NORCE Norwegian Research Centre AS, Bergen, Norway.}
\footnotetext[5]{Department of Civil and Environmental Engineering, Princeton University, Princeton, N. J., USA.}
\renewcommand{\thefootnote}{\arabic{footnote}}

\numberwithin{equation}{section}

\begin{abstract}
This paper concerns the analysis and implementation of a novel iterative staggered scheme for quasi-static brittle fracture propagation models, where the fracture evolution is tracked by a phase field variable. The
model we consider is a two-field variational inequality system, with the phase
field function and the elastic displacements of the solid material as independent variables. Using a penalization strategy, this variational inequality system is transformed into a variational equality system, which is the formulation we take as the starting point for our algorithmic developments. The proposed scheme involves a partitioning of this model into two subproblems;~\emph{phase field} and~\emph{mechanics}, with added stabilization terms to both subproblems for improved efficiency and robustness. We analyze the convergence of the proposed scheme using a fixed point argument, and find that under a natural condition, the elastic mechanical energy remains bounded, and, if the diffusive zone around crack surfaces is sufficiently thick, monotonic convergence is achieved. Finally, the proposed scheme is validated numerically with several bench-mark problems. 


\end{abstract}
\vspace{3mm}

\noindent{\bf Key words:} phase field; quasi-static; brittle fracture; fracture propagation; $L$-scheme; fixed stress; iterative algorithm; linearization; convergence analysis; fixed point; finite element;

\pagestyle{myheadings} \thispagestyle{plain} \markboth{M. K. Brun}{An adaptive iterative scheme for phase field brittle fracture propagation.}

\section{Introduction}
Fracture propagation is currently an important topic with many applications in 
various engineering fields. Specifically, phase-field descriptions are intensively investigated. The theory of brittle fracture mechanics goes back to the works of A. Griffith~\cite{griffith1921vi}, wherein a criterion for crack propagation is formulated. Despite a foundational treatment on the subject of brittle fracture, Griffith's theory fails to predict crack initiation. This deficiency can however be overcome by a variational approach, which was first proposed in~\cite{BourFraMar00,francfort1998revisiting}. Using such a variational approach, discontinuities in the displacement field $u$ across the lower-dimensional crack surface are approximated 
by an auxiliary phase-field function $\varphi$.
The latter can be viewed as an indicator function, which
introduces a diffusive transition zone between the broken and the unbroken material. 
The enforcement of irreversibility of crack growth finally 
yields a variational inequality system, of
which we seek the solution $\{u, \phi\}$.

In this work, we concentrate on improvements 
of the nonlinear solution algorithm, which is still 
a large bottleneck of phase-field fracture evolution problems.
Specifically, high iteration numbers when the crack initiates or is further growing are reported in 
many works~\cite{GeLo16,MesBouKhon15,Wi17_SISC,wick2017modified}. 
However, in most studies iteration numbers are omitted. Both staggered (splitting) schemes and monolithic schemes are frequently employed.
Important developments include 
alternating minimization/staggered schemes~\cite{Bour07,BourFraMar08,BuOrSue10,MesBouKhon15,miehe2010phase}, 
quasi-monolithic scheme with 
a partial linearization~\cite{HeWheWi15}, and fully monolithic 
schemes~\cite{GeLo16,Wi17_SISC,wick2017modified}.

The goal of this work is to propose 
a~\emph{linearized staggered scheme with stabilizing parameters}. In particular, the proposed scheme is based on recent developments on~\emph{iterative splitting schemes} coming from poroelasticity~\cite{castelletto2015accuracy, kim2009stability, mikelic2014numerical, mikelic2013convergence}. Iterative splitting schemes are widely applied to problems of coupled flow and mechanics, where at each iteration step either of the subproblems (i.e., flow or mechanics) is solved first, keeping some physical quantity constant (e.g., \emph{fixed stress} or \emph{fixed strain}), followed by solving the next subproblem with updated solution information. This procedure is then repeated until an accepted tolerance is reached. Further extensions of this technique involves tuning some artificial stabilization terms according to a derived contraction estimate in energy norms. Here, the quantity held constant during solving of the subproblems need not represent any physical quantity present in the model. This is the central idea in the so-called `$L$-scheme', which has proven to perform robustly for Richards equation~\cite{list2016study, MR2079503}, for linear and nonlinear poroelasticity~\cite{MR3827264,both2017robust}, and for nonlinear thermo-poroelasticity~\cite{2019arXiv190205783K}.

We propose here a variant of the $L$-scheme, adapted to phase field brittle fracture propagation models. This scheme is based on a partitioning of the model into two subproblems; \emph{phase field} and \emph{mechanics}. Here, the $L$-scheme acts both as a~\emph{stabilization} and as a~\emph{linearization} (as a linearization scheme, the stabilization parameters mimics the Jacobian from Newton iteration). Assuming that the mechanical elastic energy remains bounded during the iterations, and that the diffusive zone around crack surfaces is sufficiently thick, we give a proof of monotonic convergence of the proposed scheme by employing a fixed point argument. 


The efficiency and robustness of the proposed scheme is demonstrated numerically with several bench-mark problems. Moreover, we compare the number of iterations needed for convergence with `standard' staggered schemes (i.e., without stabilizing terms), and monolithic schemes in which the fully-coupled system is solved all-at-once. Furthermore, it is well known that when reaching the critical loading steps during the computation of brittle fracture phase field problems (i.e., when the crack is propagating), spikes in iteration numbers appear. For this reason, and thanks to the monotonic convergence property of the proposed scheme, we show that a (low) upper bound on the number of iterations may be enforced, while the computed results are still in very good agreement with the non-truncated solutions. Thus, using this `truncated $L$-scheme', we effectively avoid the iteration spikes at the critical loading steps at the cost of negligible loss of accuracy. We mention that this strategy is not available with e.g. Newton iteration, as the iterate solutions may behave erratically for any number of iterations before finally converging. Moreover, the assumption that the mechanical elastic energy remains bounded during the iterations is verified numerically for all tests cases.


\emph{The main aims of this work are three-fold:} Under a natural assumption, we prove the convergence of a novel iterative staggered scheme, optimized for phase field brittle fracture propagation problems. Based on these theoretical findings, we design a robust solution algorithm with monotonic convergence properties. Finally, several numerical tests are presented in which our variants of the $L$-scheme are tested in detail. 

The outline of this paper is as follows: In Section~\ref{sec:governingequations} we present the model equations and coefficients, in Section~\ref{sec:iterativescheme} we introduce the partitioned scheme and derive a convergence proof, in Section~\ref{sec:algorithms} we describe in detail our numerical algorithm in pseudo-code, and in Section~\ref{sec_tests} we provide several numerical experiments, in particular the~\emph{single edge notched tension test}, the~\emph{single edge notched shear test}, and the \emph{L-shaped panel test}. Finally, in Section~\ref{sec:conclusions} we provide some conclusions and summary of the work.

\subsection{Preliminaries}\label{sec:preliminaries}
In this section we explain the notation used throughout this article, see e.g.~\cite{evans1998partial, yosida1995functional} for more details. Given an open and bounded set $B \subset \real^d$, $d \in \{2,3\}$, and $1\leq p< \infty$, let $L^p(B) = \{f:B \rightarrow \real : \int_B |f(x)|^p \dx < \infty \}$. For $p = \infty$, let $L^\infty(B) = \{ f:B \rightarrow \real : \ess \sup_{x \in B} |f(x)| < \infty \}$. In particular, $L^2(B)$ is the Hilbert space of square integrable functions with inner product $(\cdot, \cdot)$ and norm $\norm{f} : = (f,f)^{\frac{1}{2}}$. For $k \in \nat$, $k \geq 0$, we denote by $W^{k,p}(B)$ the space of functions in $L^p(B)$ admitting weak derivatives up to $k$'th order. In particular, $H^1(B) := W^{1,2}(B)$ and we denote by $H_0^1(B)$ its zero trace subspace. 

Note that we reserve the use of bold fonts for second order tensors. Hence, if $u,v \in L^2(B)$, their inner product is $(u,v) := \int_B u(x)v(x)\dx$, and similarly, if $u,v \in (L^2(B))^d$ then we take their inner product to be $(u,v) := \int_B u(x) \cdot v(x) \dx$. Finally, if $\bu, \bv \in (L^2(B))^{d\times d}$ then their inner product is $(\bu, \bv) := \int_B \bu(x) : \bv(x) \dx$. 

We will also frequently apply several classical inequalities, in particular: \emph{Cauchy-Schwarz, Young, Poincar\'e}, and \emph{Korn}. See e.g.~\cite{MR1765047, hardy1967inequalities} for a detailed description of these.


\section{Governing equations}\label{sec:governingequations}
What follows is a brief description of the phase field approach for quasi-static brittle fracture propagation, see e.g.~\cite{francfort1998revisiting, miehe2010phase} for more details. Consider a (bounded open) polygonal domain $B \subset \real^d$, wherein $\Ccal \subset \real^{d-1}$ denotes the fracture, and $\Omega \subset \real^d$ is the intact domain, and a time interval $(0,T)$ is given with final time $T>0$. By introducing the phase field variable $\phi : B  \times (0,T) \rightarrow [0,1]$, which takes the value $0$ in the fracture, $1$ in the intact domain, and varies smoothly from $0$ to $1$ in a transition zone of (half-)thickness $\epsilon > 0$ around $\Ccal$, the evolution of the fracture can be tracked in space and time. Using the phase field approach, the fracture $\Ccal$ is approximated by $\Omega_F \subset \real^d$, where $\Omega_F := \{ x \in \real^d : \phi(x) < 1\}$.

Introducing the displacement vector $u : B \times (0,T) \rightarrow \real^d$, the model problem we consider arises as a minimization problem: An energy functional $E(u,\phi)$ is defined according to Griffith's criterion for brittle fracture \cite{griffith1921vi}, which is then sought to be minimized over all admissible $\{u,\phi\}$. From this minimization problem, the Euler-Lagrange equations are obtained by differentiation with respect to the arguments, yielding a variational equality system. Finally, a crack irreversibility condition must be enforced (the crack is not allowed to heal), which takes the form $\pt \phi \leq 0$. Thus, the variational equality system, which is the previously mentioned Euler-Lagrange equations, is transformed into a variational inequality system, which reads as follows: Find $(u(t),\phi(t)) \in V\times W :=  (H_0^1(B))^d \times W^{1,\infty}(B)$ such that for $t \in (0,T]$ there holds
\bse
\begin{alignat}{2}
\left( g(\phi) \bC \be(u), \be(v) \right) &= (b,v), &\forall v \in V, \label{var2} \\
G_c \epsilon(\nabla \phi, \nabla \psi) - \frac{G_c}{\epsilon} (1 - \phi, \psi) + (1-\kappa)(\phi |\bC \be(u)|^2,\psi) &\geq 0, &\forall \psi \in W,  \label{var1} 
\end{alignat}
\ese
where $G_c > 0$ is the critical elastic energy restitution rate, $0 < \kappa << 1$ is a regularization parameter, the purpose of which is to avoid degeneracy of the elastic energy (equivalent with replacing the fracture with a softer material), and $g(\phi) := (1-\kappa)\phi^2 + \kappa$ is a standard choice for the degradation function (see e.g.~\cite{MR3778706,wick2017modified}. Note that $g(\phi) \rightarrow \kappa$ when approaching the fracture zone). The body force acting on the domain $B$ is $b : B\times (0,T) \rightarrow \real^d$, and $|\bC \be(u)|^2 := \bC \be(u):\be(u)$ is the elastic mechanical energy, where $\be(\cdot) := (\nabla(\cdot) + \nabla(\cdot)^\top)/2$ is the symmetric gradient, and $\bC = [C_{ijkl}]_{ijkl}$ is the fourth order tensor containing the elastic material coefficients, where each $C_{ijkl} \in L^\infty(B)$. We assume that $\bC$ satisfies the usual~\emph{symmetry} and~\emph{positive definiteness} properties, i.e., $(\bC \bu, \bv) = (\bu, \bC\bv)$, and $(\bC \bu, \bu)^{\frac{1}{2}}$ defines an $L^2$-equivalent norm, i.e., there exists constants $\lambda_m, \lambda_M > 0$ such that
\begin{equation}\label{bCprop}
\lambda_m \norm{\bu} \leq (\bC \bu, \bu)^{\frac{1}{2}} \leq  \lambda_M \norm{\bu}, \quad \text{ for } \quad \bu, \bv \in (L^2(B))^{d\times d}, \ \bu, \bv \neq \mathbf{0}.
\end{equation}
 
In order to facilitate the following developments we assume continuity in time for $\{u,\phi,b\}$. Let now $0 = t^0 < t^1 < \cdots < t^N = T$ be a partition of the time interval $(0,T)$, with time step $\delta t := t^n - t^{n-1}$, and denote the time discrete solutions by 
\begin{align}
u^n &:= u(\cdot, t^n), \\
\phi^n &:= \phi(\cdot, t^n).
\end{align}
The irreversibility condition now becomes $\phi^n \leq \phi^{n-1}$ (using a backward Euler method), and the time-discrete version of the problem \eqref{var2}--\eqref{var1} reads as follows: Find $(u^n, \phi^n) \in V \times W$ such that 
\bse 
\begin{alignat}{2}
&\left( g(\phi^n) \bC \be(u^n), \be(v) \right) = (b^n,v),  &\qquad \forall v \in V, \label{ge2} \\
\nonumber &\quad G_c \epsilon(\nabla \phi^n, \nabla \psi) - \frac{G_c}{\epsilon} (1 - \phi^n, \psi) 
+ (1-\kappa)(\phi^n |\bC \be(u^n)|^2,\psi) &\\
&\qquad+ ([\Xi + \gamma (\phi^n - \phi^{n-1})]^{+}, \psi) = 0, \quad&\qquad \forall \psi \in W,  \label{ge1} 
\end{alignat}
\ese
where $b^n := b(\cdot, t^n)$. The last term in the phase field equation~\eqref{ge1} is a penalization to enforce the irreversibility condition, thus transforming the variational inequality \eqref{var1} into a variational equality, with penalization parameter $\gamma > 0$, and where $\Xi \in L^2(B)$ is given (in practice $\Xi$ will be obtained by iteration, cf. Section~\ref{sec:algorithms}). Note that we also used the notation $[x]^{+} := \max(x,0)$. From here on, we shall refer to \eqref{ge2} as the~\emph{mechanics subproblem}, and to \eqref{ge1} as the~\emph{phase field subproblem}. Regarding the degradation function $g$, it is easily seen to satisfy the following Lipschitz condition:
\begin{equation}\label{gassumptions}
\norm{g(\psi) - g(\eta)} \leq 2(1-\kappa) \norm{\psi - \eta}, \quad \forall \psi, \eta \in W. 
\end{equation}
The time-discrete system \eqref{ge2}-\eqref{ge1} was analyzed in~\cite{neitzel2017optimal}, and there it was shown that at least one global minimizer $(u^n, \phi^n) \in V \times W$ exists, provided $b^n \in (L^2(B))^d$, for each $n$. We mention also that the analysis of a pressurized phase field brittle fracture model can be found in~\cite{MiWheWi15b,MiWheWi19}.
\section{Iterative scheme}\label{sec:iterativescheme}
In this section we introduce the iterative staggered solution procedure for the fully discrete formulation of \eqref{ge2}-\eqref{ge1}. To this end, let $\Tcal_h$ be a simplicial mesh of $B$, such that for any two distinct elements of $\Tcal_h$ their intersection is either an empty set or their common vertex or edge. We denote by $h$ the largest diameter of all the elements in $\Tcal_h$, i.e., $h := \max_{K \in \Tcal_h} \diam(K)$, and let $V_h \times W_h \subset V\times W$ be appropriate (conforming) discrete spaces. We continue now with the same notation for the variables and test functions as before (omitting the usual $h$-subscript), since we will from here on mostly deal with the discrete solutions.

For each $n$, the iterative algorithm we propose defines a sequence $\{u^{n,i}, \phi^{n,i}\}$, for $i\geq 0$, initialized by $\{u^{n-1}, \phi^{n-1}\}$. The iteration is then done in two steps: First, the mechanics subproblem is solved, with the degradation function held constant. Then, the phase field subproblem is solved, with the elastic energy held constant. Note that there are also artificial stabilizing terms which are held constant during solving of the subproblems. Introducing the stabilization parameters $L_u, L_\phi > 0$ (to be determined later), the iterative algorithm reads as follows:
\bse
\begin{align}
\nonumber &\textnormal{\textbullet \ \textbf{Step 1}: Given $(u^{n,i-1}, \phi^{n,i-1}, b^n)$ find $u^{n,i}$ such that} \\
&\qquad {a_u(u^{n,i},v)} := L_u(u^{n,i} - u^{n,i-1},v) + \left( g(\phi^{n,i-1}) \bC \be(u^{n,i}), \be(v) \right) =(b^n, v),  &\forall v \in V_h.  \label{iter2} \\
\nonumber &\textnormal{\textbullet \ \textbf{Step 2}: Given $(\phi^{n,i-1}, u^{n,i}, \phi^{n-1})$ find $\phi^{n,i}$ such that} \\
\nonumber &\qquad {a_\phi(\phi^{n,i},\psi)} := L_\phi (\phi^{n,i} - \phi^{n,i-1},\psi) + G_c \epsilon(\nabla \phi^{n,i}, \nabla \psi) - \frac{G_c}{\epsilon} (1 - \phi^{n,i}, \psi) &\\
&\qquad \qquad+ (1-\kappa)(\phi^{n,i} |\bC \be(u^{n,i})|^2,\psi) 
+ (\eta^i (\Xi + \gamma (\phi^{n,i} - \phi^{n-1})), \psi)= 0, &\forall \psi \in W_h, \label{iter1} 
\end{align}
\ese
where, in order to avoid the $[\cdot]^+$-bracket, we also introduced the function $\eta^i \in L^\infty(B)$ defined for a.e. $x \in B$ by
\begin{equation}
\eta^i(x) = 
\begin{cases}
1, \quad \textnormal{ if } \ \Xi(x) + \gamma(\phi^{n,i}(x) - \phi^{n-1}(x)) \geq 0, \\
0, \quad \textnormal{ if } \ \Xi(x) + \gamma(\phi^{n,i}(x) - \phi^{n-1}(x)) < 0.
\end{cases}
\end{equation}
\subsection{Convergence analysis}
We now proceed to analyze the convergence of the scheme \eqref{iter2}-\eqref{iter1}. Our aim is to show a contraction of successive difference functions in energy norms, which implies convergence by the Banach Fixed Point Theorem (see e.g.~\cite{cheney2013analysis}). To this end we define the following difference functions 
\begin{align}
e_u^i &:= u^{n,i} - u^n,\\
e_\phi^i &:= \phi^{n,i} - \phi^n,
\end{align}
where $\{u^n, \phi^n\}$ denotes the (exact) solutions to \eqref{var2}-\eqref{var1} at time $t^n$. Using the symmetry properties of $\bC$, the following set of difference equations are then obtained by subtracting \eqref{iter2}-\eqref{iter1} solved by $\{u^n, \phi^n\}$ from the same equations solved by the iterate solutions:
\bse
\begin{alignat}{2}
&L_u(e_u^i - e_u^{i-1}, v) + (g(\phi^{n})\bC \be(e_u^i), \be(v)) + ((g(\phi^{n,i-1}) - g(\phi^n))\bC \be(u^{n,i}), \be(v)) =0, &\forall v \in V_h. \label{itere2} \\
\nonumber &L_\phi(e_\phi^i - e_\phi^{i-1}, \psi) + G_c \epsilon (\nabla e_{\phi}^i, \nabla \psi) 
+ \frac{G_c}{\epsilon}(e_\phi^i, \psi) + \gamma (\eta^ie_\phi^i, \psi) 
+ (1-\kappa)(e_\phi^{i} |\bC \be(u^{n,i})|^2, \psi) \\
&\qquad \qquad \quad \ + (1 - \kappa)\left(\phi^n \bC \be(e_u^{i}):\be(u^{n,i} + u^{n}), \psi \right) 
 = 0, &\forall \psi \in W_h. \label{itere1}
\end{alignat}
\ese
Furthermore, we introduce the following assumption related to the elastic mechanical strain.
\begin{assum}[Boundedness of elastic strain]\label{hypo}
We assume there exists a constant $M>0$ such that
\begin{equation}
\ess \sup_{x\in B} |\be(u^{n}(x))| \leq M, \quad \forall n.
\end{equation}
Moreover, we assume that $M$ is large enough such that the above bound holds also for the iterate elastic strain, i.e.,
\begin{equation}
\ess \sup_{x\in B} |\be(u^{n,i}(x))| \leq M, \quad \forall (n,i).
\end{equation}
\end{assum}
Note that $M$ is nothing else than an upper bound for the elastic strain in the system for the converged solution, which is arguably finite for any reasonable problem. Note also that with sufficient regularity of the domain, coefficients, source terms, and initial data, the above assumption is satisfied, i.e., the problem \eqref{ge2}-\eqref{ge1} admits a solution $u^n \in (W^{1,\infty}(B))^d$, thus implying the existence of $M$. Alternatively to introducing the constant $M$, we could introduce instead a so-called `cut-off operator' in the iterate equations \eqref{iter2}-\eqref{iter1}, as seen in e.g. \cite{MR2039576, MR2116915}. Note that in all numerical tests to be done in the next sections, we provide figures validating the second part of this assumption (cf. Section~\ref{verify}). With the above definitions, we state our main theoretical result.
\begin{theorem}[Convergence of the scheme]\label{convthm}
The scheme~\eqref{iter1}--\eqref{iter2} defines a contraction satisfying
\begin{align}\label{convrate}
\nonumber &\left(\frac{L_\phi}{2} + \frac{G_c}{\epsilon} + \frac{G_c \epsilon}{c_P} - 8\xi\frac{(1-\kappa)^2}{\kappa} \right) \norm{e_\phi^i}^2 
 + \left(\frac{L_u}{2} + \frac{\kappa \lmin^2}{2c_Pc_K}\right)\norm{e_u^{i}}^2 \\
&\qquad \leq \left(\frac{L_\phi}{2} + 8\xi\frac{(1-\kappa)^2}{\kappa}\right) \norm{e_\phi^{i-1}}^2
+ \frac{L_u}{2} \norm{e_u^{i-1}}^2,
\end{align}
if $L_u, L_\phi > 0$, and if the model parameter $\epsilon > 0$ is sufficiently large such that
\begin{equation}
\epsilon^2 - 16\xi\frac{(1-\kappa)^2}{\kappa}\frac{c_P}{G_c}\epsilon + c_P > 0,
\end{equation}
where $\xi := (M\lmax/\lmin)^2 > 0$, and where $c_P, c_K > 0$ are the Poincar\'e and Korn constants, respectively, depending only on the domain $B$ and spatial dimension $d$.
\end{theorem}
\begin{proof}
We begin by taking $v = e_u^i$ and $\psi = e_\phi^i$ in \eqref{itere2} and \eqref{itere1}, respectively, add the resulting equations together and obtain
\begin{align}
\nonumber &\left(\frac{L_\phi}{2} + \frac{G_c}{\epsilon}\right) \norm{e_\phi^i}^2 + \frac{L_\phi}{2} \norm{e_\phi^i - e_\phi^{i-1}}^2 + G_c\epsilon \norm{\nabla e_\phi^i}^2 + \gamma(\eta^i e_\phi^i, e_\phi^i)\\
\nonumber &\quad + (1-\kappa)(e_\phi^i |\bC \be(u^{n,i})|^2, e_\phi^i) + \frac{L_u}{2}\norm{e_u^i}^2 + \frac{L_u}{2} \norm{e_u^i - e_u^{i-1}}^2 + (g(\phi^n) \bC \be(e_u^i),\be(e_u^i))\\
\nonumber&\qquad= \frac{L_\phi}{2} \norm{e_\phi^{i-1}}^2 + \frac{L_u}{2} \norm{e_u^{i-1}}^2 - (1-\kappa)(\phi^n \bC \be(e_u^{i}):\be(u^{n,i} + u^n), e_\phi^i) \\
&\qquad \quad - ((g(\phi^{n,i-1}) - g(\phi^n))\bC \be(u^{n,i}), \be(e_u^{i})), \label{est1}
\end{align}
where we used the following inner product identity
\begin{equation}
2(x-y,x) = \norm{x}^2 + \norm{x-y}^2 - \norm{y}^2.
\end{equation}
Discarding some non-negative terms from the left hand side of \eqref{est1}, using the fact that $\ess \sup_{x\in B} \phi^n(x) \leq 1$, in addition to the Lipschitz property of the degradation function $g$ \eqref{gassumptions}, yields
\begin{align}
\nonumber &\left(\frac{L_\phi}{2} + \frac{G_c}{\epsilon}\right) \norm{e_\phi^i}^2 
 + G_c\epsilon\norm{\nabla e_\phi^i}^2
+ \frac{L_u}{2}\norm{e_u^i}^2 +  \kappa (\bC\be(e_u^i),\be(e_u^i))\\
\nonumber&\quad \leq \frac{L_\phi}{2} \norm{e_\phi^{i-1}}^2 + \frac{L_u}{2} \norm{e_u^{i-1}}^2 + (1-\kappa) \int_B |\bC \be(e_u^{i}):\be(u^{n,i} + u^n) e_\phi^i|\dx \\
\nonumber&\qquad+ \int_B |(g(\phi^{n,i-1}) - g(\phi^n)) \bC \be(u^{n,i}):\be(e_u^i)|\dx\\
&\qquad \quad \leq \frac{L_\phi}{2} \norm{e_\phi^{i-1}}^2 + \frac{L_u}{2} \norm{e_u^{i-1}}^2 + 
2(1-\kappa)\lmax M\Big(\norm{e_\phi^i} + \norm{e_\phi^{i-1}}\Big)\norm{\be(e_u^i)},\label{est2}
\end{align}
where we also invoked the Assumption~\ref{hypo} in the last line, and applied the Cauchy-Schwarz inequality. Using the Young inequality, the properties of elastic tensor \eqref{bCprop}, and rearranging, leads to
\begin{align}
\nonumber &\left(\frac{L_\phi}{2} + \frac{G_c}{\epsilon} - 2(1-\kappa)\lmax M\frac{1}{2\delta_1}\right) \norm{e_\phi^i}^2 
 + G_c\epsilon\norm{\nabla e_\phi^i}^2\\
\nonumber&\quad+ \frac{L_u}{2}\norm{e_u^i}^2 + \Big(\kappa \lmin^2 - 2(1-\kappa)\lmax M(\delta_1 + \delta_2) \Big) \norm{\be(e_u^i)}^2\\
&\qquad \leq \left(\frac{L_\phi}{2} + 2(1-\kappa)\lmax M \frac{1}{2\delta_2}\right)\norm{e_\phi^{i-1}}^2
+ \frac{L_u}{2} \norm{e_u^{i-1}}^2,\label{est3}
\end{align}
for some constants $\delta_1, \delta_2 > 0$. Choosing $\delta_1 = \delta_2 = \kappa \lmin^2 / 8(1-\kappa)\lmax M$ yields \eqref{est3} as 
\begin{align}
\nonumber &\left(\frac{L_\phi}{2} + \frac{G_c}{\epsilon} - 8\xi\frac{(1-\kappa)^2}{\kappa}\right) \norm{e_\phi^i}^2 
 + G_c\epsilon\norm{\nabla e_\phi^i}^2
 + \frac{L_u}{2}\norm{e_u^i}^2 +  \frac{\kappa \lmin^2}{2} \norm{\be(e_u^i)}^2\\
&\qquad \leq \left(\frac{L_\phi}{2} + 8\xi\frac{(1-\kappa)^2}{\kappa}\right) \norm{e_\phi^{i-1}}^2
+ \frac{L_u}{2} \norm{e_u^{i-1}}^2.\label{est4}
\end{align}
Next, by applying the Poincar\'e inequality on $\norm{e_\phi^{i}}$, and by applying successively the Poincar\'e and Korn inequalities on $\norm{e_u^{i}}$, we obtain
\begin{equation}\label{est5}
\norm{e_\phi^{i}}^2 \leq c_P \norm{\nabla e_\phi^{i}}^2 \quad \text{ and } \quad 
\norm{e_u^{i}}^2 \leq c_P c_K\norm{\be(e_u^{i})}^2,
\end{equation}
where $c_P, c_K$ are the (squares of the) Poincar\'e and Korn constants, respectively (depending only on the domain $B$ and spatial dimension $d$). Finally, employing these bounds on the left hand side of \eqref{est4} yields 
\begin{align}
\nonumber &\left(\frac{L_\phi}{2} + \frac{G_c}{\epsilon} + \frac{G_c \epsilon}{c_P} - 8\xi\frac{(1-\kappa)^2}{\kappa} \right) \norm{e_\phi^i}^2 
 + \left(\frac{L_u}{2} + \frac{\kappa \lmin^2}{2c_Pc_K}\right)\norm{e_u^{i}}^2 \\
&\qquad \leq \left(\frac{L_\phi}{2} + 8\xi\frac{(1-\kappa)^2}{\kappa}\right) \norm{e_\phi^{i-1}}^2
+ \frac{L_u}{2} \norm{e_u^{i-1}}^2.\label{est6}
\end{align}
Thus, for \eqref{est6} to be a contraction estimate, $\epsilon$ must satisfy the following second order inequality
\begin{equation}\label{epsineq}
P(\epsilon) := \epsilon^2 - 16\xi\frac{c_P}{G_c}\frac{(1-\kappa)^2}{\kappa}\epsilon + c_P > 0.
\end{equation}
Setting the left hand side of \eqref{epsineq} equal to zero yields a second order polynomial, the discriminant of which must satisfy one of the following three statements:
\begin{enumerate}
\item
If $$64\xi^2\dfrac{(1-\kappa)^4}{\kappa^2} > \frac{G_c^2}{c_P},$$ then $P(\epsilon) = 0$ has two distinct positive real roots $\epsilon_1, \epsilon_2 > 0$, in which case \eqref{est6} is a contraction for $\epsilon \in (0,\epsilon_1) \cup (\epsilon_2,\infty)$.
\item
If $$64\xi^2\dfrac{(1-\kappa)^4}{\kappa^2} = \frac{G_c^2}{c_P},$$ then $P(\epsilon) = 0$ has one positive real root, $\epsilon_0 > 0$, of multiplicity two, in which case \eqref{est6} is a contraction for all $\epsilon \neq \epsilon_0, \epsilon > 0$.
\item
If $$64\xi^2\dfrac{(1-\kappa)^4}{\kappa^2} < \frac{G_c^2}{c_P},$$ then $P(\epsilon) = 0$ has two complex roots, in which case \eqref{est6} is a contraction for all $\epsilon > 0$.
\end{enumerate}
\end{proof}
\begin{remark}[Convergence rate]
According to the above proof, if the scheme is not converging for a given value of $\epsilon$, then a larger or a smaller value may be chosen to rectify the situation. However, since crack surfaces become singular as $\epsilon \rightarrow 0$ (thus necessitating finer meshing, i.e., $h \rightarrow 0$), we choose to state Theorem~\ref{convthm} with the condition that $\epsilon$ be large enough. We note also that due to some unknown constants in the convergence rate~\eqref{convrate}, it is not known whether this rate is optimal. Furthermore, working with a large $\epsilon$
is substantiated by the theory of phase field fracture being based on
$\Gamma$ convergence~\cite{AT90,AT92}. Applying this
to phase field fracture was first done in~\cite{BourFraMar00}.
Specifically, the setting is suitable when $h = o(\epsilon)$; namely when
$\epsilon$ is sufficiently large.
\end{remark}
\section{Algorithm}\label{sec:algorithms}
In practice, we apply the stabilizations and penalizations proposed in the previous sections
as outlined below. It is well-known (e.g., \cite{NoWri06}) that the choice of $\gamma$ is 
critical. If $\gamma$ is too low, crack irreversibility will not be enforced.
On the other hand, if $\gamma$ is too large, the linear equation system is 
ill-conditioned and influences the performance of the nonlinear solver. For this reason, $\gamma$ is 
updated in at each iteration step. Better, in terms of robustness, is the augmention in such 
an iteration by an additional $L^2$ function $\Xi$, yielding a so-called 
\emph{augmented Lagrangian iteration} going back to \cite{FoGlo83,GloTa89}. For phase-field 
fracture this idea was first applied in \cite{WheWiWo14}. 
Thus, combining the staggered iteration for the solid and phase-field
  systems with the update of the penalization parameter $\Xi$ yields the following algorithm:


\begin{Algo}
\begin{algorithmic}
\State At the loading step $t^n$. 
\State Choose initial $\Xi^0$. Set $\gamma>0$.
\Repeat
\State Iterate on $i$ (augmented Lagrangian loop)
\State Solve two-field problem, namely
\State\;\;\;\;   Solve elasticity in Problem \eqref{iter2}
\State\;\;\;\;   Solve the nonlinear phase-field in Problem \eqref{iter1}
\State Update 
\[
\Xi^{i+1} = [\Xi^{i} + \gamma (\phi^{n,i+1} - \phi^{n-1})]^+
\]
\Until 
\begin{equation}
\label{tol_1}
{
\max(\|a_u(u^{n,i},v_k) - (b^n,v_k)\|,\|a_\phi(\phi^{n,i},\psi_l)\|) \leq
\operatorname{TOL},
} 
\end{equation}
{
\[
\text{for } k=1,\ldots, \text{dim}(V_h), \quad l=1,\ldots, \text{dim}(W_h).
\]
}
\State Set: $(u^n,\phi^n):= (u^{n,i}, \phi^{n,i})$.
\State Increment $t^n \rightarrow t^{n+1}$.
\end{algorithmic}
\end{Algo}
For the stabilization parameters $L_u, L_\phi$, we have the following requirements (somewhat similar to $\gamma$): If the stabilization is too small, the stabilization effects vanish. If the stabilization is too large, we revert to an unacceptably slow convergence, and potentially, may converge to a solution corresponding to an undesirable local minimum of the original problem. In order to deal with these issues, we employ here a simple, yet effective strategy: We draw $L := L_u = L_\phi$ from a range of suitable values and compare the results, i.e., $L \in \{1.0e-6, 1.0e-3, 1.0e-2, 1.0e-1\}$. Moreover, we include also the configurations $L_u = 0$, $L_\phi > 0$ and $L_u = L_\phi = 0$ in all the numerical tests to be done in the following.
\begin{remark}
In this paper we use $\operatorname{TOL} = 10^{-6}$.
\end{remark}
\subsection{Nonlinear solution, linear subsolvers and programming code}
Both subproblems (phase field and mechanics) may be nonlinear. In our theory presented above,
we assumed a standard elasticity tensor. However, 
the model~\eqref{iter2}--\eqref{iter1} is too simple for most 
mechanical applications. 
More realistic 
phase-field fracture applications require a splitting of the stress tensor
(based on an energy split) in order to account 
for fracture development only under tension, but not under 
compressive forces. Consequently, we follow 
here \cite{MieWelHof10a} and split $\bsig$ into tensile 
$\bsig^+$ and compressive parts $\bsig^-$ :
\begin{align*}
\bsig^+ &:= 2\mu_s \be^+ + \lambda_s <\text{tr}(\be)> \bid,\\
\bsig^- &:= 2\mu_s (\be - \be^+) + \lambda_s \bigl(\text{tr}(\be) - <\text{tr}(\be)>\bigr) \bid,
\end{align*}
and 
\[
\be^+ = \bP \blambda^+ \bP^T,
\]
where the elasticity tensor $\bC$ has been replaced by the Lam\'e parameters, $\mu_s$ and $\lambda_s$. Moreover, $\bid$ is the $d\times d$ identity matrix, and $<\cdot >$ is the positive part of a function.
In particular, for $d=2$, we have
\[
\blambda^+ := \blambda^+ (u):=
\begin{pmatrix}
<\lambda_1(u) > & 0 \\
0 & <\lambda_2(u) >
\end{pmatrix}
,
\]
where $\lambda_1(u)$ and $\lambda_2(u)$ are 
the eigenvalues of the strain tensor $\be := \be(u)$,
and $v_1(u)$ and $v_2(u)$ the corresponding {(normalized)} 
eigenvectors.
Finally, the matrix $\bP$ is defined as $\bP:= \bP(u):= [v_1 | v_2]$;
namely, it consists of the column vectors $v_i, \ i=1,2$.
We notice that another frequently employed stress-splitting law was proposed in 
\cite{AmorMarigoMaurini2009}. 

The modified scheme reads:
\bse
\begin{align}
\nonumber &\textnormal{\textbullet \ \textbf{Step 1}: given $(u^{n,i-1}, \phi^{n,i-1}, b^n)$ find $u^{n,i}$ such that}\\
&\qquad L_u(u^{n,i} - u^{n,i-1},v) + \left( g(\phi^{n,i-1}) \bsig^+(u^{n,i}), \be(v) \right) 
+ \left( \bsig^-(u^{n,i}), \be(v) \right) =(b^n, v),  &\forall v \in V_h, \label{iter2_mod} \\
\nonumber&\textnormal{\textbullet \ \textbf{Step 2}: given $(\phi^{n,i-1}, u^{n,i}, \phi^{n-1})$ find $\phi^{n,i}$ such that} \\
\nonumber &\qquad L_\phi (\phi^{n,i} - \phi^{n,i-1},\psi) + G_c \epsilon(\nabla \phi^{n,i}, \nabla \psi) - \frac{G_c}{\epsilon} (1 - \phi^{n,i}, \psi) \\
&\qquad \qquad+ (1-\kappa)(\phi^{n,i} \bsig^+(u^{n,i}):\be(u^{n,i}),\psi) 
+ (\eta^i( \Xi + \gamma (\phi^{n,i} - \phi^{n-1})), \psi)= 0, &\forall \psi \in W_h. \label{iter1_mod} 
\end{align}
\ese
These modifications render the displacement system \eqref{iter2_mod} 
nonlinear, for which we use a Newton-type solver. The phase field equation is also nonlinear due to the 
penalization term and the stress splitting. 
Our version of Newton's method is based on a 
residual-based monotonicity criterion (e.g., \cite{Deuflhard2011})
outlined in \cite{wick2017modified}[Section 3.2].
Inside Newton's method, the linear subsystems are solved 
with a direct solver; namely UMFPACK \cite{DaDu97}. 
All numerical tests presented in Section \ref{sec_tests}
are implemented in the open-source finite element
library deal.II \cite{dealII85,BangerthHartmannKanschat2007}.
Specifically, the code is based on a simple adaptation of 
the multiphysics template \cite{Wi11_fsi_with_deal} in which specifically the previously 
mentioned Newton solver is implemented.

\section{Numerical experiments}
\label{sec_tests}


In this section, we present several numerical tests to 
substantiate our algorithmic developments. The goals of all three numerical examples are comparisons between an
  unlimited number of staggered iterations (although bounded by $500$) denoted by `$L$', and a low, fixed number, denoted by `$LFI$', where we use $30$ (Ex. 1 and Ex. 2), and $20$ (Ex. 3)
  staggered iterations, respectively. These comparisons are performed in
  terms of the number of iterations and the correctness of the solutions in
  terms of the so-called~\emph{load-displacement curve}, measuring the stresses of the
top boundary versus the number of loading steps.


\subsection{Single edge notched tension test}
This test was applied for instance in \cite{MieWelHof10a}.
The configuration is displayed in Figure \ref{ex_1_config_and_mesh}.
We use the system \eqref{iter2}-\eqref{iter1}.
Specifically, we study our proposed iterative schemes on different 
mesh levels, denoted as refinement (Ref.) levels 4, 5, 6 (uniformly refined), with $1024$ elements ($2210$ Dofs for the
displacements, $1\,105$ Dofs for the phase-field, $h = 0.044$), 
$4\,096$ elements ($8\,514$ Dofs for the
displacements, $4257$ Dofs for the phase-field, $h = 0.022$),
and $16\,384$ elements ($33\,410$ Dofs for the
displacements, $16\,705$ Dofs for the phase-field, $h = 0.011$).

\begin{figure}[H]
\centering
{\includegraphics[width=5cm]{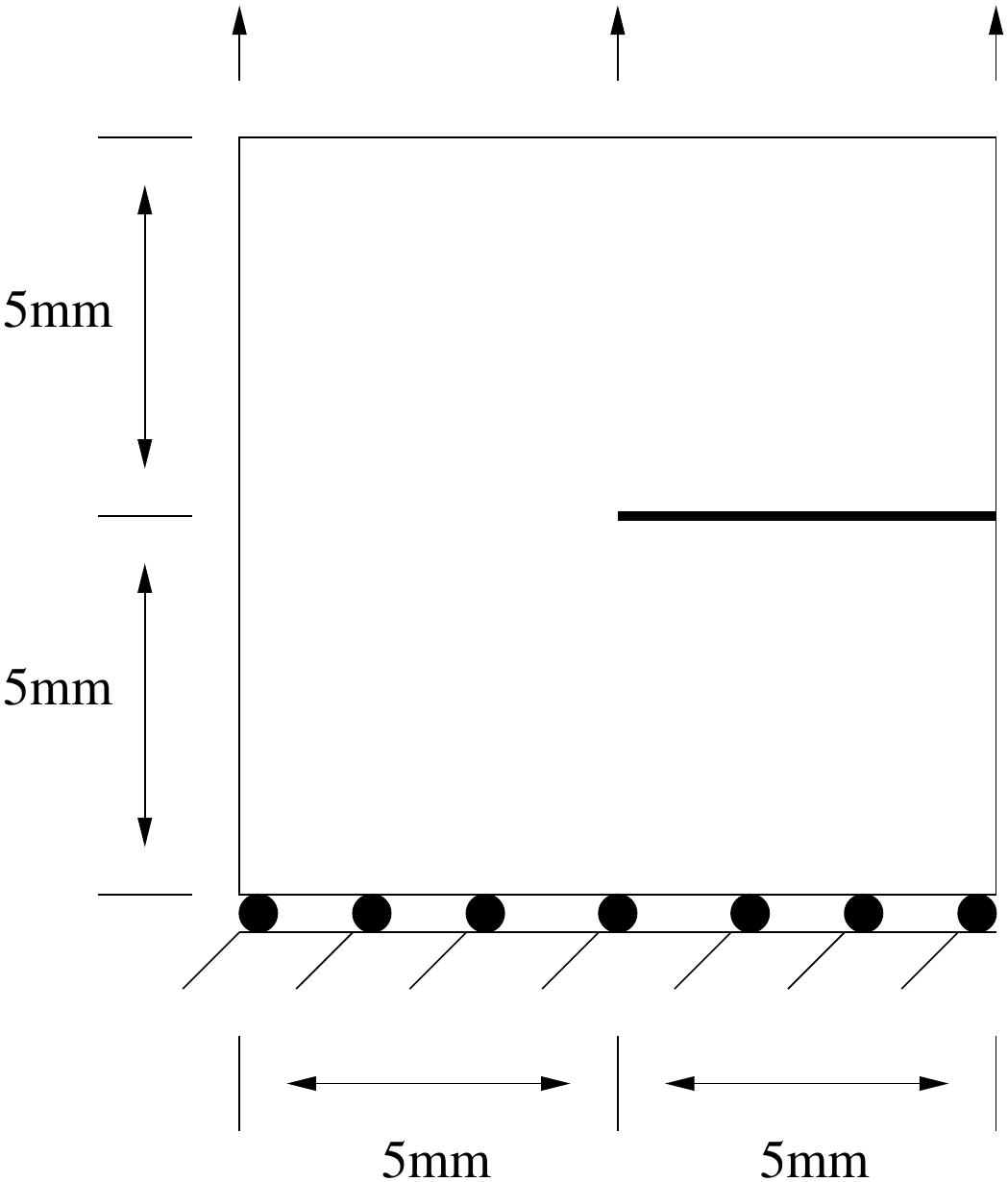}}\hspace*{1cm}
{\includegraphics[width=5cm]{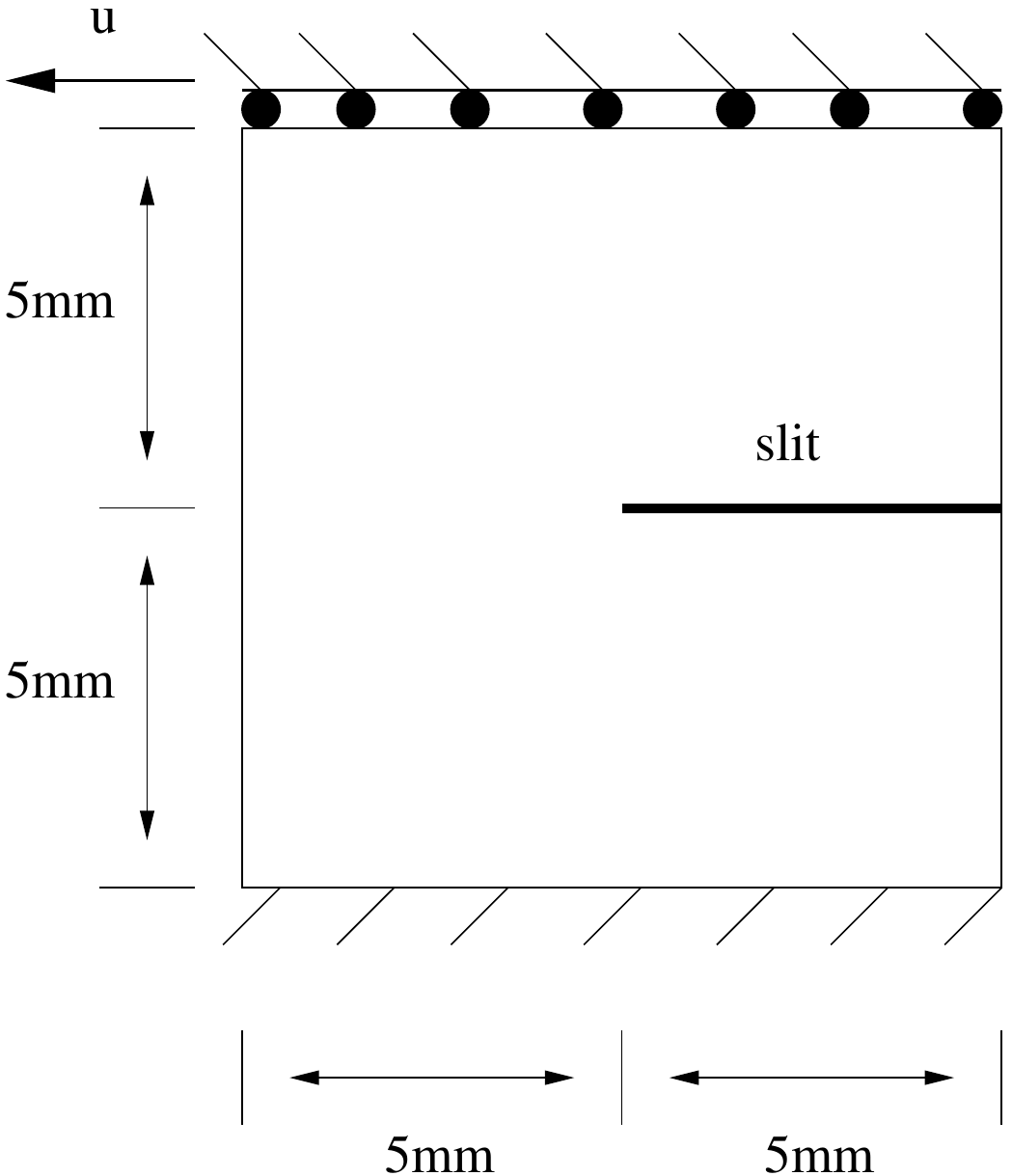}}\hspace*{1cm}
{\includegraphics[width=6cm]{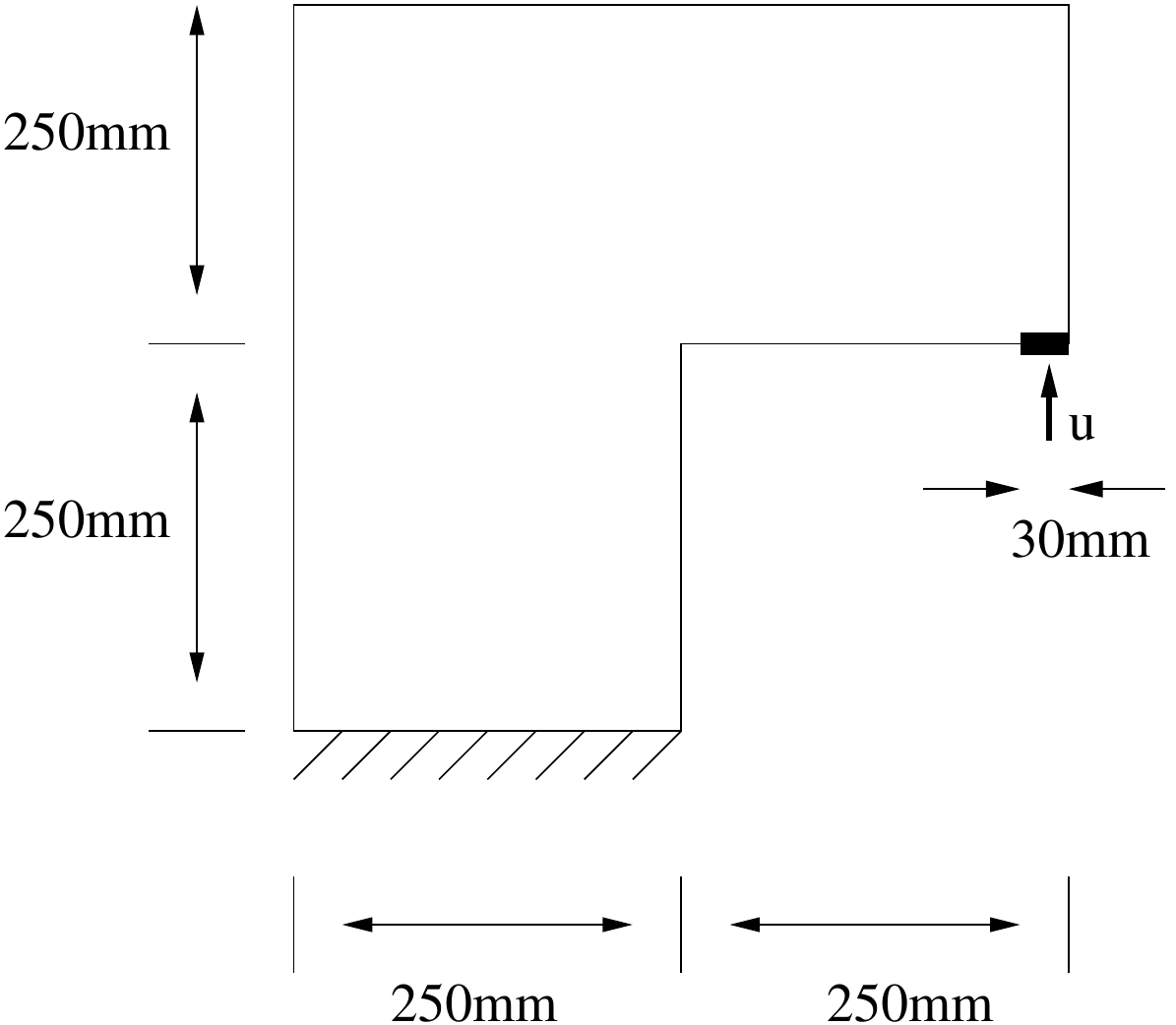}}
\caption{Examples 1,2,3: Configurations.
Left: single edge notched tension test. 
In detail, the boundary conditions are: $u_y$ = \SI{0}{mm} (homogeneous
Dirichlet)
and traction free (homogeneous Neumann conditions) 
in $x$-direction on the bottom. On the top boundary $\Gamma_\textnormal{top}$,
we prescribe $u_x$ = \SI{0}{mm} and $u_y$ as provided in \eqref{diri_ex_2}. All other 
boundaries including the slit are traction free (homogeneous Neumann
conditions). Single edge notched shear test (middle) and L-shaped panel test (right).
We prescribe the following conditions: On the left and right
boundaries, $u_y$ = \SI{0}{mm} and traction-free in $x$-direction. On the bottom 
part, we use $u_x = u_y$ = \SI{0}{mm} and on $\Gamma_\textnormal{top}$, we prescribe $u_y$ = \SI{0}{mm} 
and $u_x$ as stated in \eqref{diri_ex_2}. Finally, the lower part of the 
slit is fixed in $y$-direction, i.e., $u_y$ = \SI{0}{mm}. 
For the L-shaped panel test (at right), 
the lower left boundary is fixed: $u_x = u_y =$ \SI{0}{mm}.
A displacement condition for $u_y$ is prescribed by 
\eqref{uy_test_2} in the right corner
on a section $\Gamma_u$ that has \SI{30}{mm} length.
}
\label{ex_1_config_and_mesh}
\end{figure}

Specifically, we use
$\mu_s$ = \SI{80.77}{kN/mm^2}, $\lambda_s$ = \SI{121.15}{kN/mm^2}, and
$G_c$ = \SI{2.7}{N/mm}.
The crack growth is driven  by a non-homogeneous Dirichlet
condition for the displacement field on $\Gamma_\textnormal{top}$, the top boundary of $B$ .
We increase the displacement on $\Gamma_\textnormal{top}$ over time, namely we apply
non-homogeneous Dirichlet conditions:
\begin{align}
u_y &= t \bar{u}, \quad \bar{u} = \SI{1}{mm/s}, \label{diri_ex_1}
\end{align}
where $t$ denotes the current loading time. Furthermore, we set $\kappa = 10^{-10}$ [mm] and $\eps = 2h$ [mm]. We evaluate the surface load vector on the $\Gamma_\textnormal{top}$ as
\begin{equation}
\label{eq_Fx_Fy}
\tau = (F_x,F_y) := \int_{\Gamma_\textnormal{top}} \bsig(u)\nu\, \textnormal{d}s,
\end{equation}
with normal vector $\nu$,
and we are particularly interested in $F_y$ for Example 1 
and $F_x$ for Example 2 (Section \ref{sec_ex_2}). Graphical solutions are displayed in the Figures \ref{miehe_tension_a}
and \ref{miehe_tension_b} showing the phase-field variable and the
discontinuous displacement field. 
Our findings of using different stabilization parameters $L$ are compared  
in the Figures \ref{miehe_tension_c}, \ref{miehe_tension_c1}, \ref{miehe_tension_d},
\ref{miehe_tension_e}, and \ref{miehe_tension_f}. 
Different mesh refinement studies are shown 
in the Figures \ref{miehe_tension_e} and \ref{miehe_tension_f}. Here, the
number of staggered iterations does not increase with finer mesh levels, which
shows the robustness of our proposed methodology.
\begin{figure}[H]
\centering
{\includegraphics[width=8cm]{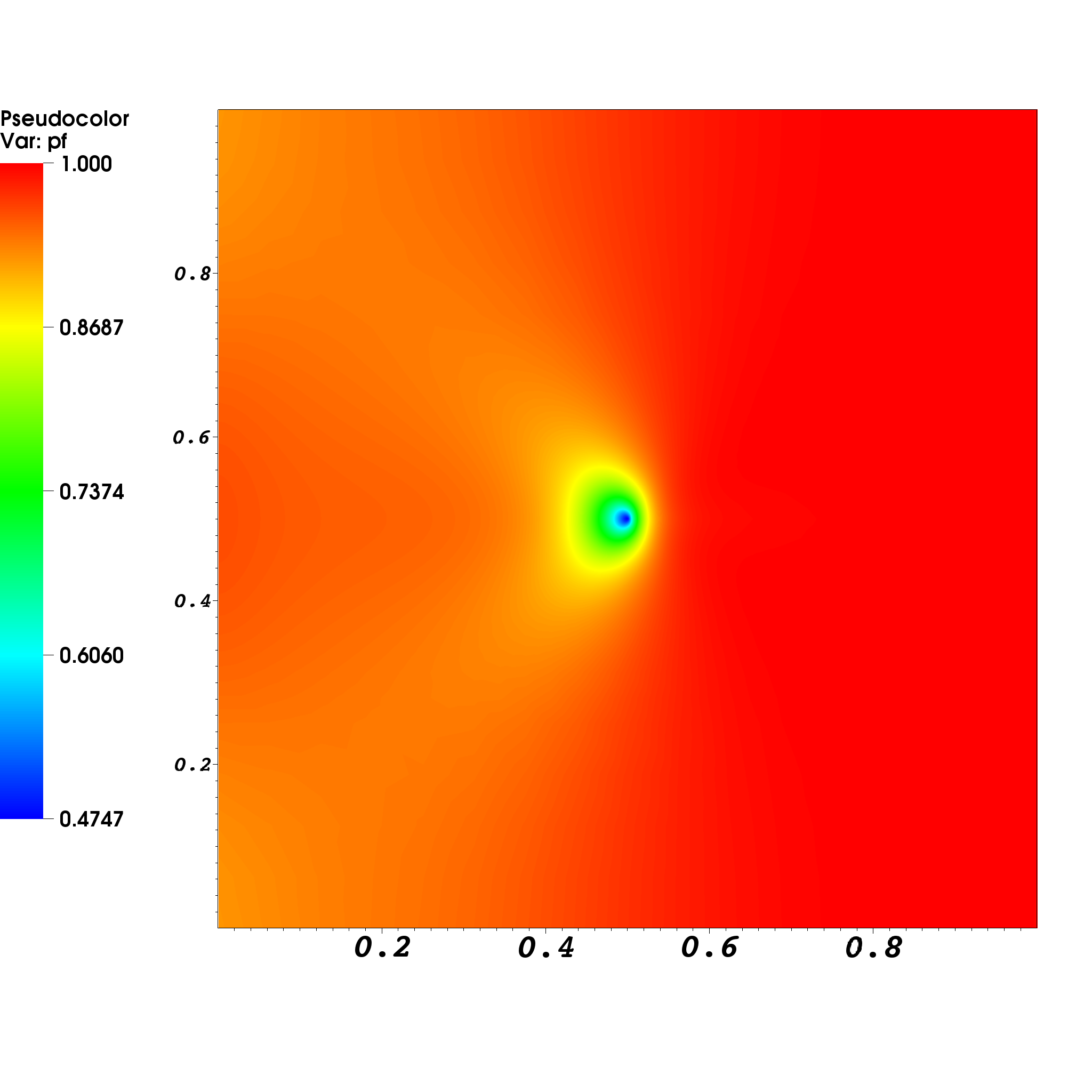}} 
{\includegraphics[width=8cm]{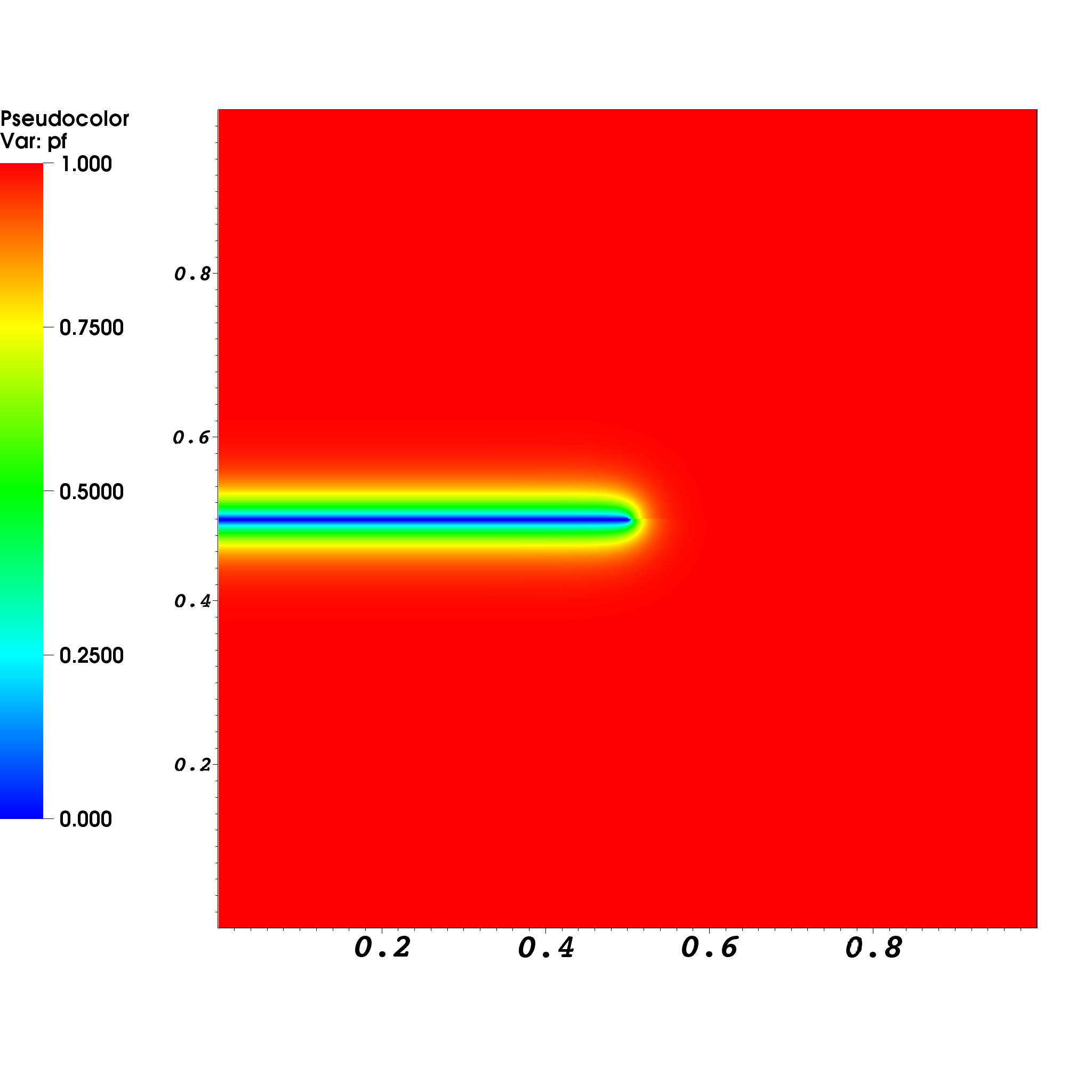}} 
\caption{Example 1: Single edge notched tension test: crack path at loading
  step $59$ (left) and $60$ (right). We see brutal crack growth in which 
the domain is cracked within one loading step.}
\label{miehe_tension_a}
\end{figure}
\begin{figure}[H]
\centering
{\includegraphics[width=7cm]{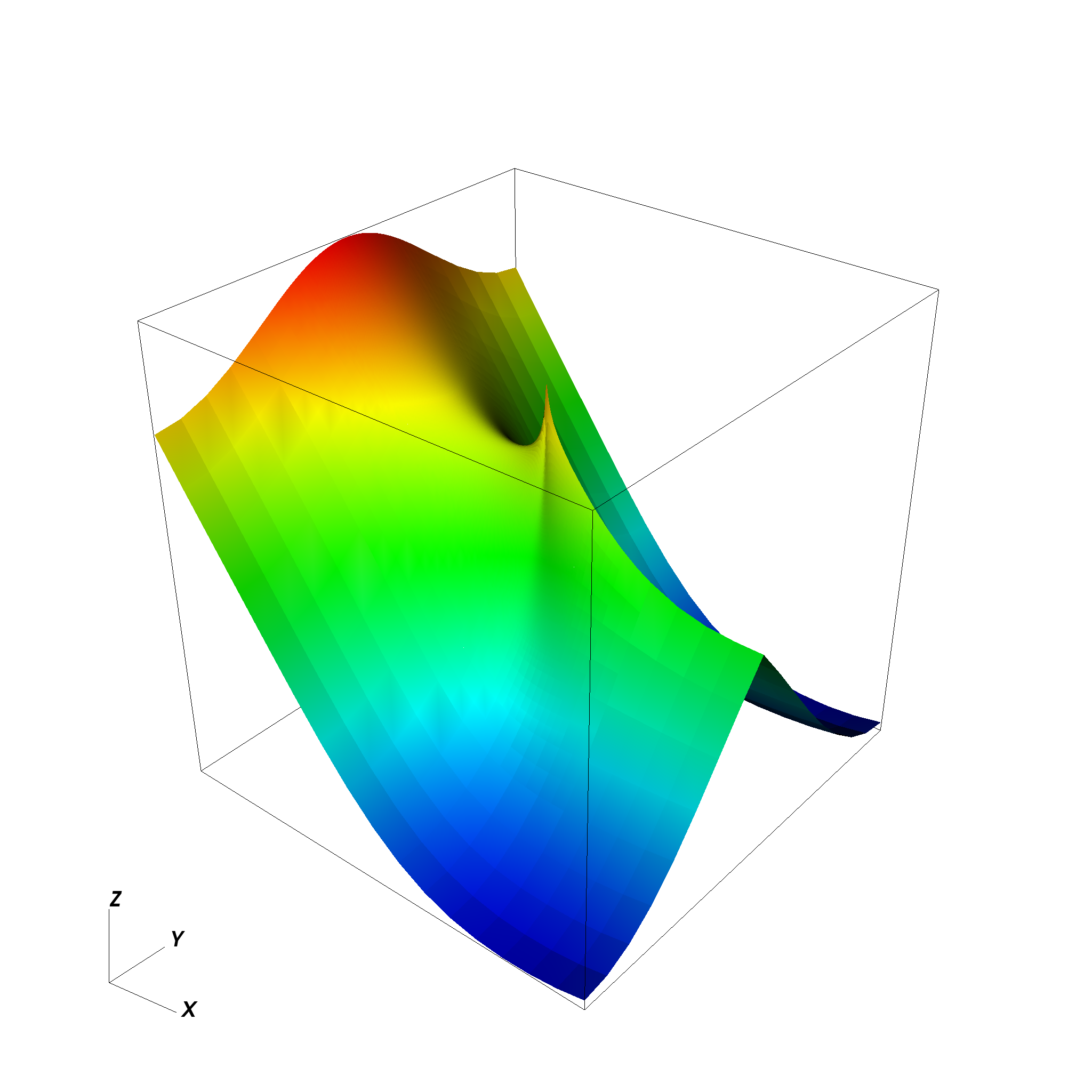}} 
{\includegraphics[width=7cm]{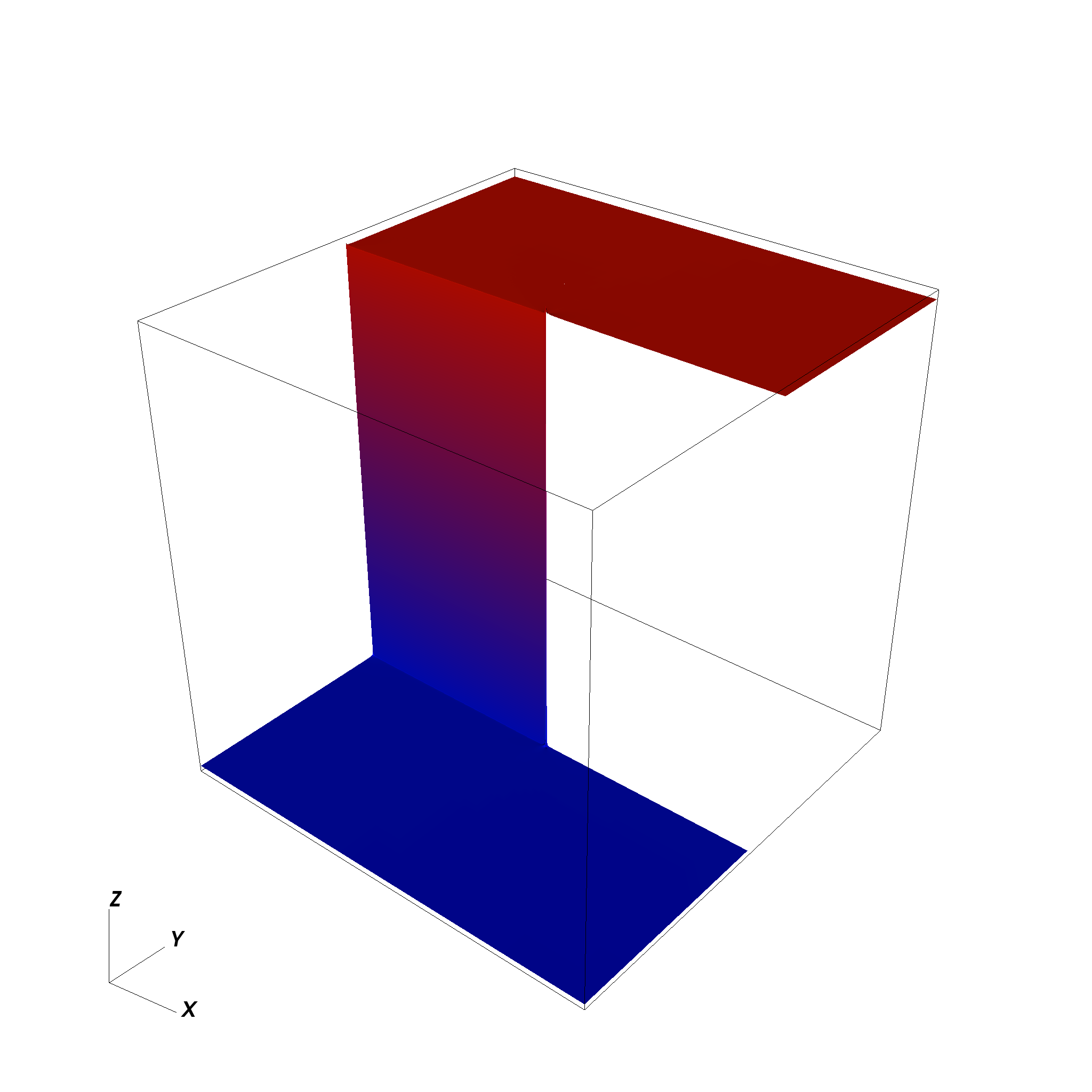}} 
\caption{Example 1: Single edge notched tension test: 3D plot of the
  displacement variable $u_x$ at the loading steps $59$ and $60$. At right, 
  the domain is totally fractured. In particular, we see the initial 
  crack build in the geometry in the right part where the domain has a true
  discontinuity.
  In the left part, the domain is cracked using the phase-field
  variable. Here, the displacement variable is still continuous since we 
  are using $C^0$ finite elements for the spatial discretization.}
\label{miehe_tension_b}
\end{figure}
\begin{figure}[H]
\centering
{\includegraphics[width=8cm]{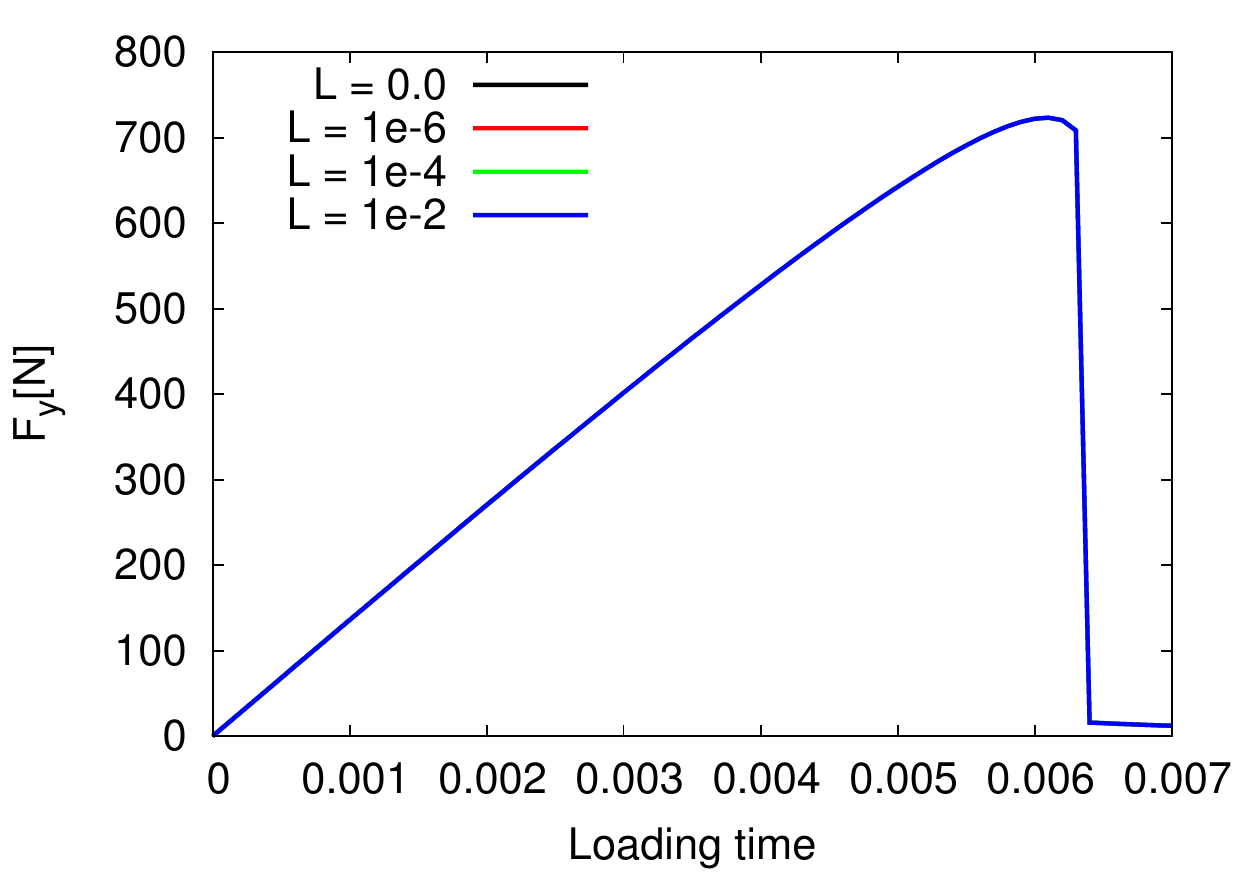}} 
{\includegraphics[width=8cm]{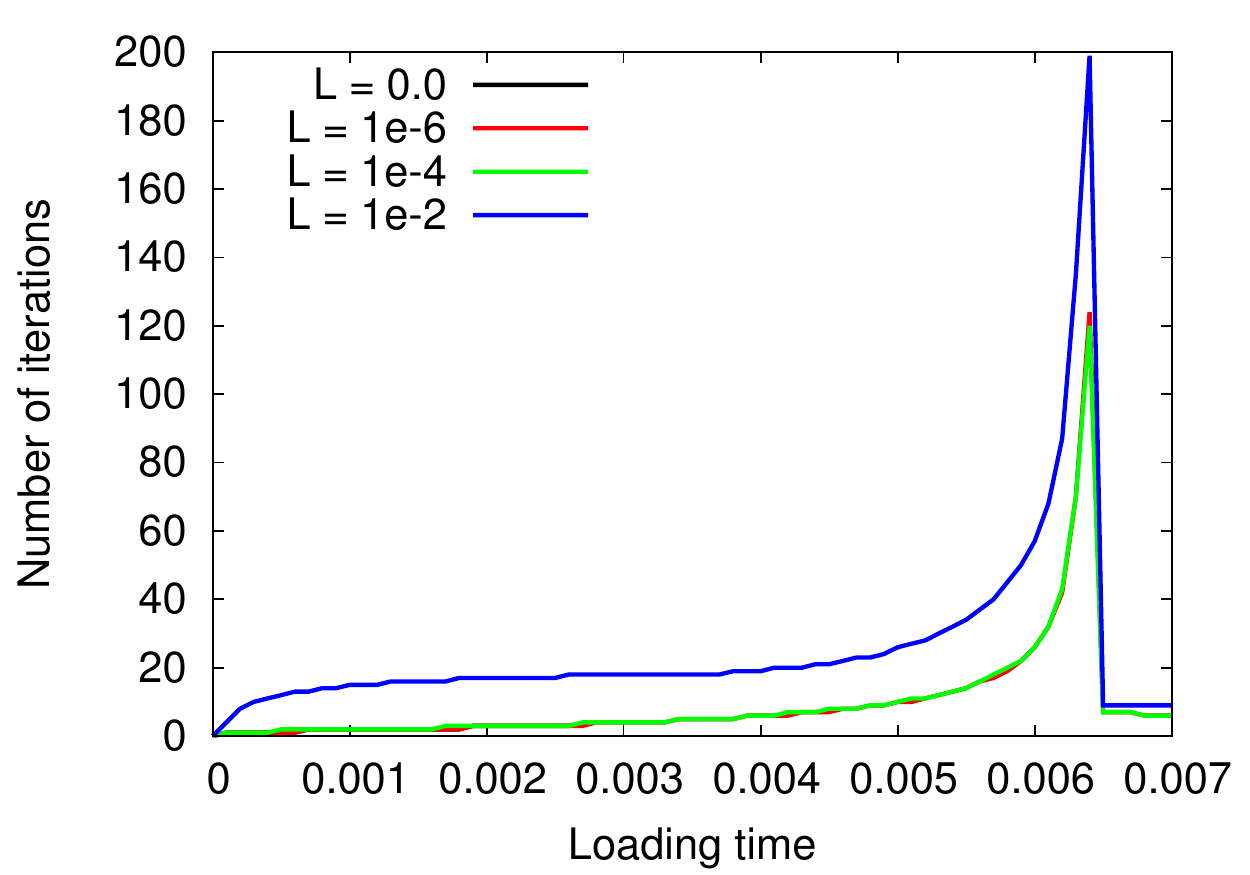}} 
\caption{Example 1: Comparison of different $L$. At left, the stresses are
  shown. At right, the number of staggered iterations is displayed.}
\label{miehe_tension_c}
\end{figure}
\begin{figure}[H]
\centering
{\includegraphics[width=8cm]{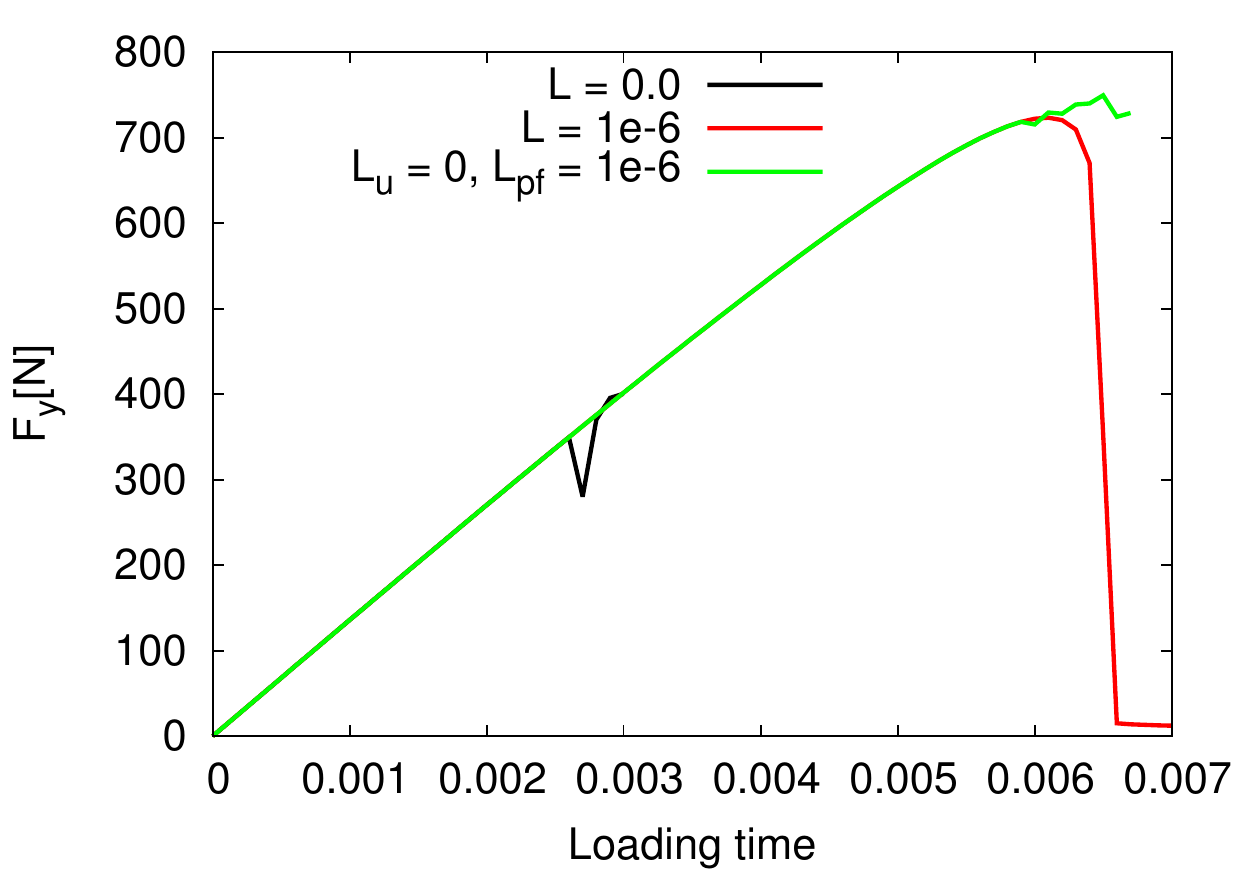}} 
{\includegraphics[width=8cm]{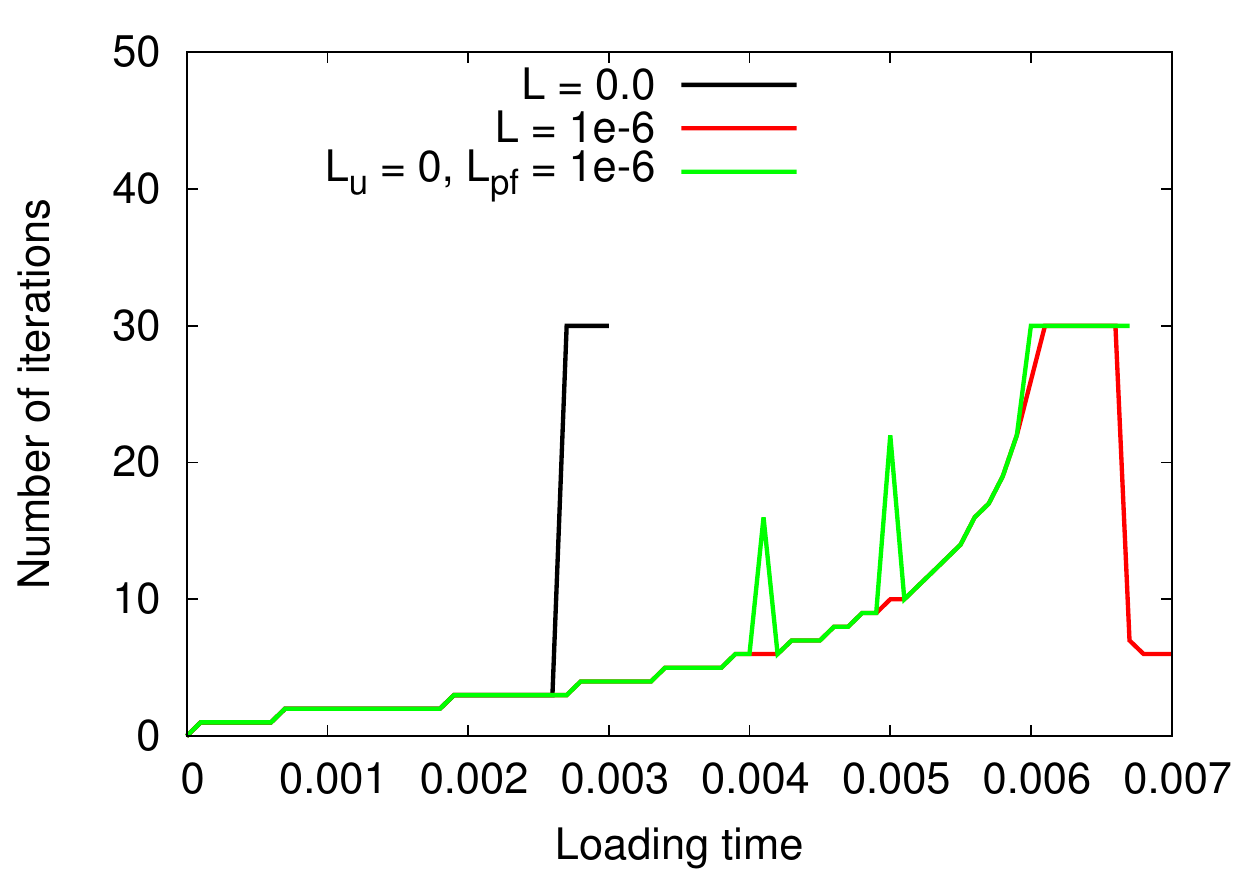}} 
\caption{Example 1: Comparison of different $L$. In this example, possibly due to
brutal crack growth, stabilizing only phase field subproblem does not work.
}
\label{miehe_tension_c1}
\end{figure}
\begin{figure}[H]
\centering
{\includegraphics[width=8cm]{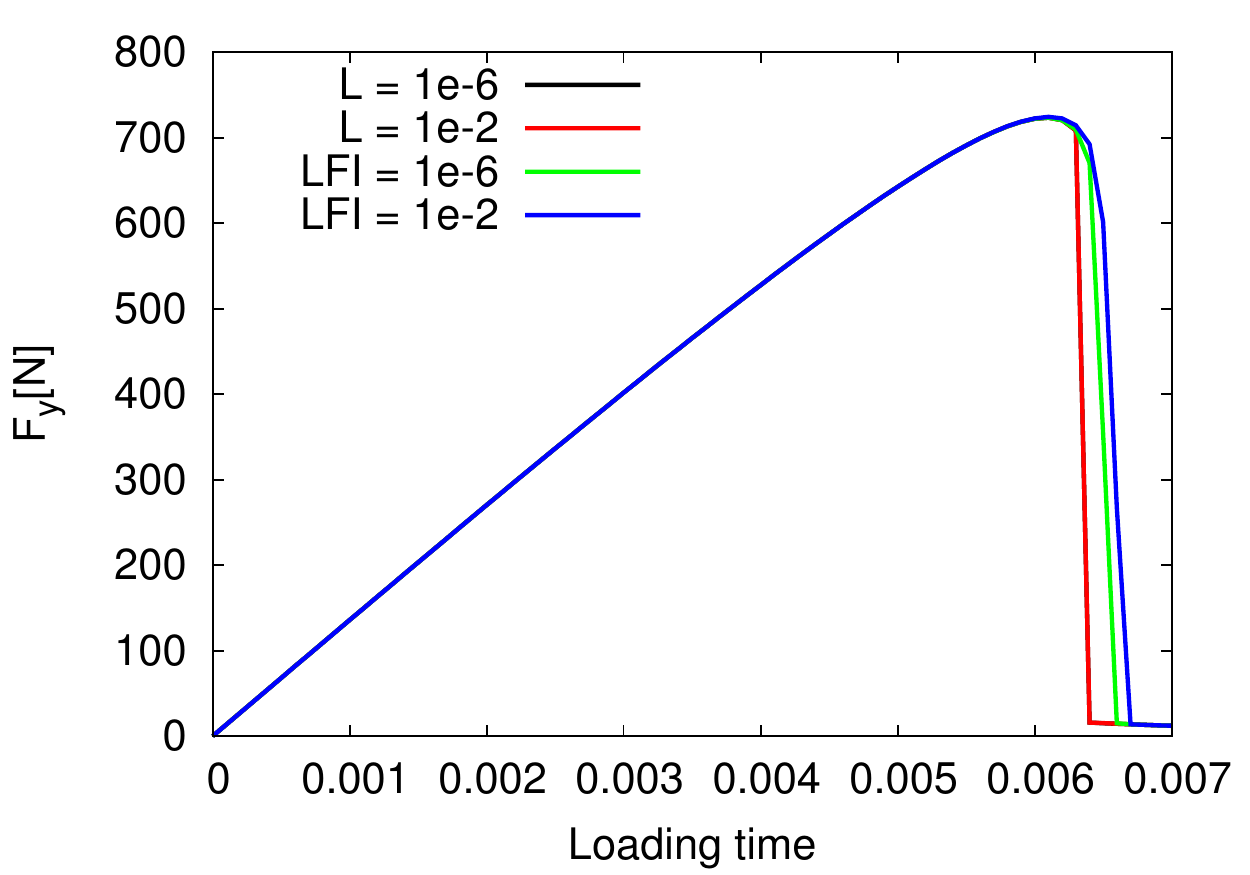}} 
{\includegraphics[width=8cm]{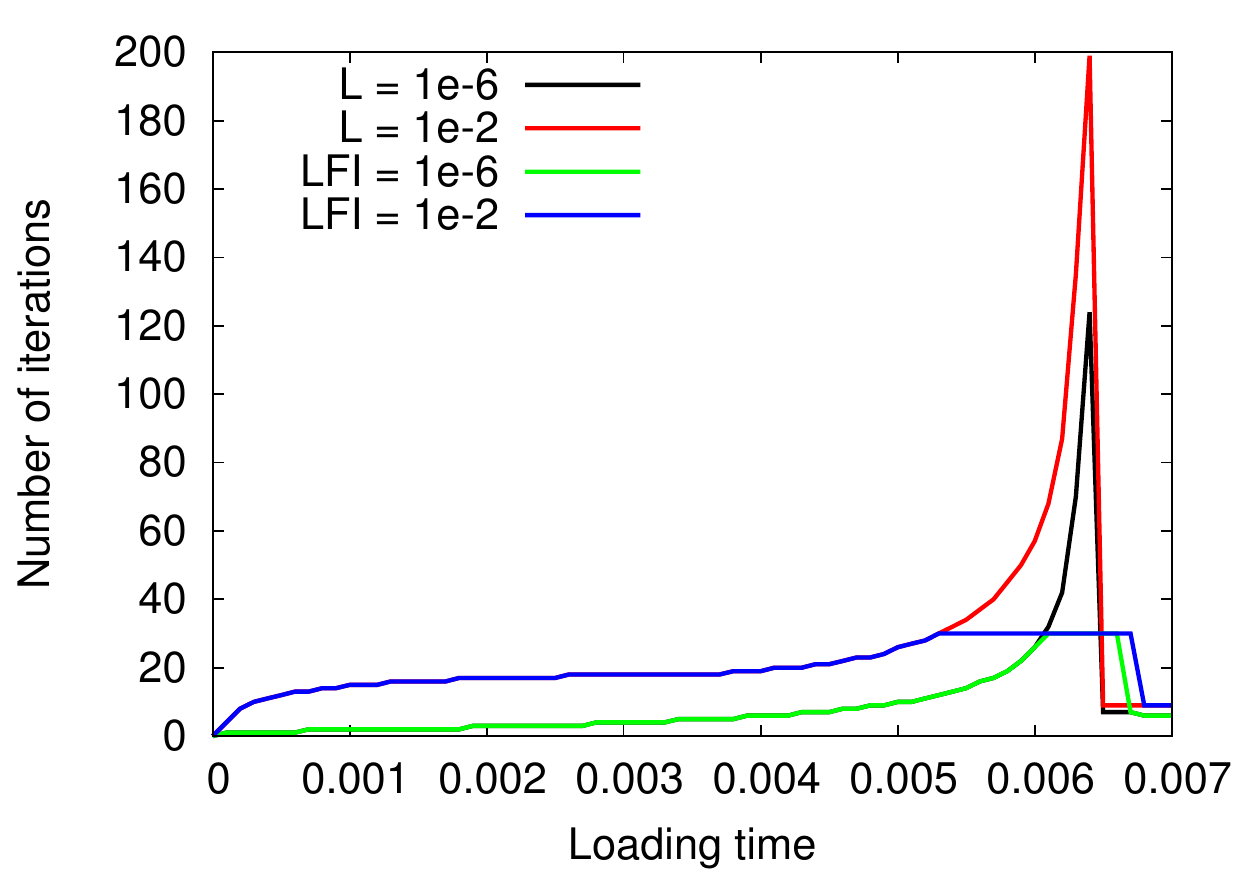}} 
\caption{Example 1: Comparison of different $L$ for an open number of
  iterations and a fixed number of iterations (LFI) with a maximum of $30$ iterations. At left, the stresses are
  shown. At right, the number of staggered iterations is displayed.}
\label{miehe_tension_d}
\end{figure}
\begin{figure}[H]
\centering
{\includegraphics[width=8cm]{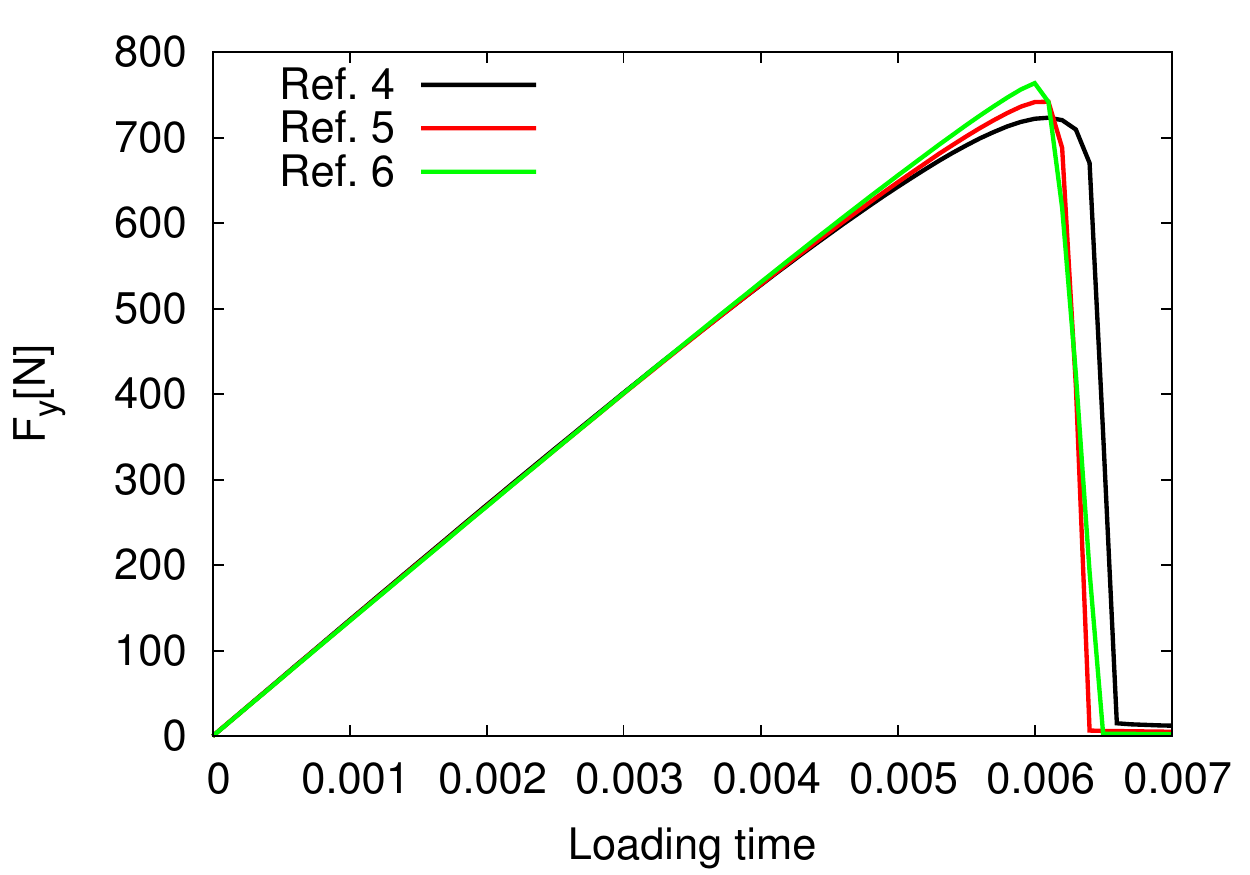}} 
{\includegraphics[width=8cm]{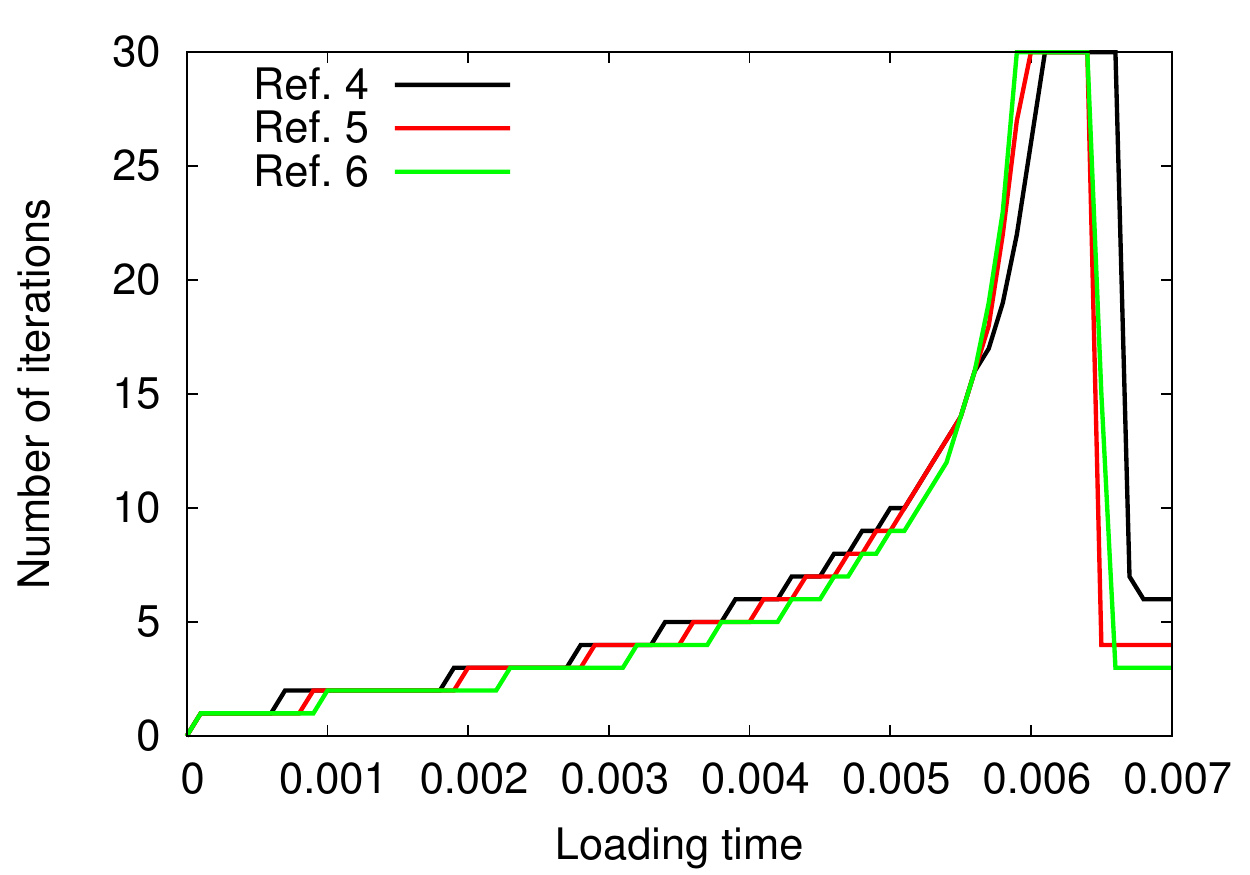}} 
\caption{Example 1: Using $L=1e-6$ comparing different mesh refinement levels $4,5,6$.
At left, the stresses are
  shown. At right, the number of staggered iterations is displayed.}
\label{miehe_tension_e}
\end{figure}
\begin{figure}[H]
\centering
{\includegraphics[width=8cm]{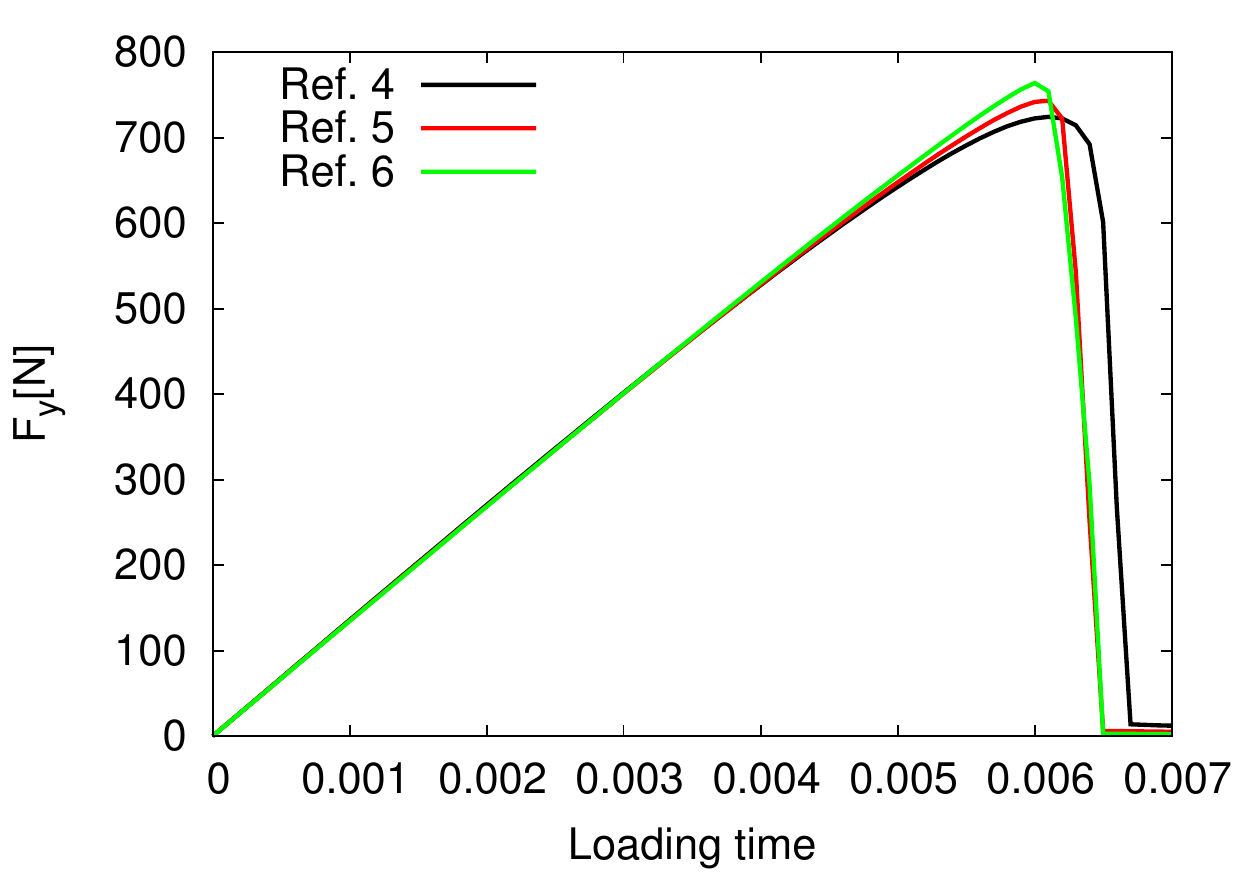}} 
{\includegraphics[width=8cm]{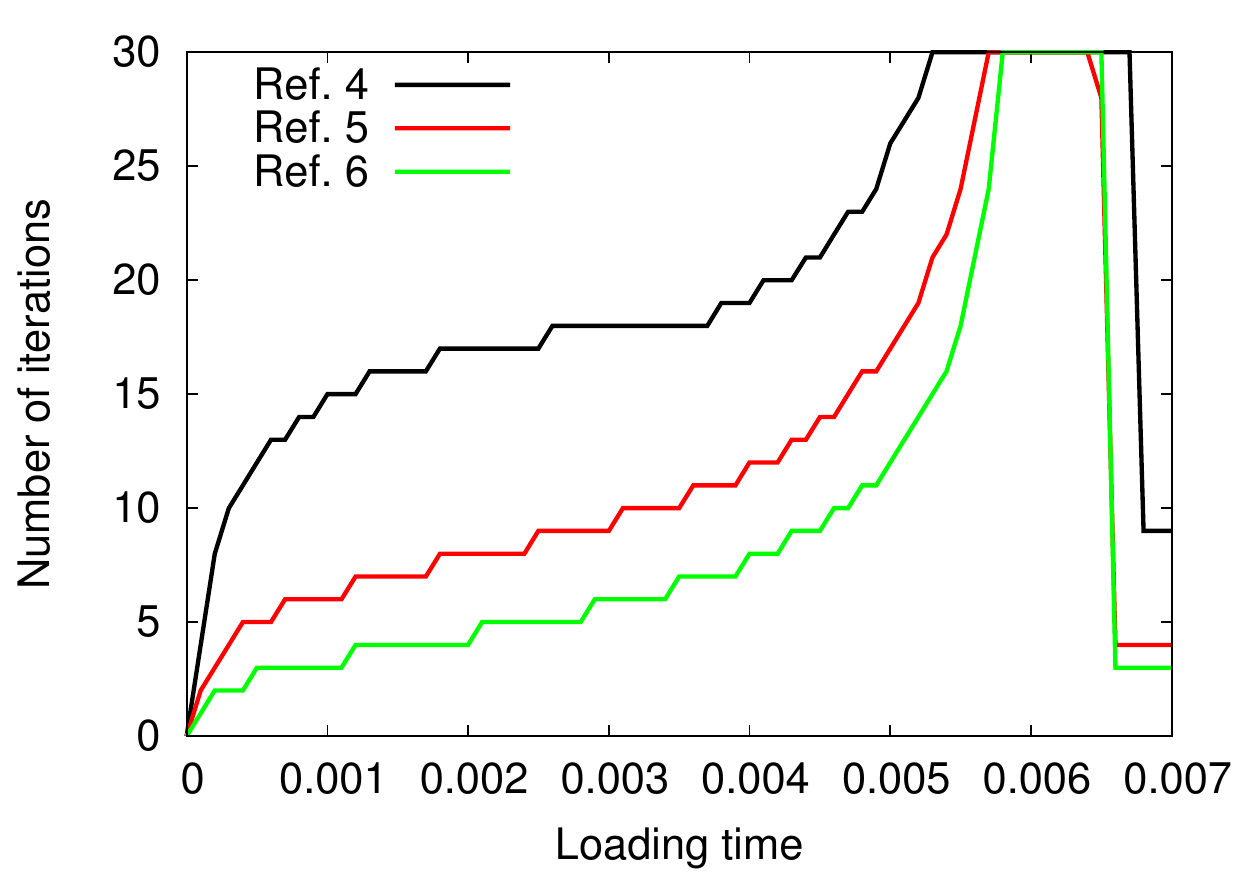}} 
\caption{Example 1: Using $L=1e-2$ comparing different mesh refinement levels $4,5,6$.
At left, the stresses are
  shown. At right, the number of staggered iterations is displayed.}
\label{miehe_tension_f}
\end{figure}

\newpage
\subsection{Single edge notched shear test}
\label{sec_ex_2}
The configuration of this second setting is 
very similar to Example 1 and was first proposed 
in a phase-field context in \cite{MieWelHof10a}.
We now use the model with strain-energy split 
\eqref{iter2_mod}--\eqref{iter1_mod}.
The parameters and the geometry (see Figure \ref{ex_1_config_and_mesh}) are the same as in the 
previous test case. The boundary condition is changed from tensile forces to 
a shear condition (see also again Figure \ref{ex_1_config_and_mesh}):
\begin{align}
u_x &= t \bar{u}, \quad \bar{u} = \SI{1}{mm/s}, \label{diri_ex_2}
\end{align}
As quanitity of interest we evaluate $F_x$ in \eqref{eq_Fx_Fy}. Our findings are shown in the Figures \ref{miehe_shear_a},
  \ref{miehe_shear_b},
\ref{miehe_shear_b1}, \ref{miehe_shear_c}, \ref{miehe_shear_d},
\ref{miehe_shear_e},
and \ref{miehe_shear_f}. The major difference to Example 1 is that the scheme is converging even with $L_u = 0$, as computationally justified in
Figure \ref{miehe_shear_b1}. As in Example 1, the load-displacement curves 
are close to the published literature and, again, the proposed $L$ scheme is
robust under mesh refinement (see Figures \ref{miehe_shear_c} - \ref{miehe_shear_f}).
\begin{figure}[H]
\centering
{\includegraphics[width=8cm]{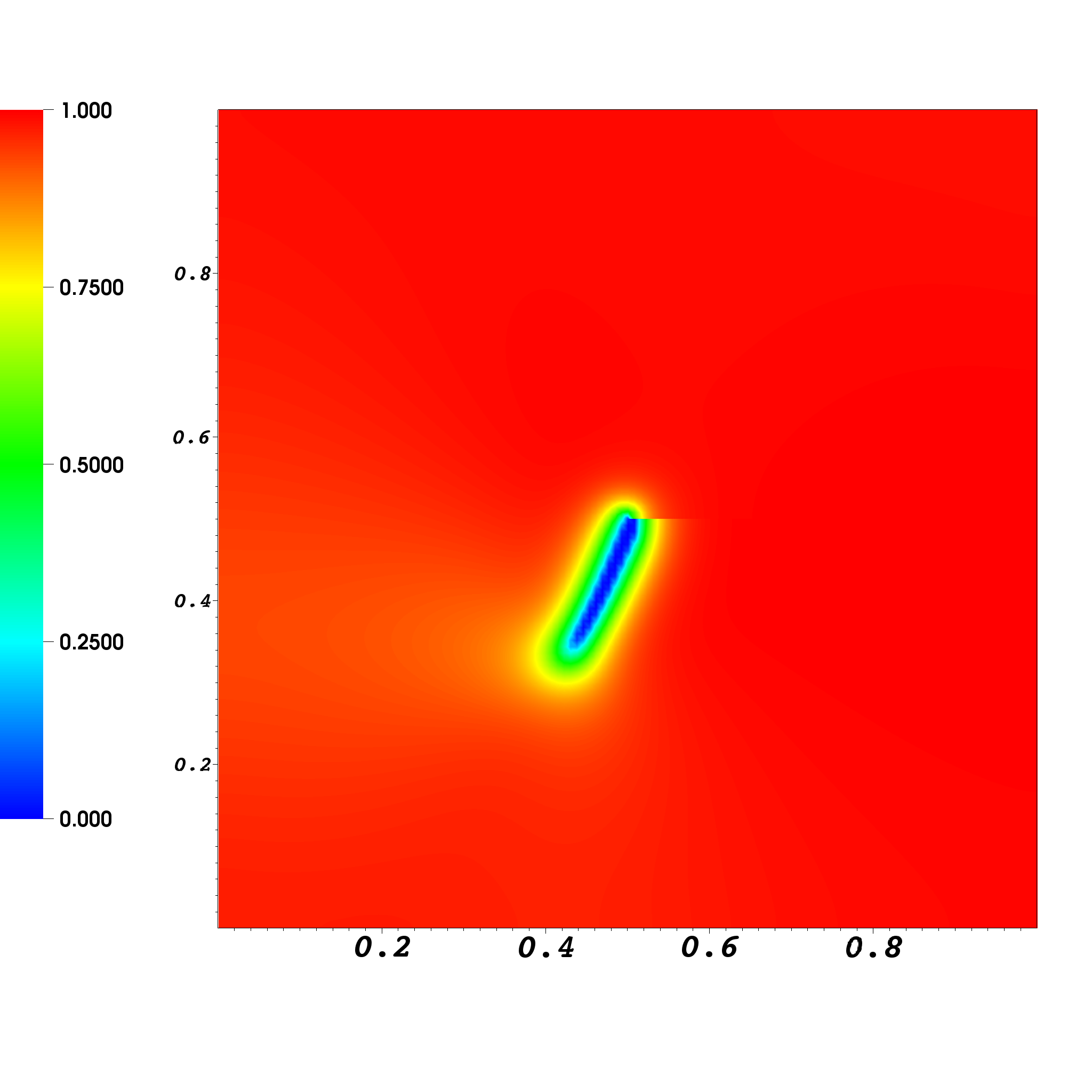}} 
{\includegraphics[width=8cm]{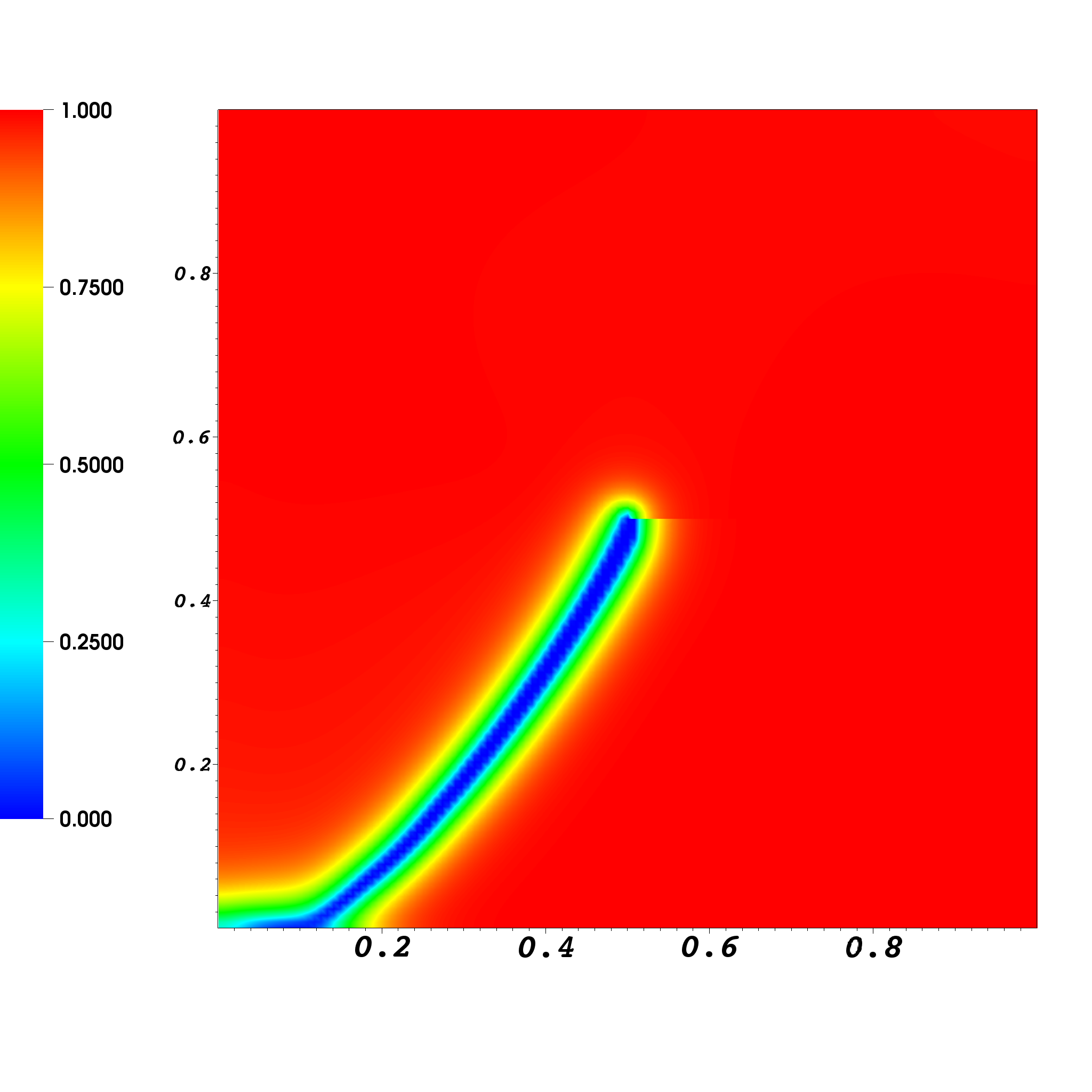}} 
\caption{Example 2: Single edge notched shear test: Crack path at loading
  step $110$ (left) and $135$ (right).}
\label{miehe_shear_a}
\end{figure}
\begin{figure}[H]
\centering
{\includegraphics[width=8cm]{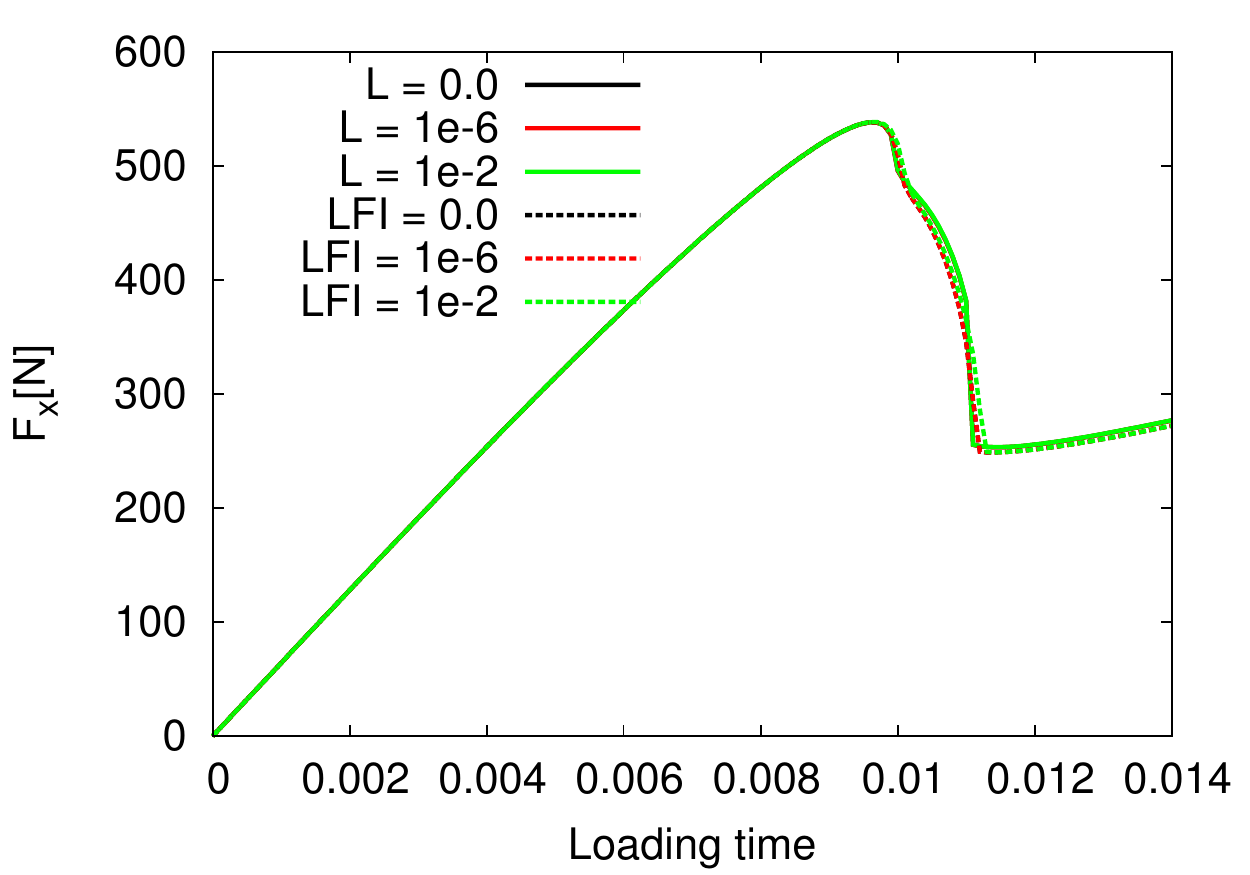}} 
{\includegraphics[width=8cm]{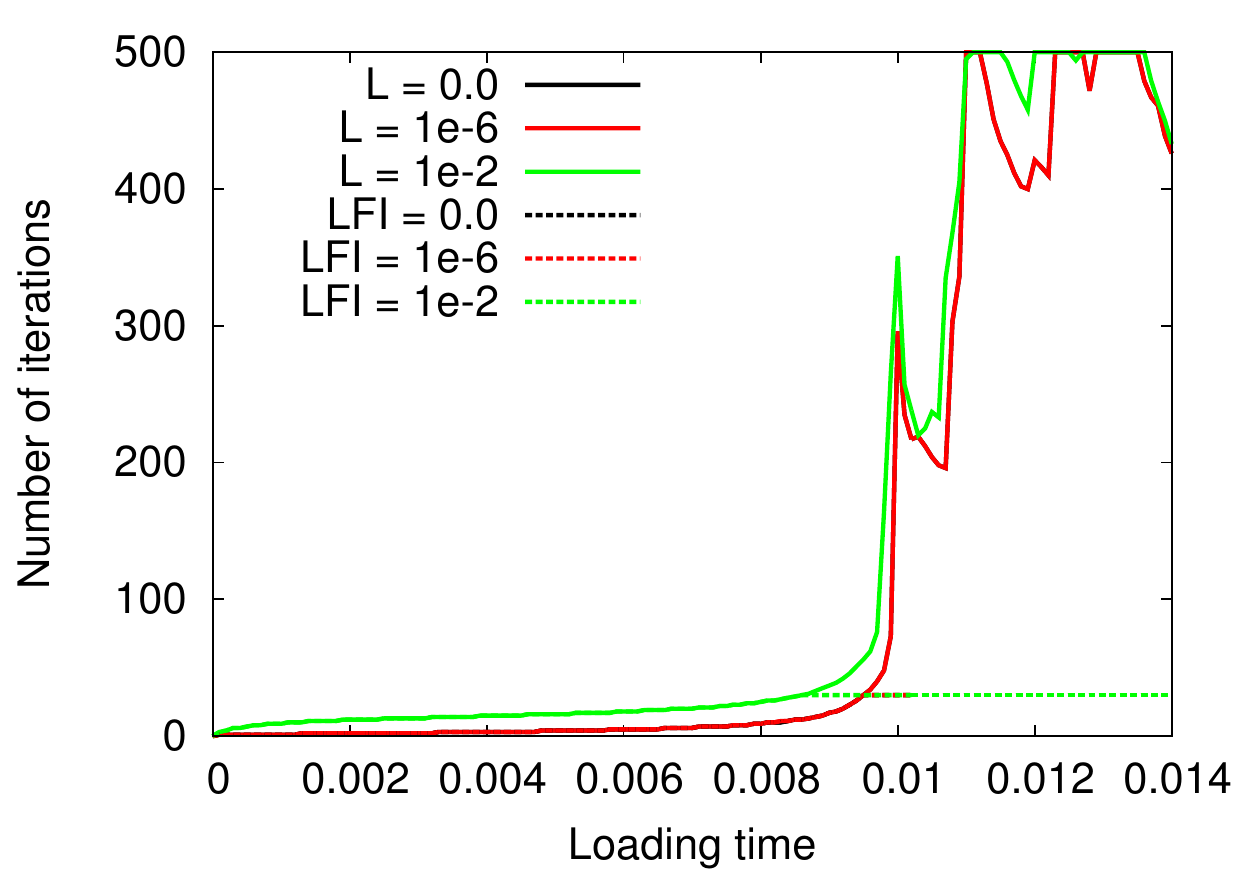}} 
\caption{Example 2: Comparison of different $L$ with an open number of
  staggered iterations (fixed by $500$) and a fixed number ($LFI$) with $30$
  iterations per loading step. At left, the
load-displacement curves displaying the evolution of $F_x$
versus the loading time. At right, the number of iterations is displayed.}
\label{miehe_shear_b}
\end{figure}
\begin{figure}[H]
\centering
{\includegraphics[width=8cm]{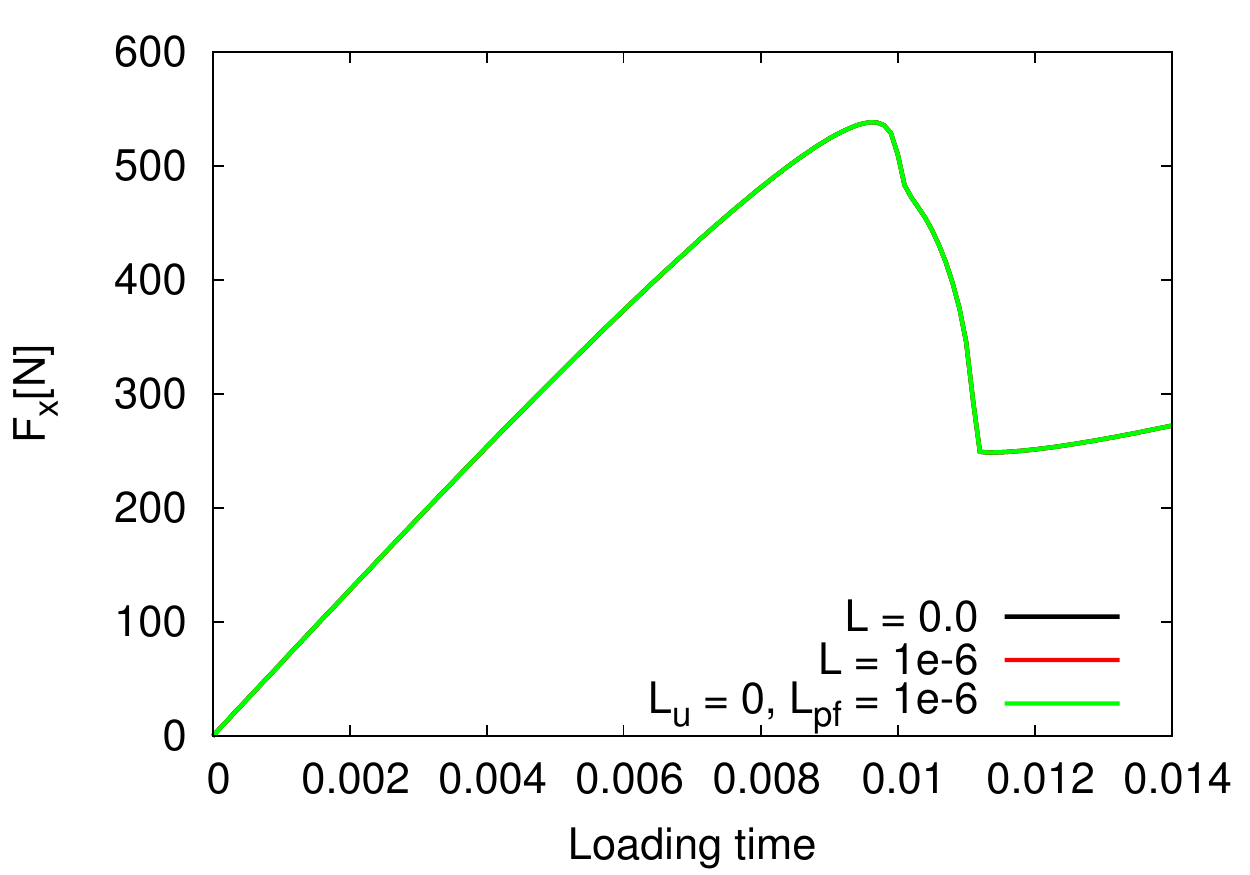}} 
{\includegraphics[width=8cm]{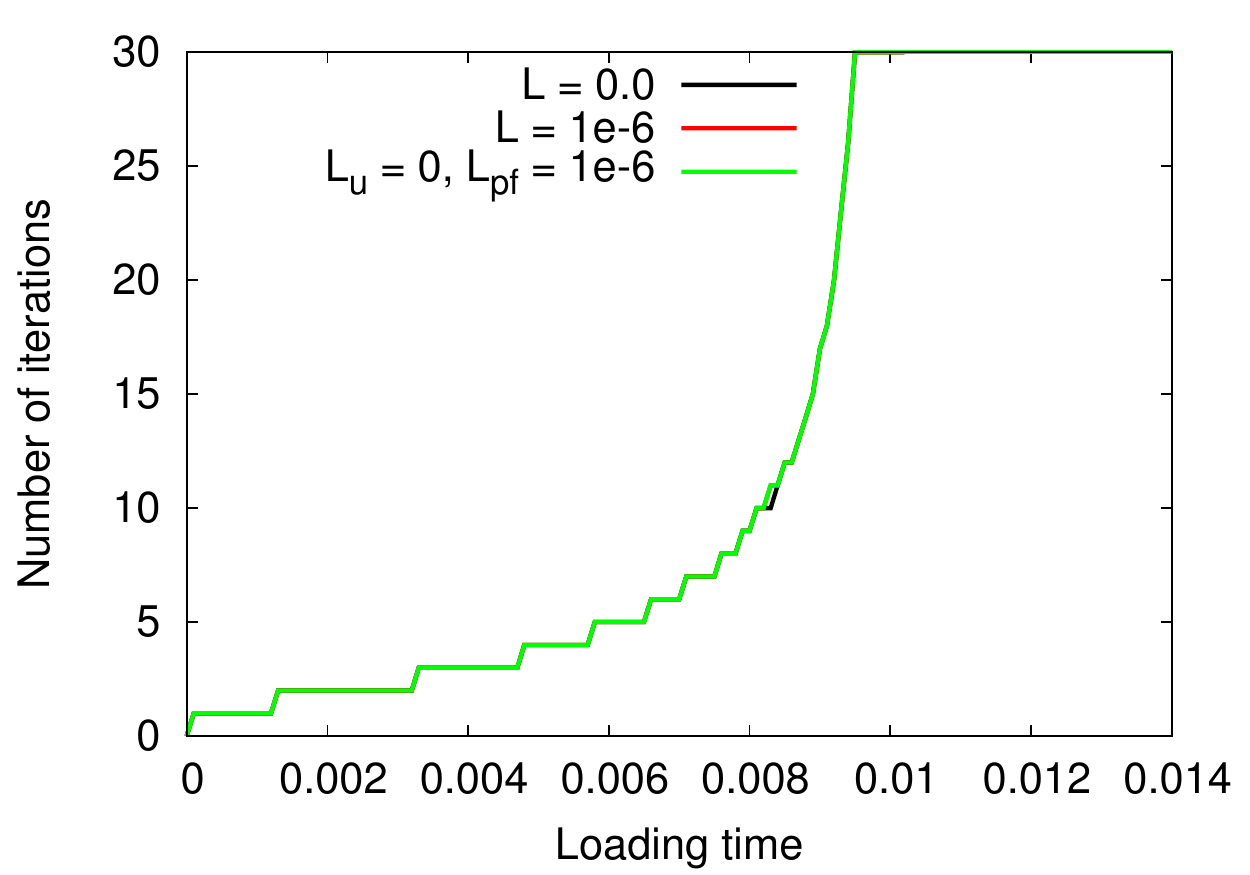}} 
\caption{Example 2: Comparison of different $L$. Observe that 
stabilizing the mechanics subproblem in this example has no or little effect. At left, the
load-displacement curves displaying the evolution of $F_x$
versus $u_y$ are shown. At right, the number of staggered iterations is displayed.}
\label{miehe_shear_b1}
\end{figure}
\begin{figure}[H]
\centering
{\includegraphics[width=8cm]{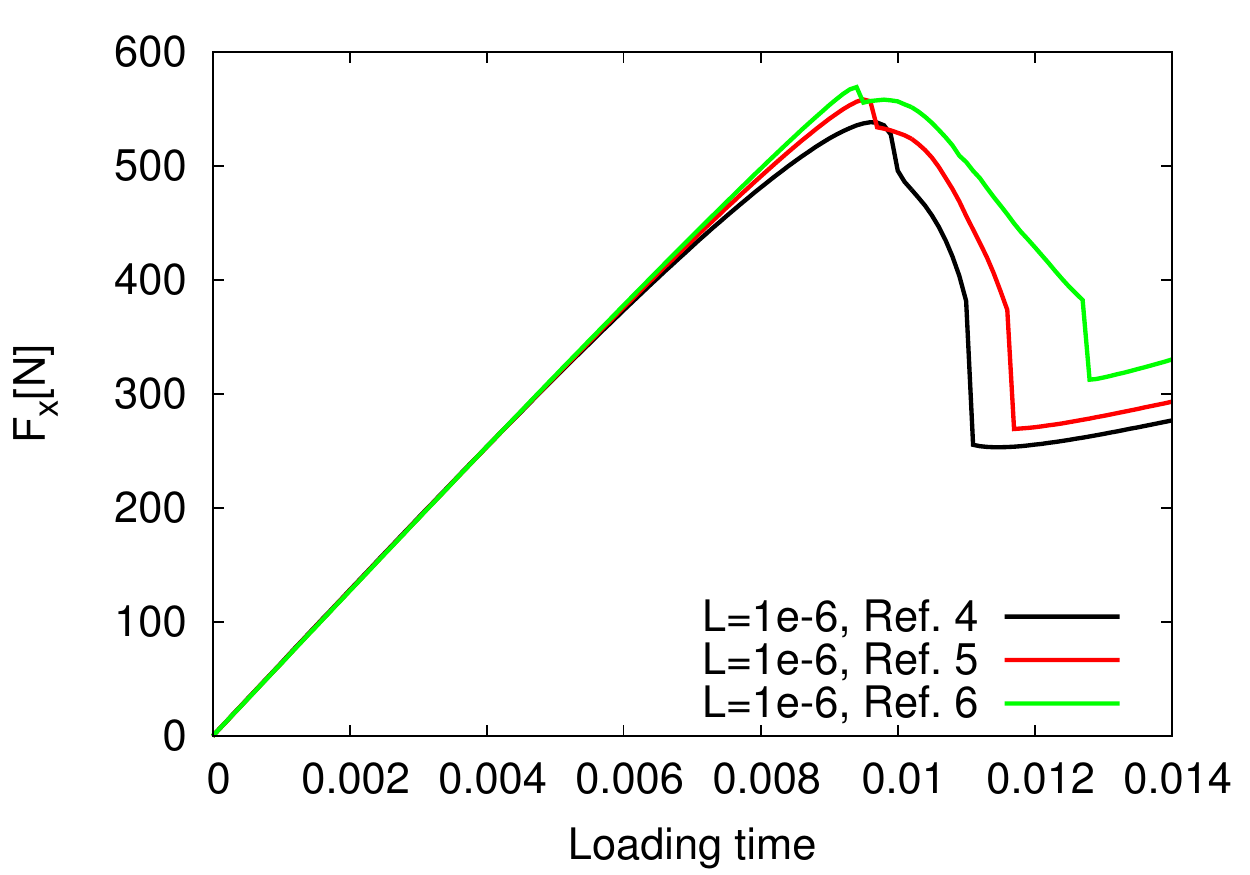}} 
{\includegraphics[width=8cm]{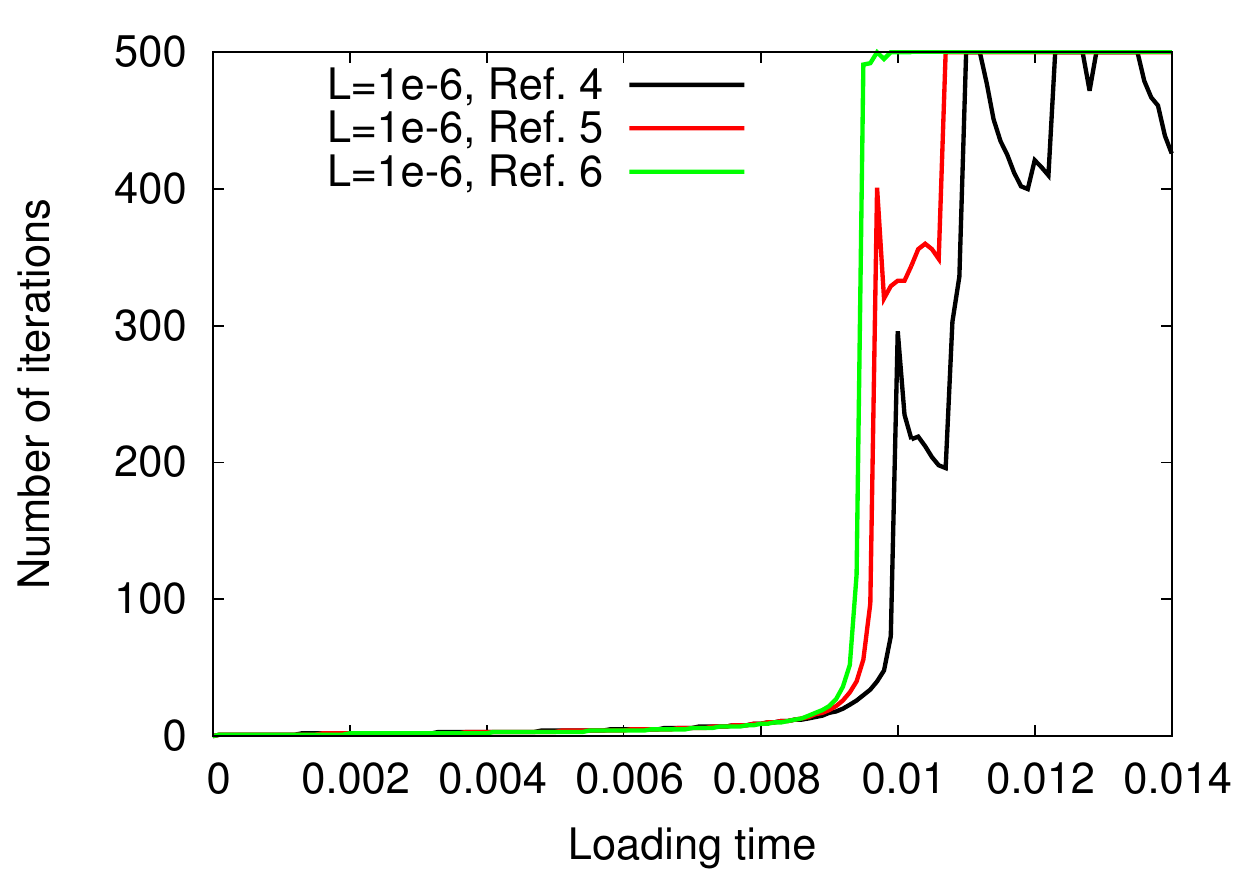}} 
\caption{Example 2: Using $L=1e-6$, comparing different mesh refinement levels $4,5,6$.
At left, the
load-displacement curves displaying the evolution of $F_x$
versus the loading time. 
At right, the number of iterations is displayed.}
\label{miehe_shear_c}
\end{figure}
\begin{figure}[H]
\centering
{\includegraphics[width=8cm]{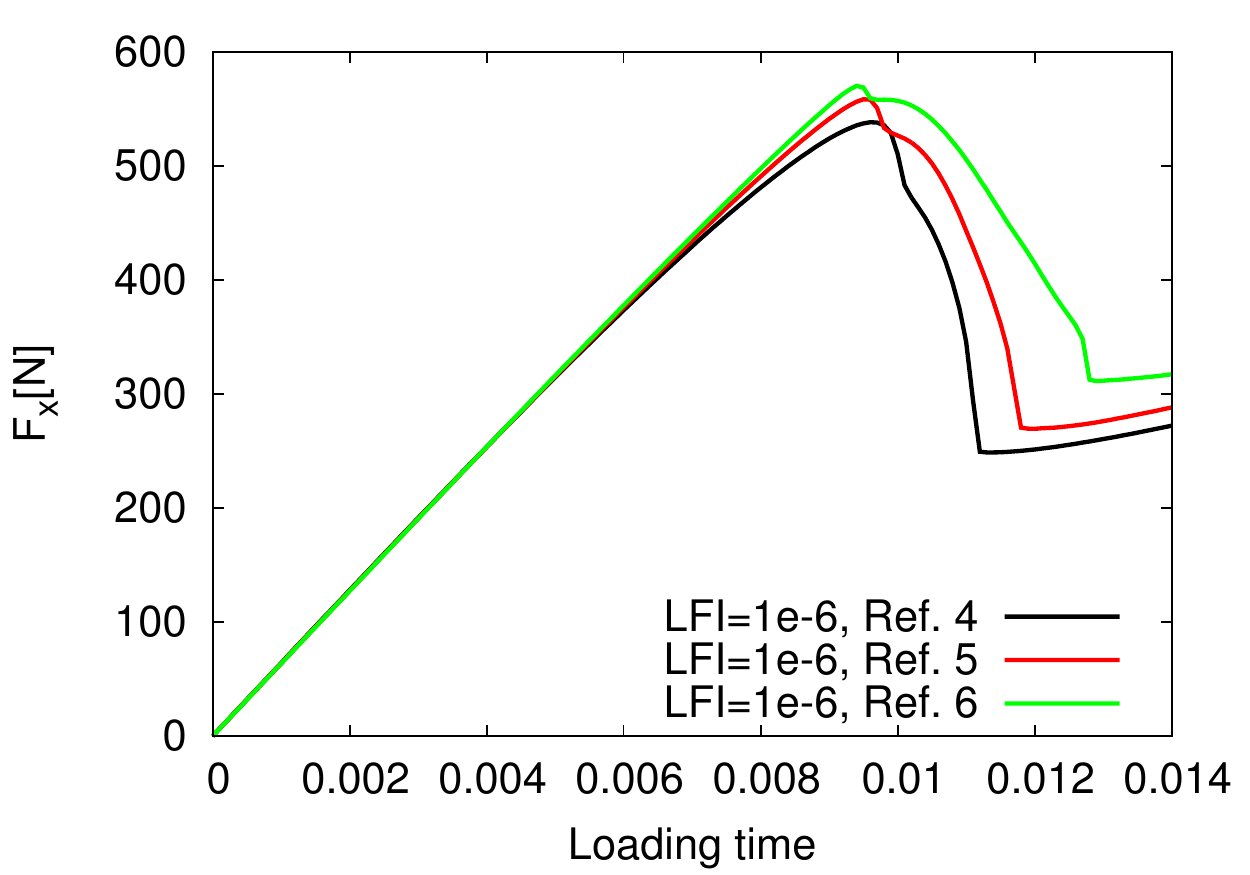}} 
{\includegraphics[width=8cm]{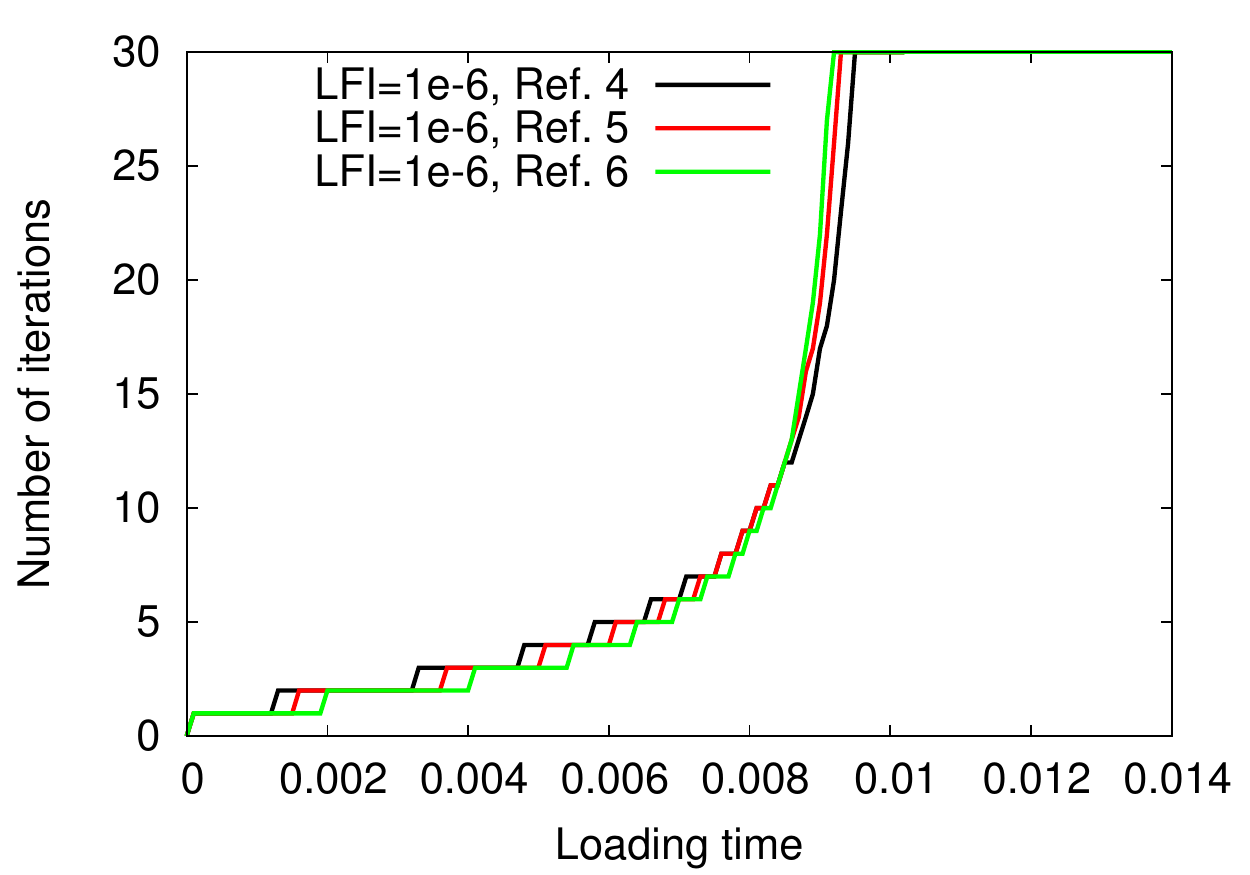}} 
\caption{Example 2: Using $L=1e-6$ and fixing the number of iterations by
  $30$, we compare different mesh refinement levels $4,5,6$.
At left, the
load-displacement curves displaying the evolution of $F_x$
versus the loading time. 
At right, the number of iterations is displayed.}
\label{miehe_shear_d}
\end{figure}
\begin{figure}[H]
\centering
{\includegraphics[width=8cm]{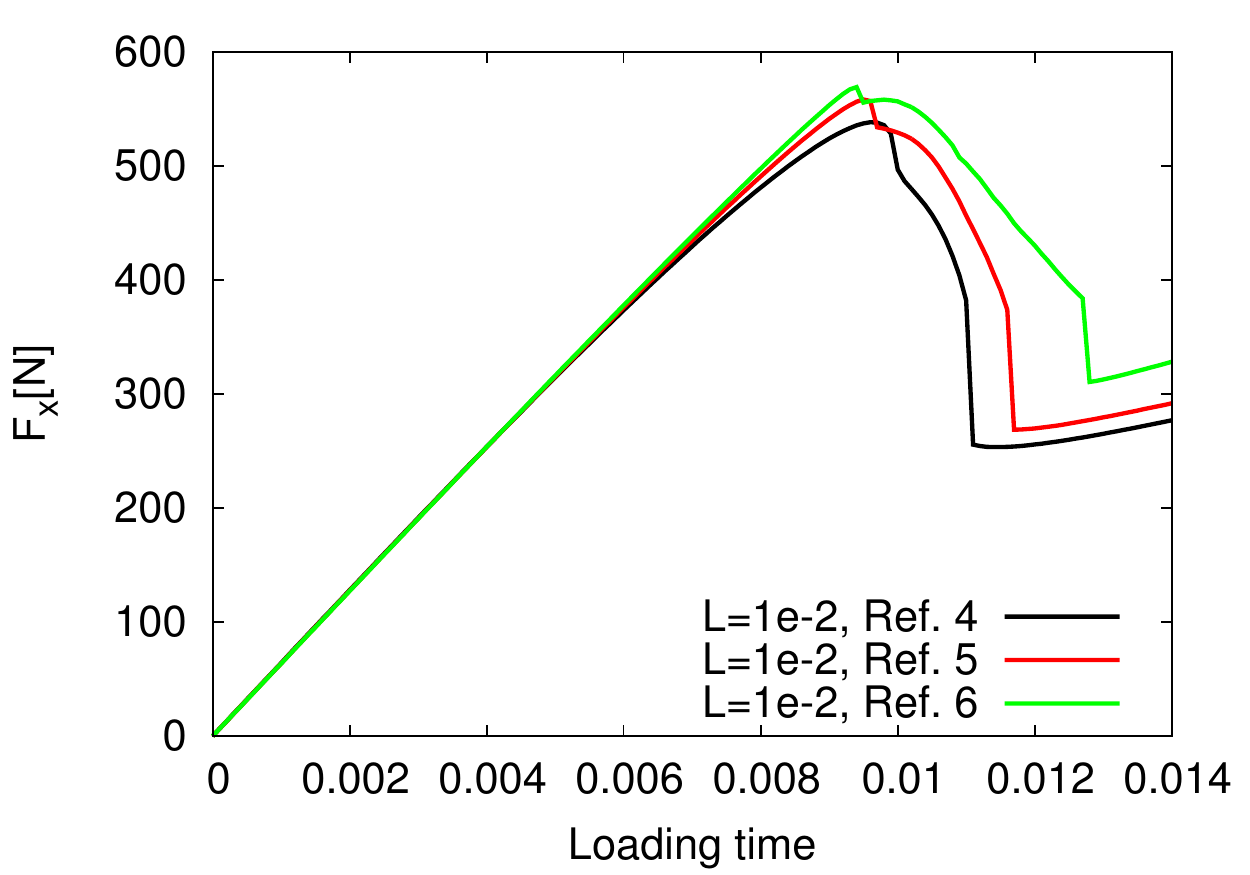}} 
{\includegraphics[width=8cm]{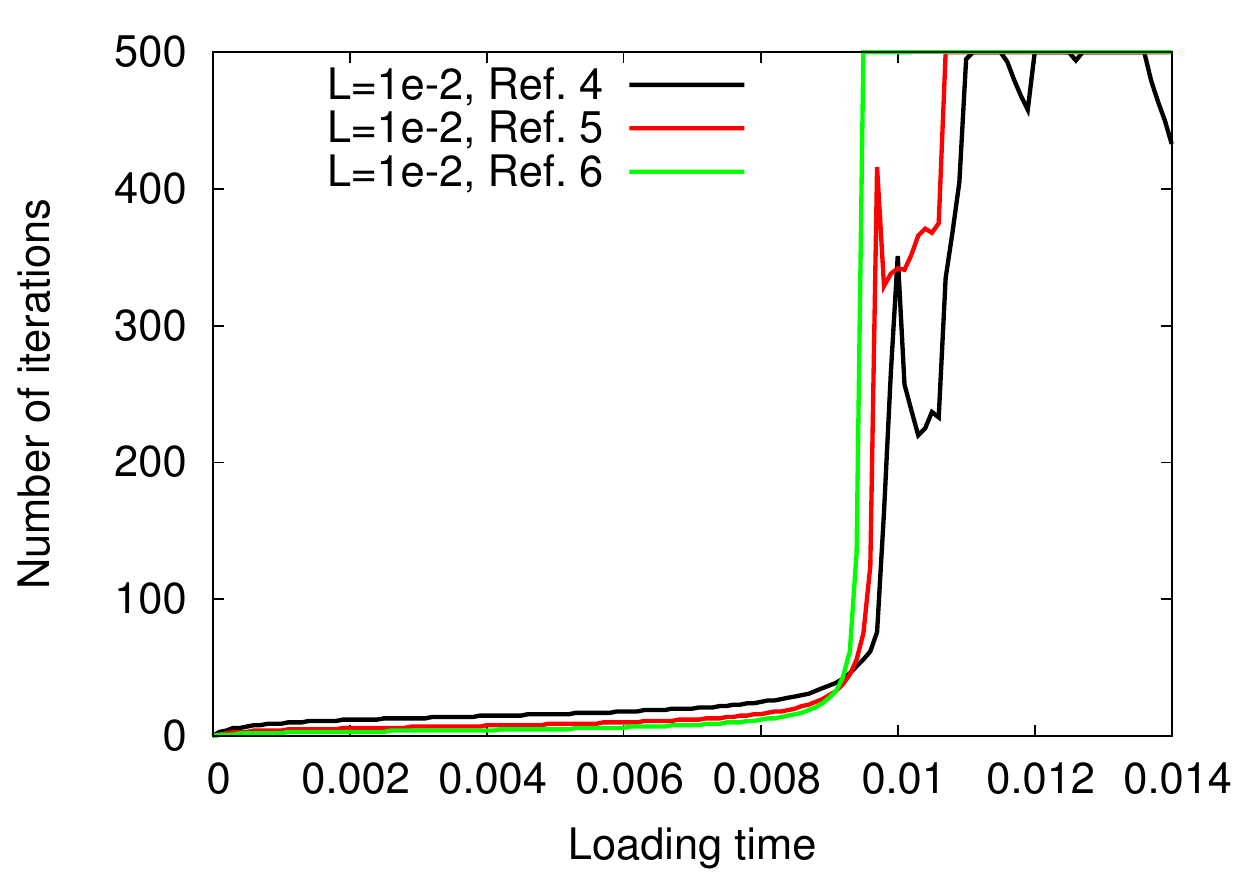}} 
\caption{Example 2: Using $L=1e-2$, we compare different mesh refinement levels $4,5,6$.
At left, the
load-displacement curves displaying the evolution of $F_x$
versus the loading time. 
At right, the number of iterations is displayed.}
\label{miehe_shear_e}
\end{figure}
\begin{figure}[H]
\centering
{\includegraphics[width=8cm]{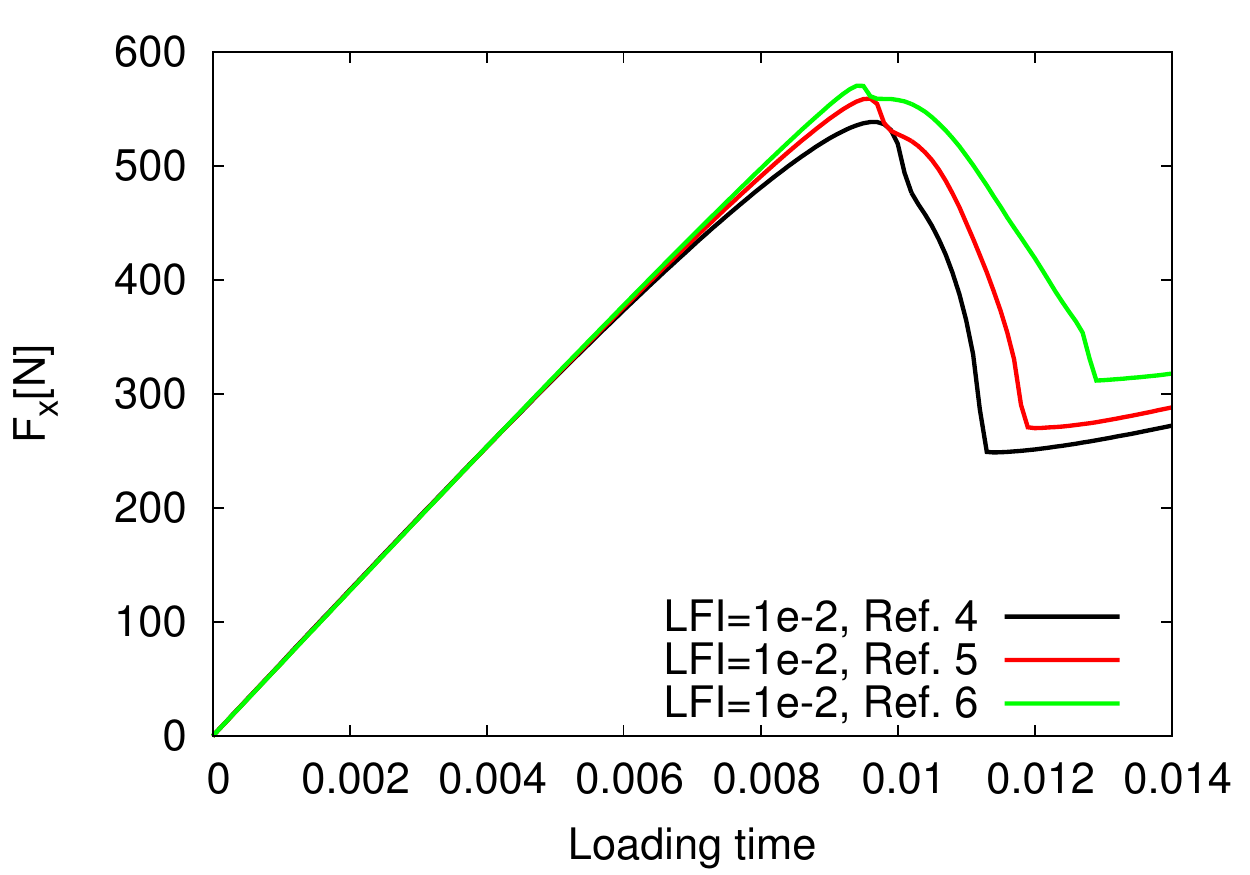}} 
{\includegraphics[width=8cm]{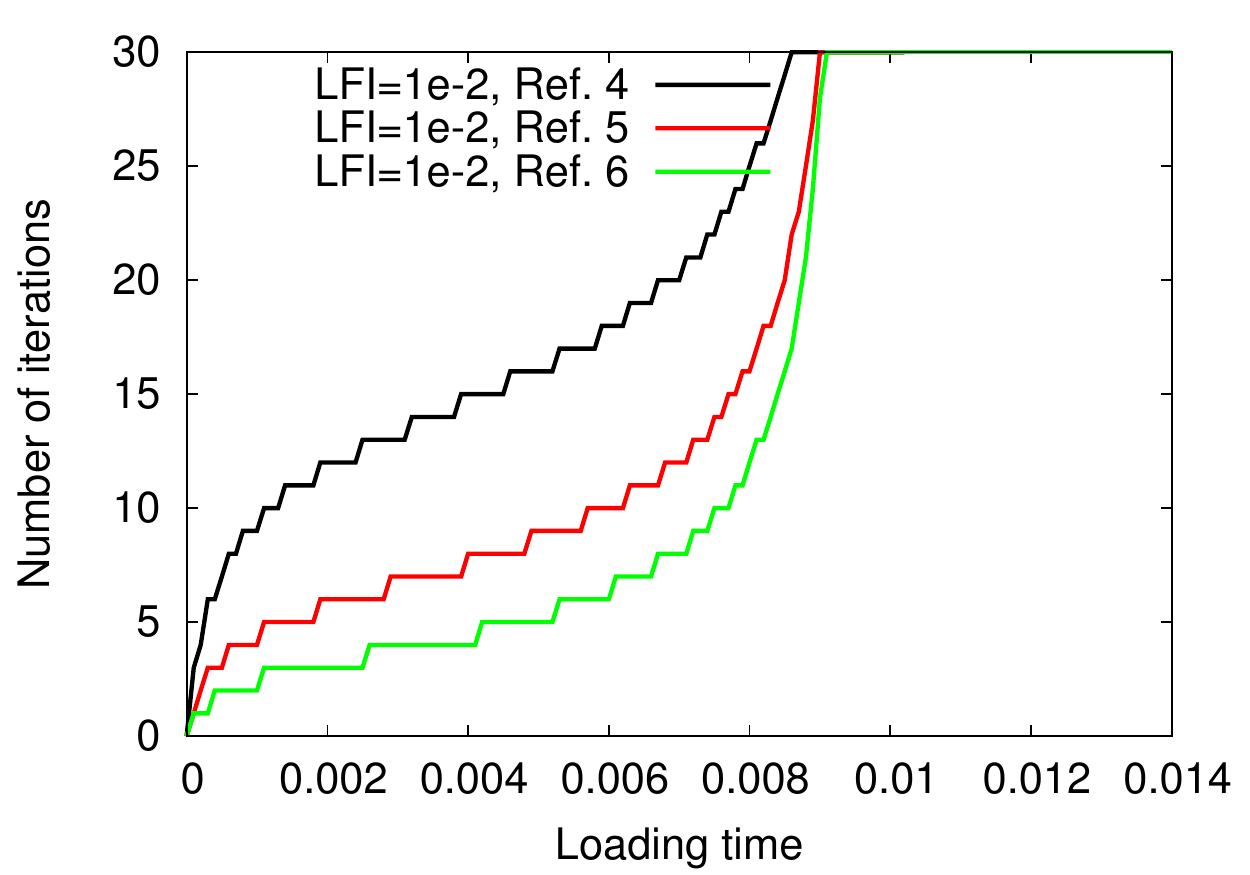}} 
\caption{Example 2: Using $L=1e-2$ and fixing the number of iterations by
  $30$, we compare different mesh refinement levels $4,5,6$.
At left, the
load-displacement curves displaying the evolution of $F_x$
versus the loading time. 
At right, the number of iterations is displayed.}
\label{miehe_shear_f}
\end{figure}
The results are very comparable to the published
literature.
In particular, it is nowadays known that the proposed Miehe et al. stress
splitting 
does not release all stresses once the specimen is broken (see
\cite{AmGeraLoren15}) and it is also known that we do not see convergence 
of the curves when both $h$ and $\eps$ are refined (see \cite{HeWheWi15}).
\subsection{L-shaped panel}
For the configuration of this third example we refer to~\cite{AmGeraLoren15, MesBouKhon15, Wi17_SISC}, which are based on an experimental setup~\cite{Winkler2001}.
We use again the model with strain-energy split; namely~\eqref{iter2_mod}-\eqref{iter1_mod}.
Moreover, in this test a carefully imposed 
irreversibility constraint is important since the specimen 
is pushed, pulled, and again pushed (see Figure \ref{l_shaped_b} for the
loading history on the small boundary part $\Gamma_u$). In the 
pulling phase the fracture vanishes if the penalization is 
not strong enough. 

The geometry and boundary conditions are displayed in Figure \ref{ex_1_config_and_mesh}.
In contrast to the previous examples, no initial crack prescribed.
The initial mesh is $1,2$ and $3$ times uniformly refined, leading to 
$300,1200, 4800$ mesh elements, with $h=$ \SI{29.1548}{mm}, \SI{14.577}{mm}, \SI{7.289}{mm}, respectively.

We increase the displacement $u_D :=u_y= u_y(t)$ on 
$\Gamma_{u}:=\{(x,y)\in B| \, \SI{470}{mm} \leq x\leq \SI{500}{mm}, \, y= \SI{250}{mm}\}$ 
over time, where 
$\Gamma_u$ is a section 
of \SI{30}{mm} length on the right corner of the specimen.  
We apply
a loading-dependent, non-homogeneous Dirichlet condition (see also Figure \ref{l_shaped_b}):
\begin{equation}
\label{uy_test_2}
\begin{aligned}
u_y &= t\cdot \bar{u}, &\quad \bar{u} = \SI{1}{mm/s}, \quad \SI{0.0}{s}\leq t < \SI{0.3}{s},\\
u_y &= (0.6 - t)\cdot \bar{u}, &\quad \bar{u} = \SI{1}{mm/s}, \quad
\SI{0.3}{s}\leq t < \SI{0.8}{s},\\
u_y &= (-1 + t)\cdot \bar{u}, &\quad \bar{u} = \SI{1}{mm/s}, \quad \SI{0.8}{s}\leq t\leq \SI{2.0}{s},
\end{aligned}
\end{equation}
where $t$ denotes the total loading time. Due to this cyclic loading the total
displacement at the end time $T$ = \SI{2}{s} is \SI{1}{mm}.
\begin{figure}[H]
\centering
{\includegraphics[width=8cm]{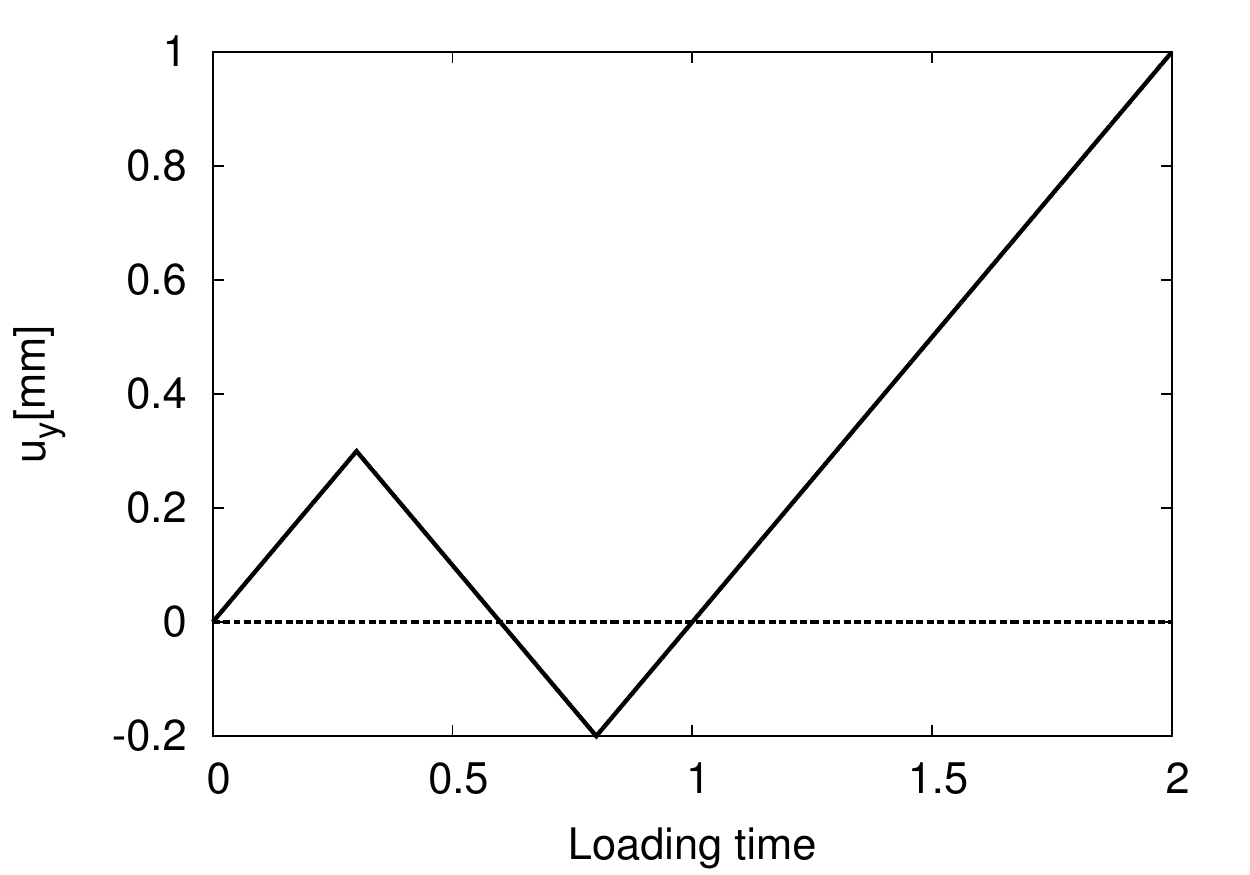}} 
\caption{Example 3: Loading history on $\Gamma_u$ for the L-shaped panel test.}
\label{l_shaped_b}
\end{figure}

\begin{figure}[H]
\centering
{\includegraphics[width=8cm]{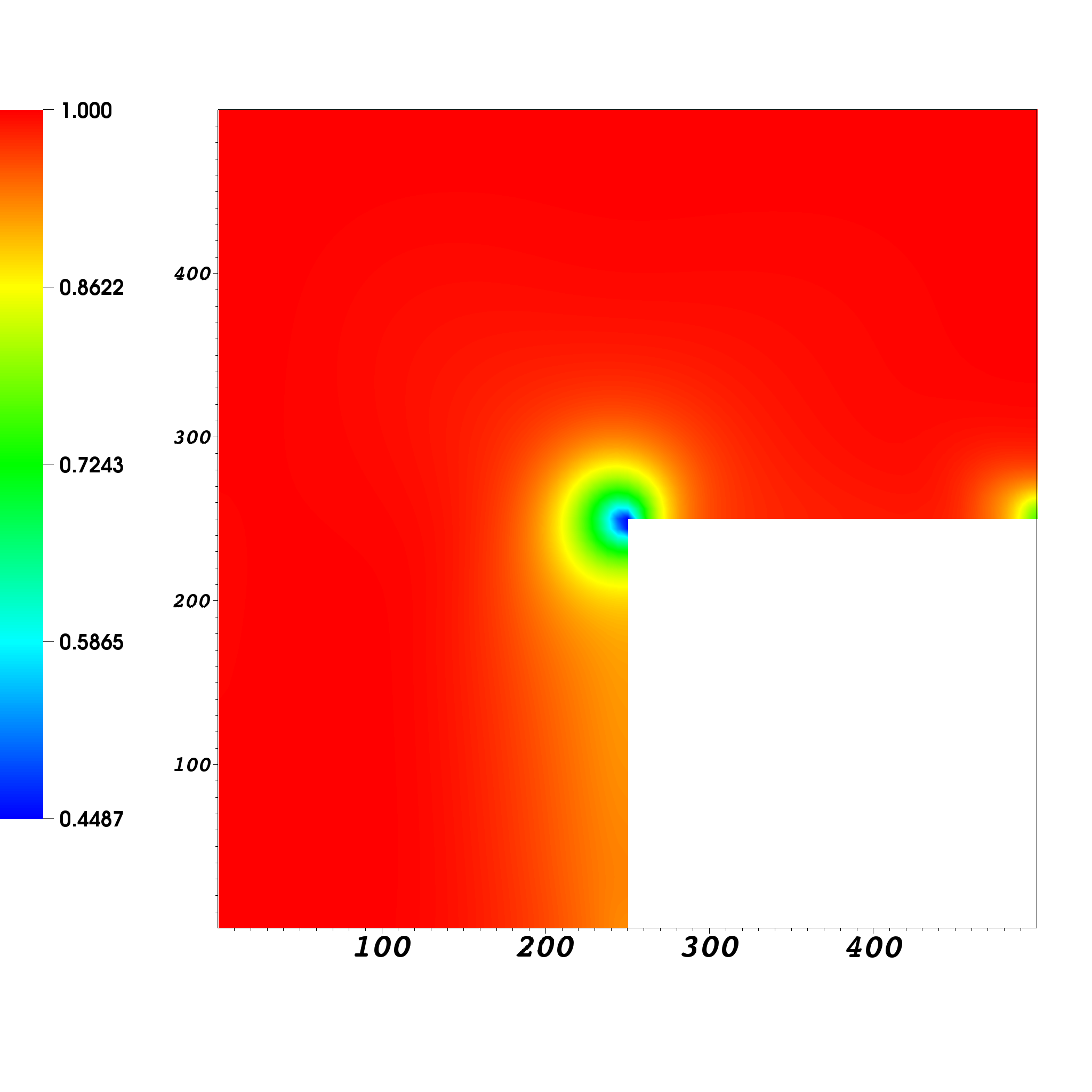}} 
{\includegraphics[width=8cm]{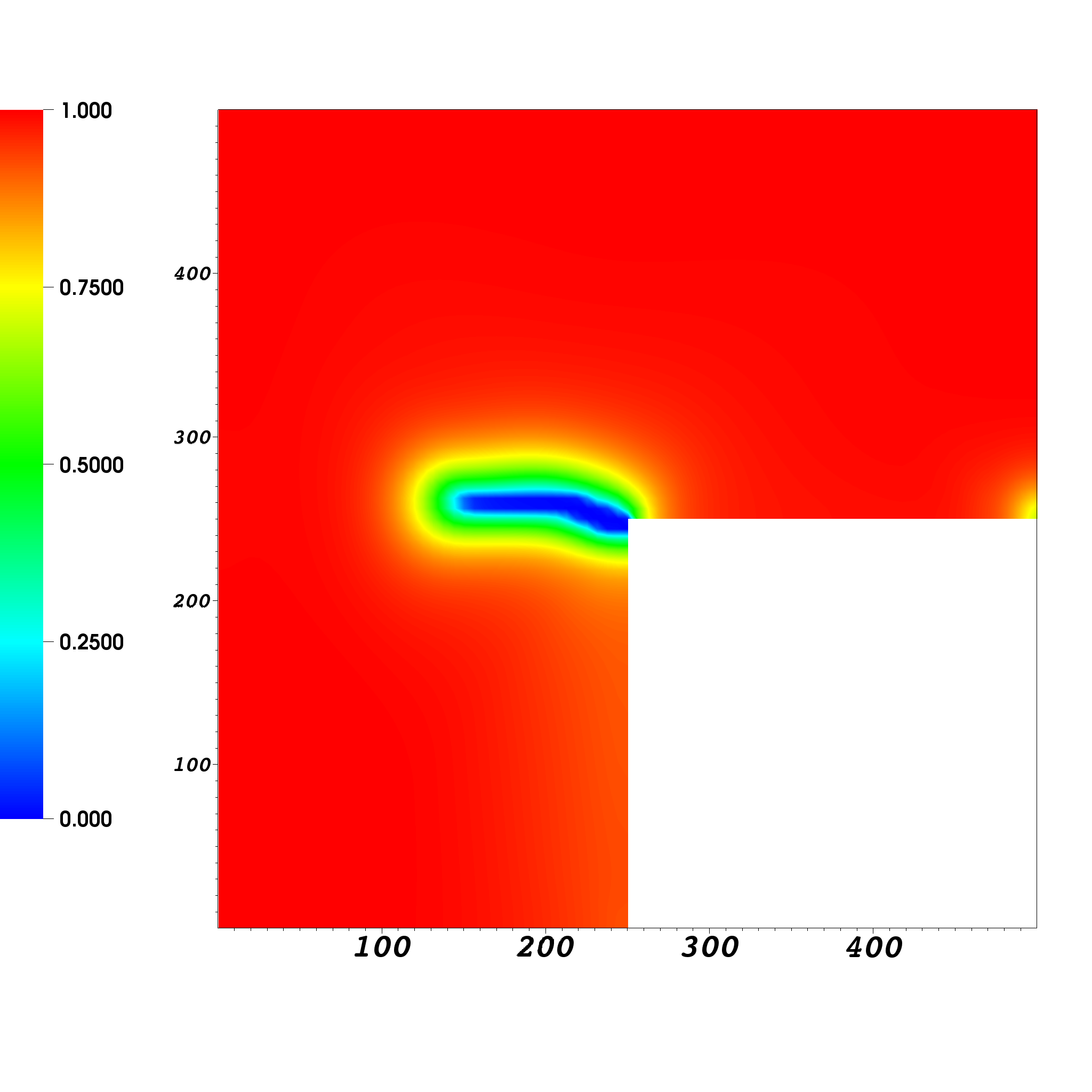}} 
{\includegraphics[width=8cm]{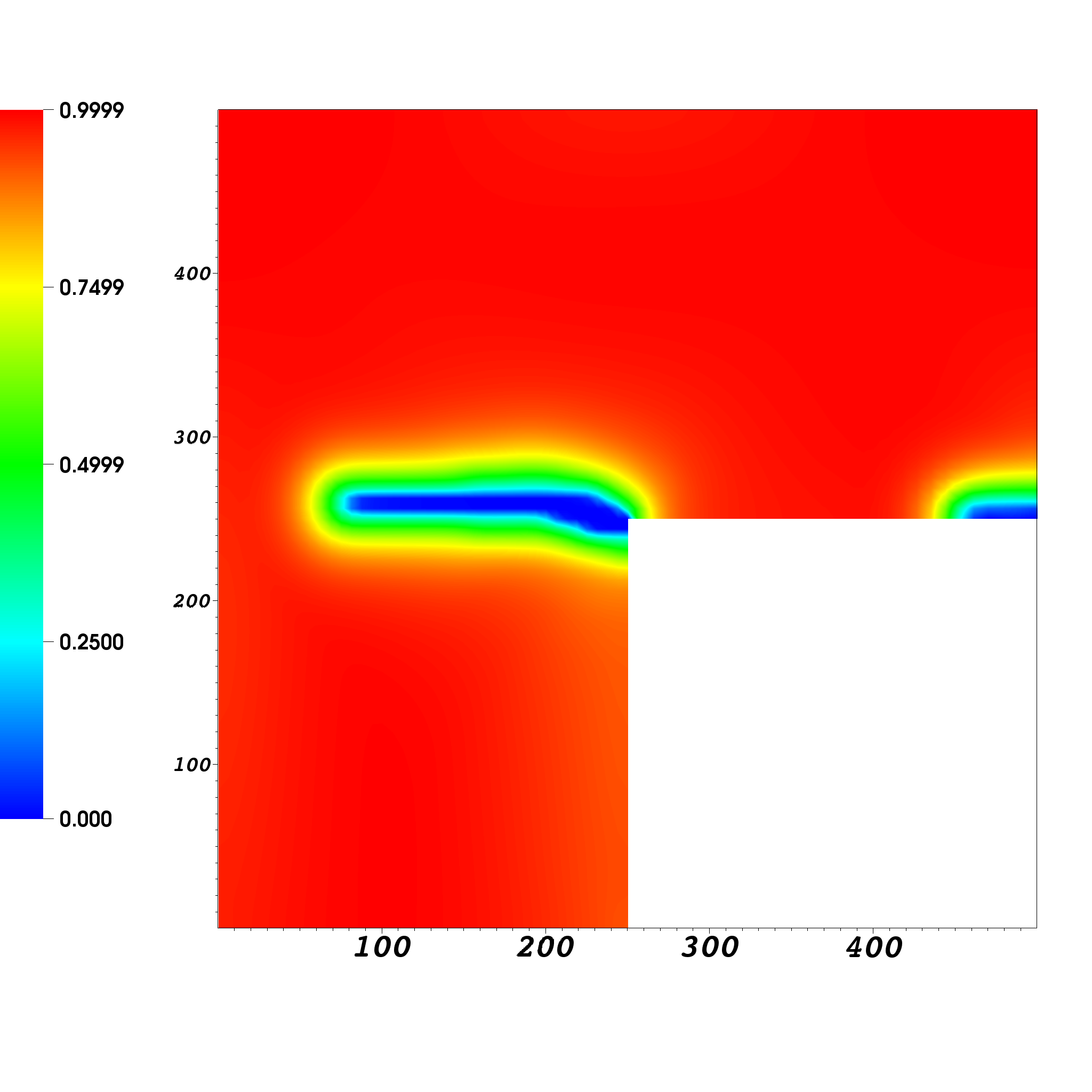}} 
{\includegraphics[width=8cm]{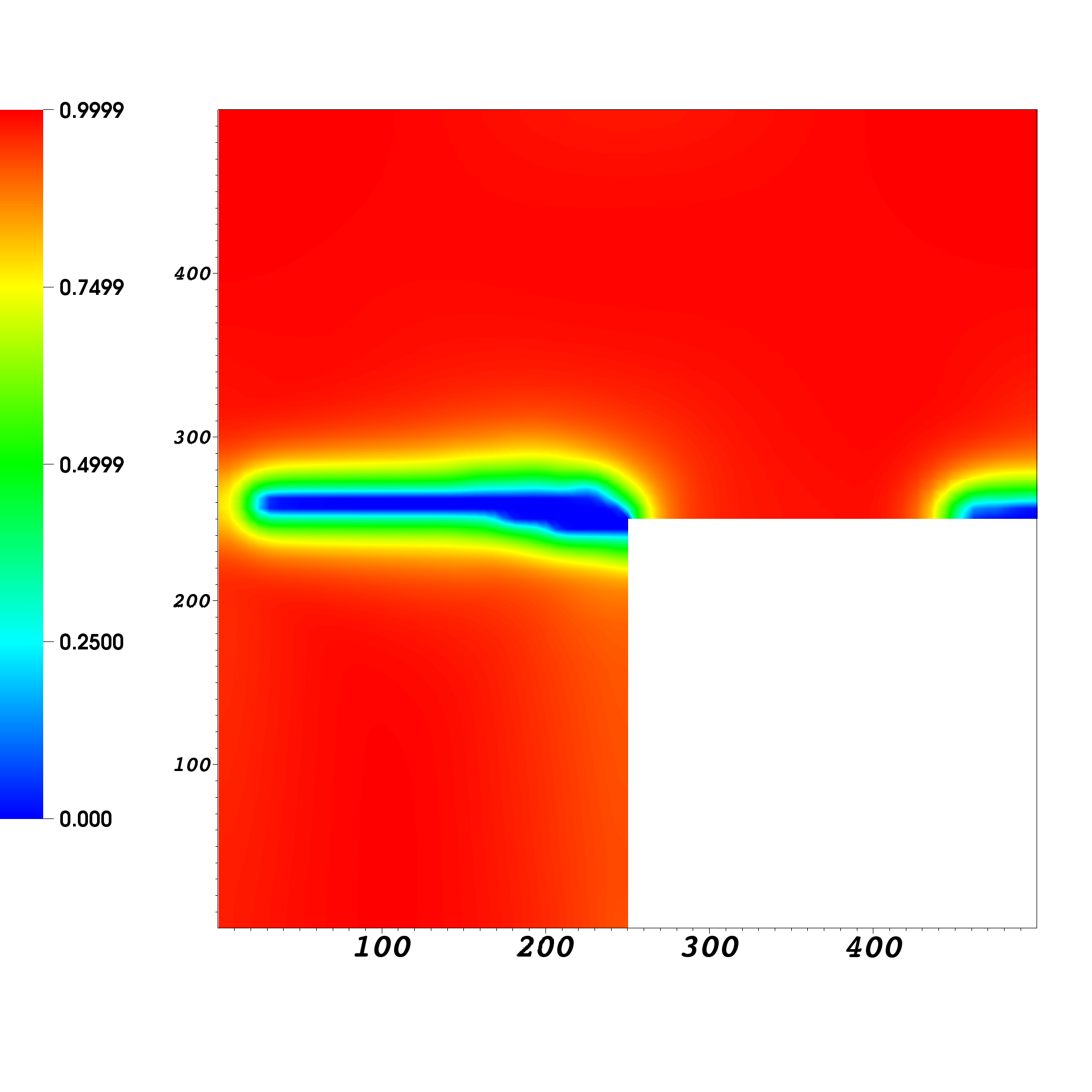}} 
\caption{Example 3: crack path of the L-shaped panel test at the loading 
steps $220, 300, 1450, 2000$.}
\label{l_shaped_a}
\end{figure}

\newpage
We use
$\mu_s$ = \SI{10.95}{kN/mm^2}, $\lambda_s$ = \SI{6.16}{kN/mm^2}, and
$G_c$ = \SI{8.9d-5}{kN/mm}. 
The time (loading) step size is $\delta t$ =\SI{e-3}{s}.
Furthermore, we set $k = 10^{-10}h$[mm] and 
$\eps=2h$.  As before, we observe the number of Newton iterations and 
we evaluate the surface load vector 
\begin{equation*}
\tau = (F_x,F_y) := \int_{\Gamma_{u}} \bsig(u)\nu\, \textnormal{d}s,
\end{equation*}
with normal vector $\nu$,
and now we are particularly interested in $F_y$. The crack path at the chosen time step snapshots in Figure \ref{l_shaped_a}
corresponds to the published literature \cite{Wi17_SISC,MesBouKhon15,AmGeraLoren15}.
The load-displacement curves and the number of iterations for different $L$
and corresponding mesh refinement studies are displayed in the Figures 
\ref{l_shaped_c1}, \ref{l_shaped_c}, \ref{l_shaped_d}, \ref{l_shaped_e} 
and \ref{l_shaped_f}.

\begin{figure}[H]
\centering
{\includegraphics[width=8cm]{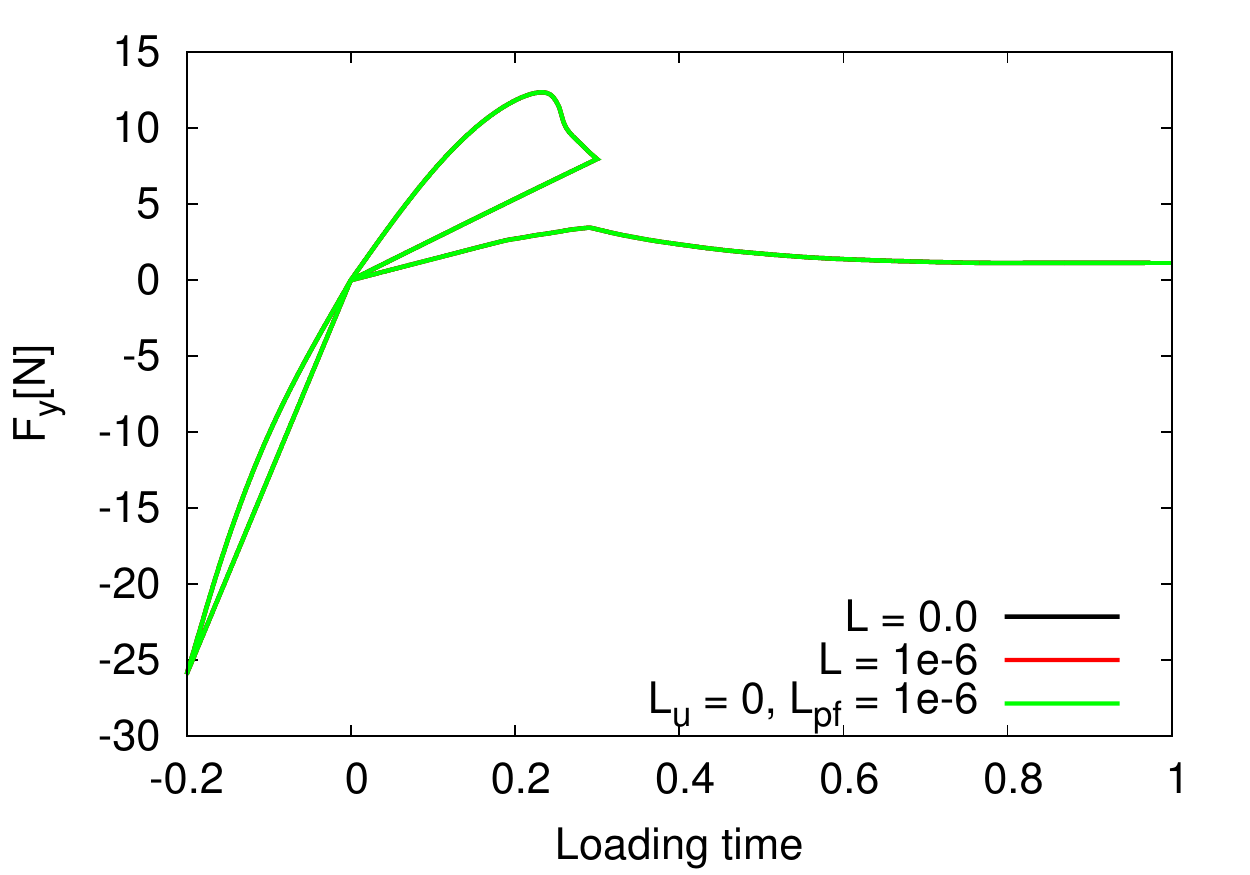}} 
{\includegraphics[width=8cm]{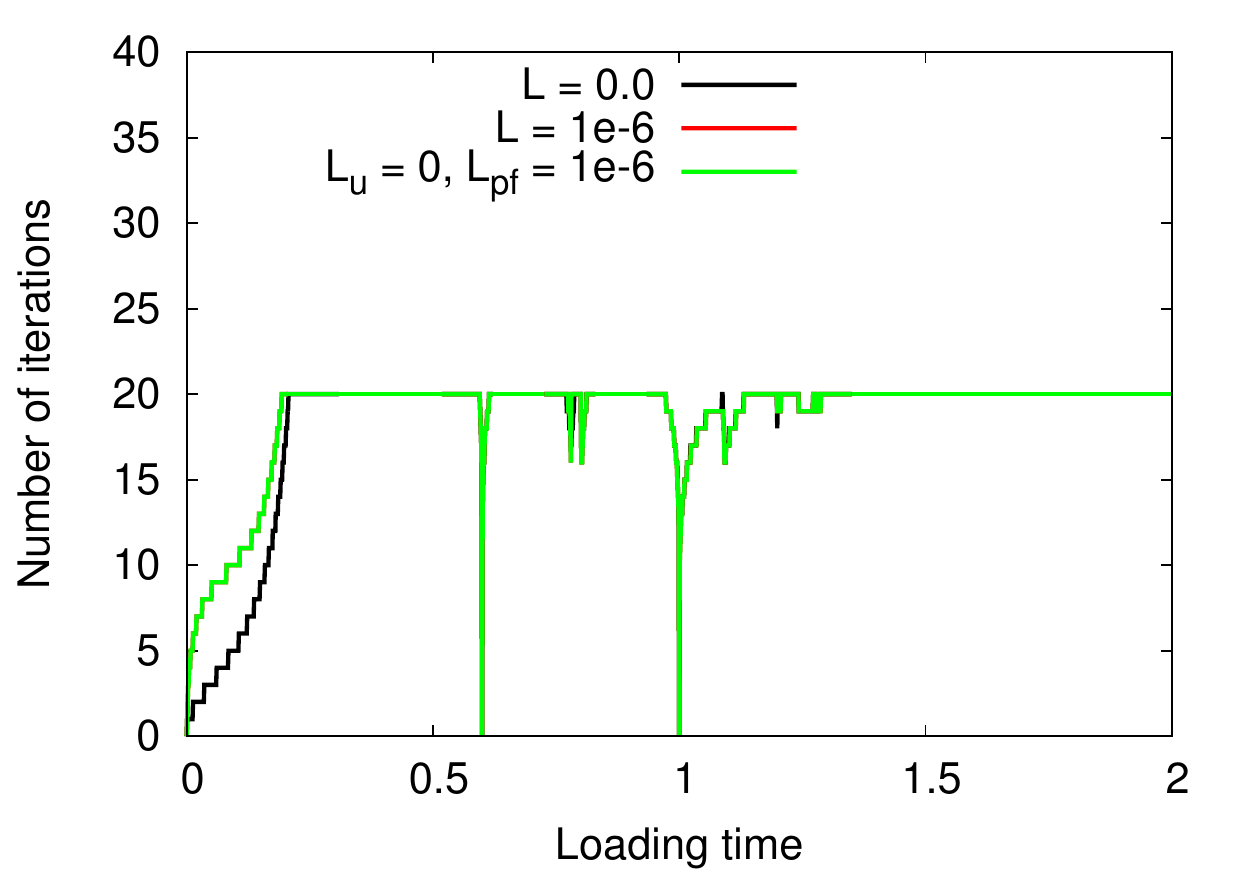}} 
\caption{Example 3: Comparison of different $L$. Observe that stabilizing the mechanics subproblem has no effect in this example. At left, the
load-displacement curves displaying the evolution of $F_y$
versus $u_y$ are shown. At right, the number of staggered iterations is displayed.}
\label{l_shaped_c1}
\end{figure}

\begin{figure}[H]
\centering
{\includegraphics[width=8cm]{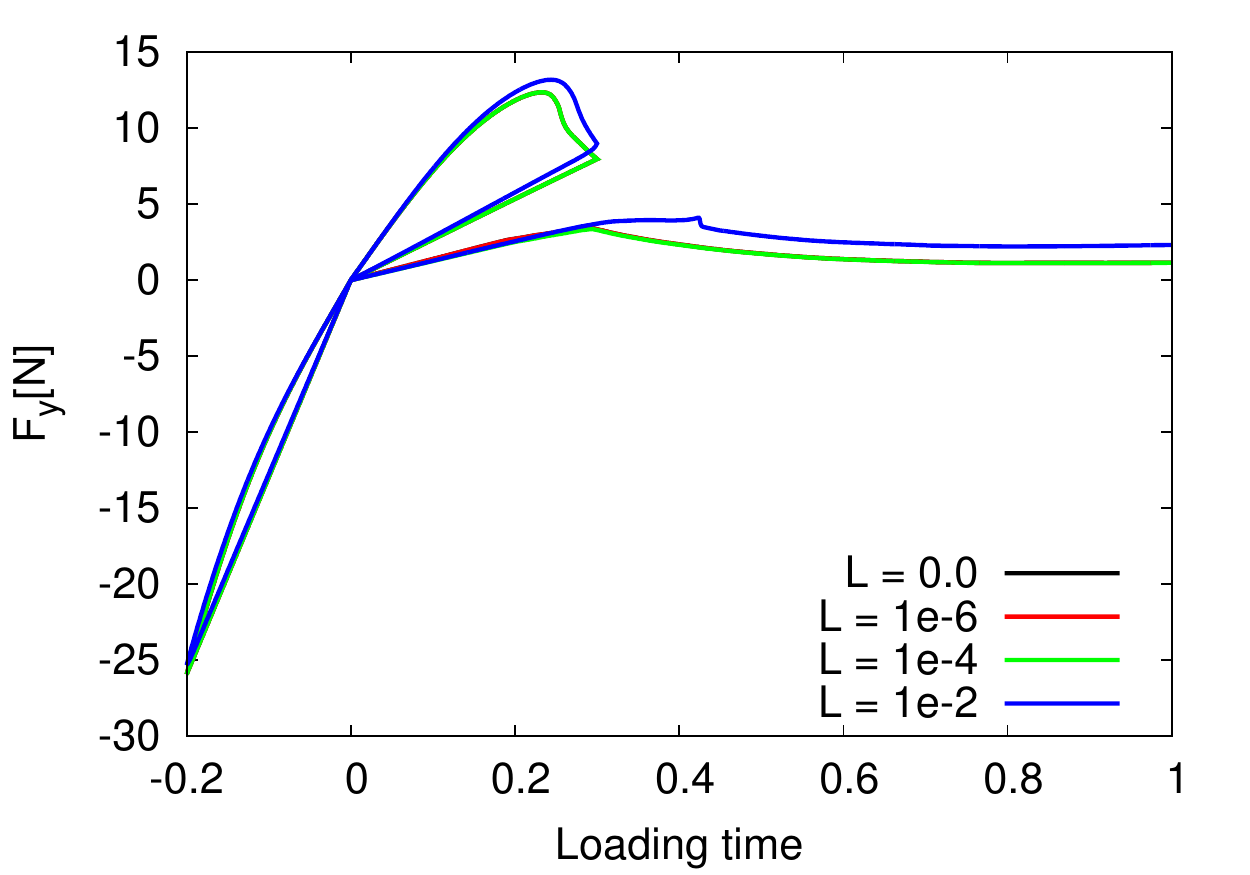}} 
{\includegraphics[width=8cm]{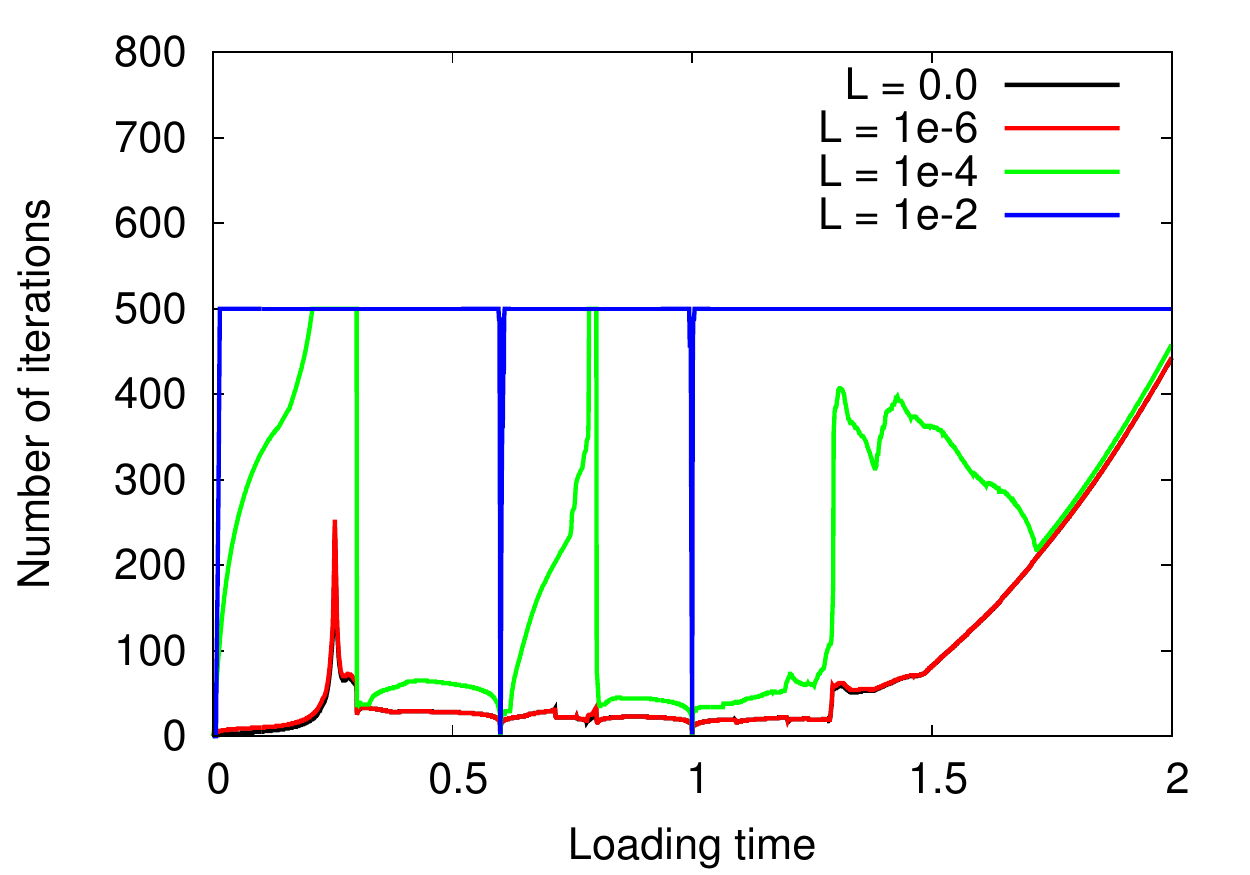}} 
\caption{Example 3: Comparison of different $L$ with an open number of
  staggered iterations (fixed by $500$ though). At left, the
load-displacement curves displaying the evolution of $F_y$
versus $u_y$ are shown. At right, the number of staggered iterations is displayed.}
\label{l_shaped_c}
\end{figure}

\begin{figure}[H]
\centering
{\includegraphics[width=8cm]{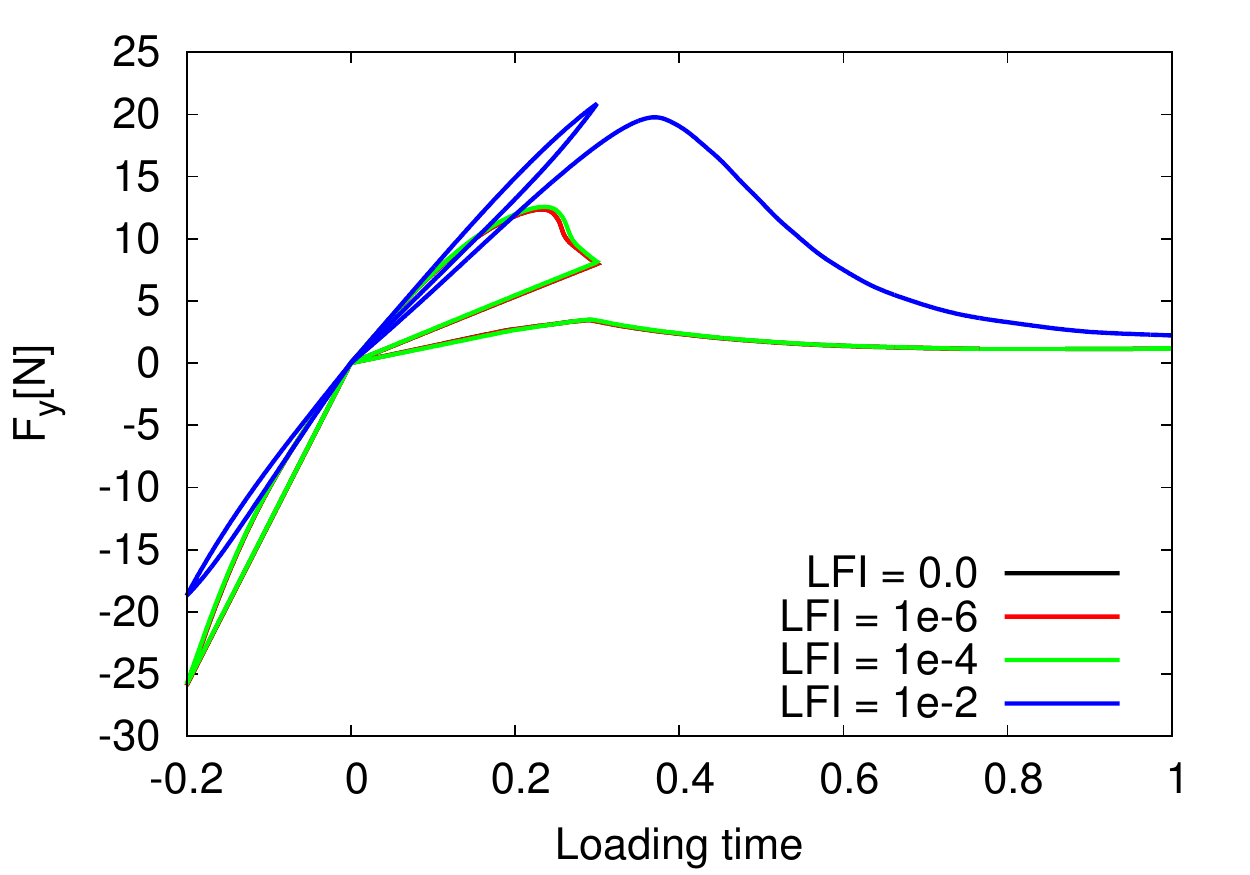}} 
{\includegraphics[width=8cm]{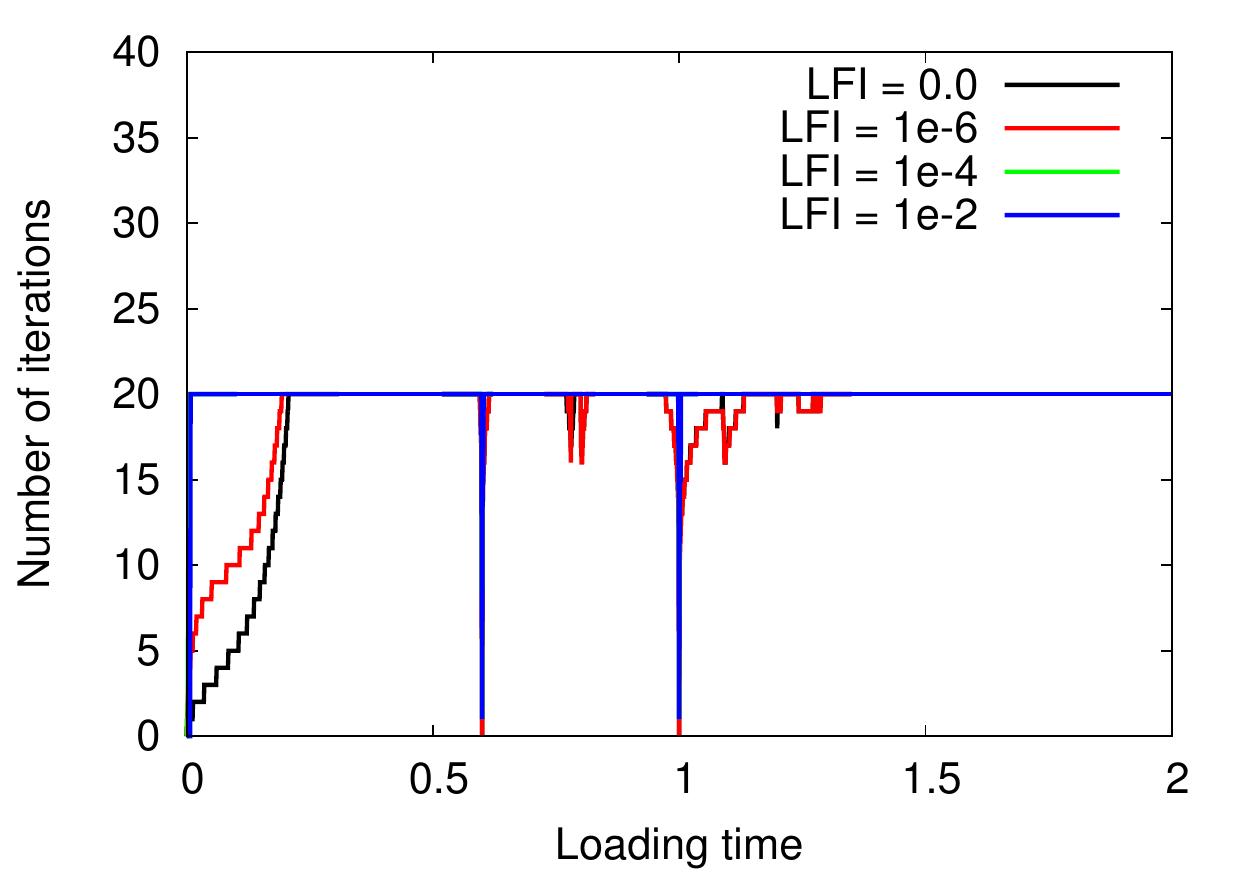}} 
\caption{Example 3: Comparison of different $L$ with a fixed number of $20$
  staggered iterations. At left, the
load-displacement curves displaying the evolution of $F_y$
versus $u_y$ are shown. At right, the number of staggered iterations is displayed.}
\label{l_shaped_d}
\end{figure}

\begin{figure}[H]
\centering
{\includegraphics[width=8cm]{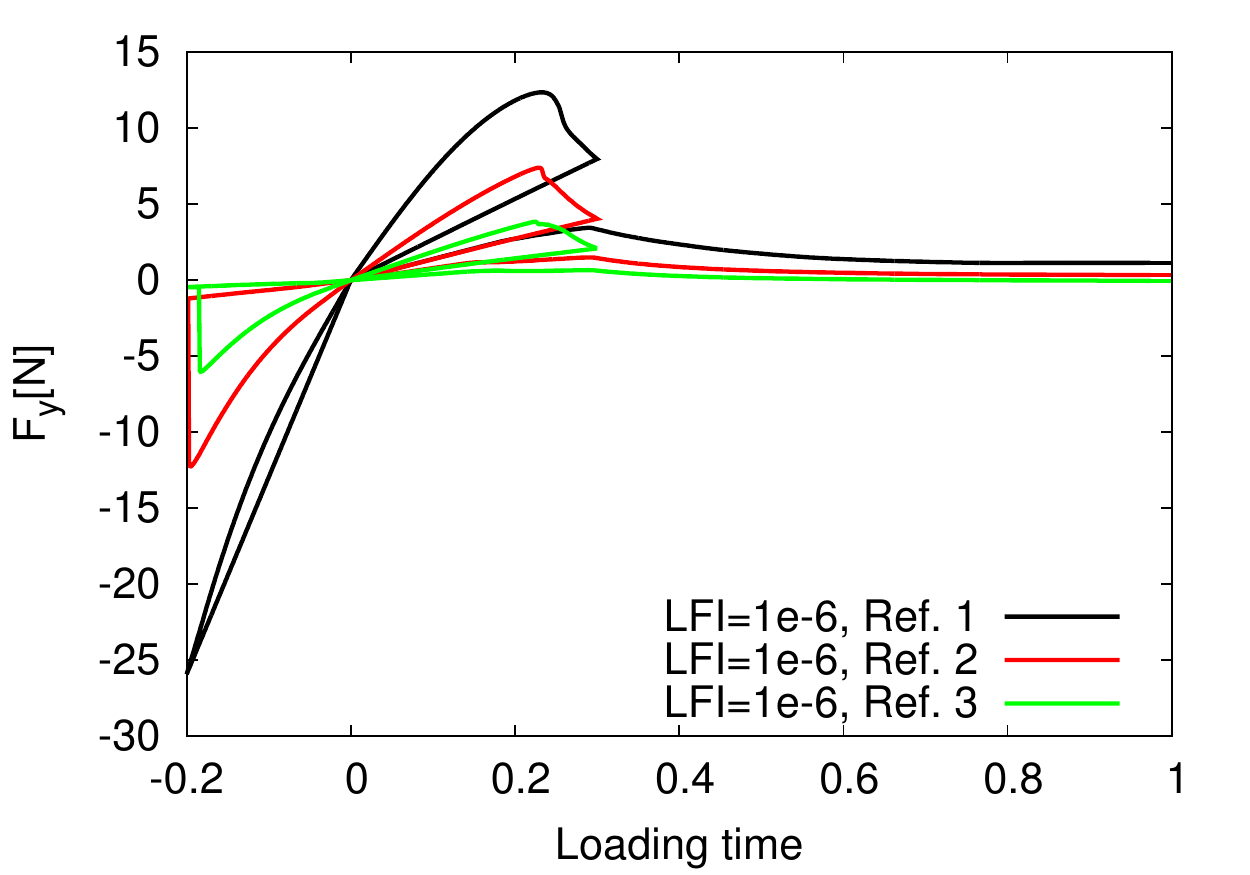}} 
{\includegraphics[width=8cm]{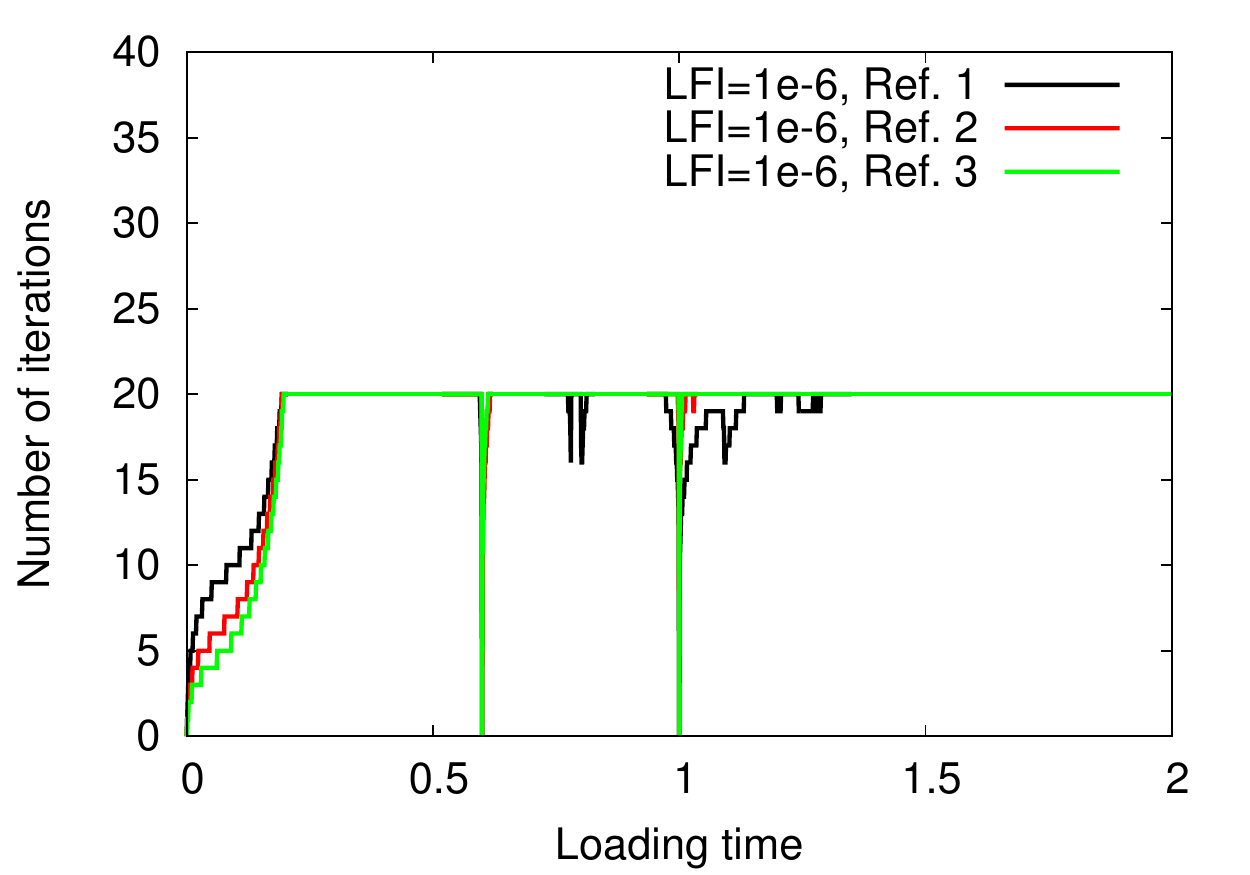}} 
\caption{Example 3: Using $L=1e-6$ and a fixed number of $20$
  staggered iterations, we compare the results on different refinement levels $1,2,3$. At left, the
load-displacement curves displaying the evolution of $F_y$
versus $u_y$ are shown. At right, the number of staggered iterations is displayed.}
\label{l_shaped_e}
\end{figure}

\begin{figure}[H]
\centering
{\includegraphics[width=8cm]{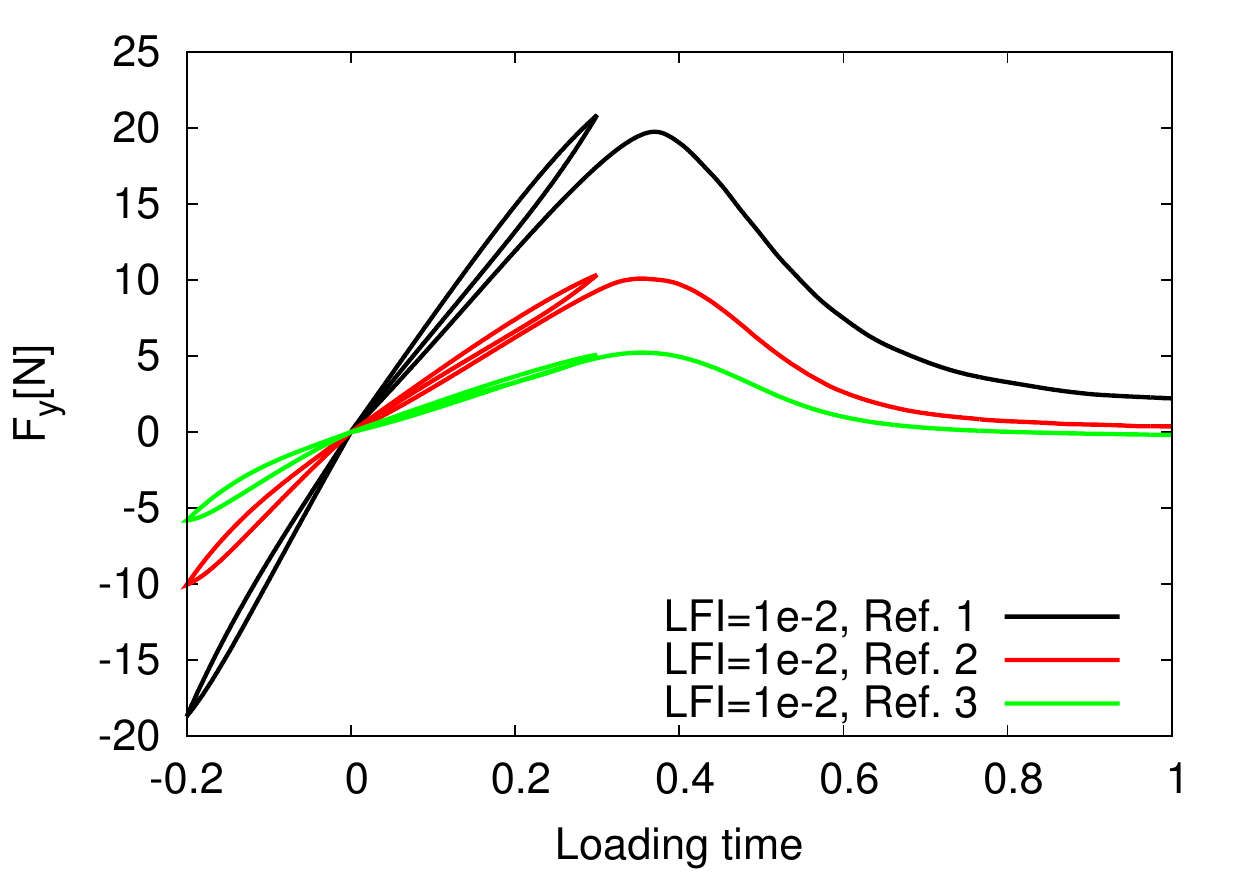}} 
{\includegraphics[width=8cm]{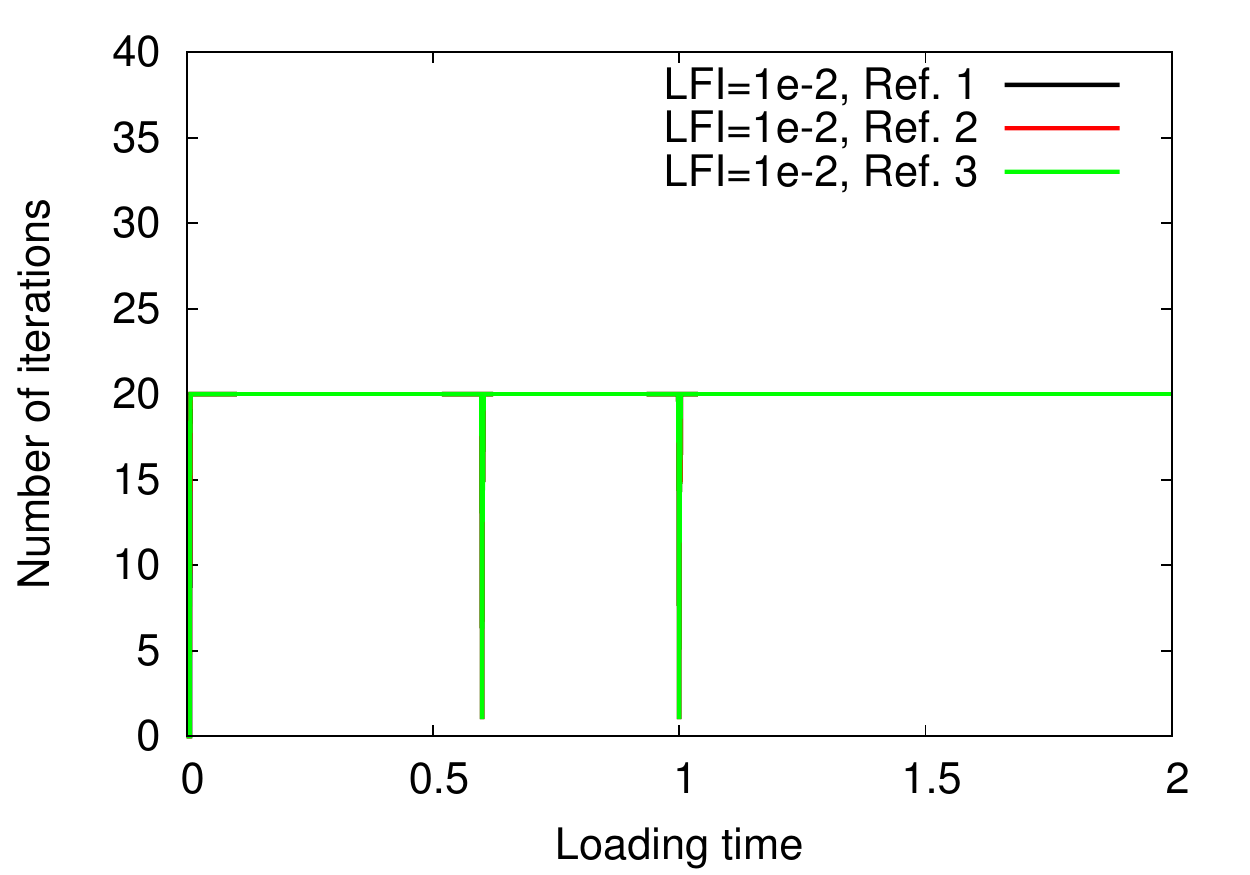}} 
\caption{Example 3: Using $L=1e-2$ and a fixed number of $20$
  staggered iterations, we compare the results on different refinement levels $1,2,3$. At left, the
load-displacement curves displaying the evolution of $F_y$
versus $u_y$ are shown. At right, the number of staggered iterations is displayed.}
\label{l_shaped_f}
\end{figure}

\newpage
\subsection{Verification of Assumption~\ref{hypo}}\label{verify}
In this last set of computations, we verify whether 
Assumption~\ref{hypo} holds true in our computations. We choose some 
prototype settings, namely on the coarest mesh level Ref. 4 and
$L_u = L_{\varphi} = 1e-6$. In Figure \ref{pic_comp_strains}, we observe 
that $\ess \sup_{x\in B} |\be(u^{n}(x))|$ varies, but always can be bounded 
from above with $M>0$. The value of $\ess \sup_{x\in B} |\be(u^{n}(x))|$ is 
the final strain when the $L$-scheme terminates. The minimum and maximum
values 
shows that there are no significant variations in $\ess \sup_{x\in B}
|\be(u^{n}(x))|$ during the $L$-scheme iterations with respect to 
the finally obtained value.

\begin{figure}[H]
\centering
{\includegraphics[width=8cm]{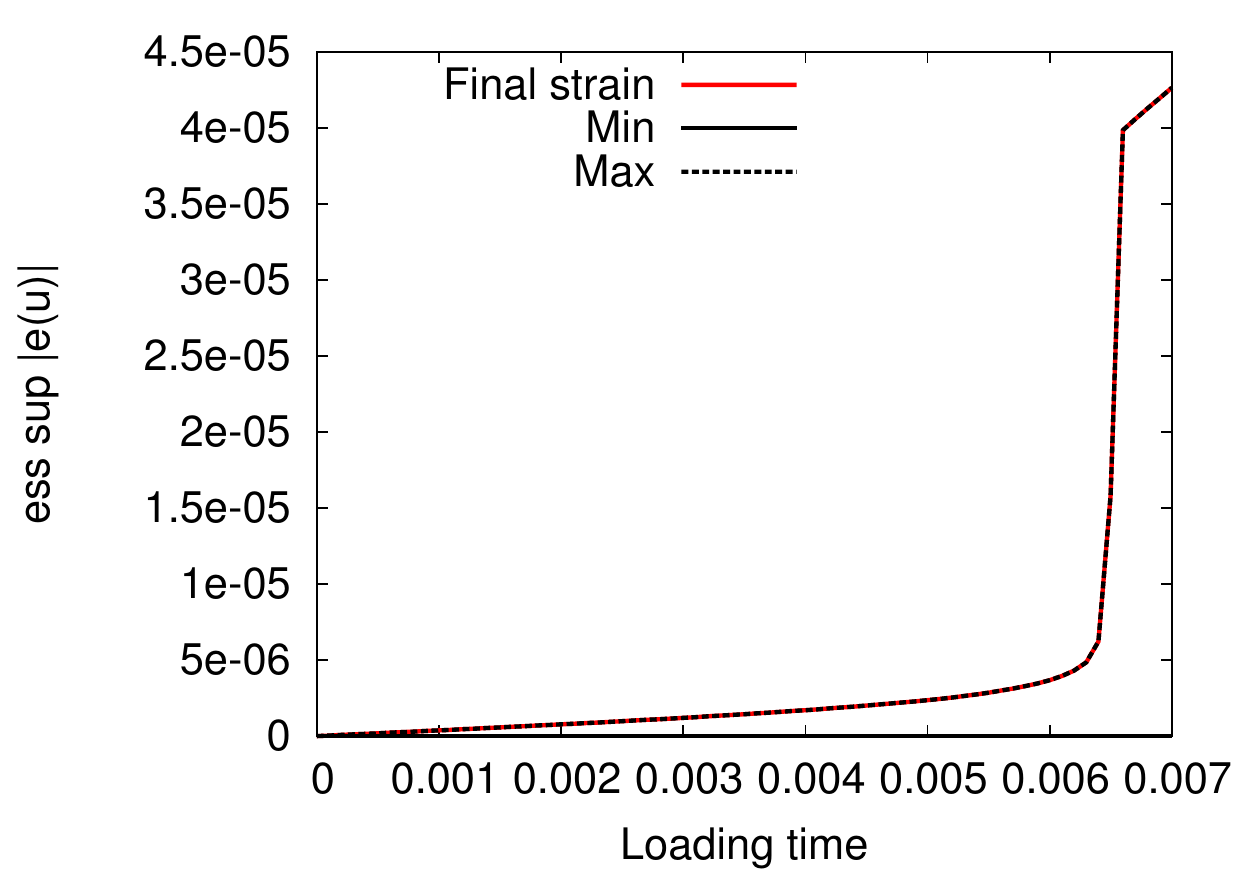}} 
{\includegraphics[width=8cm]{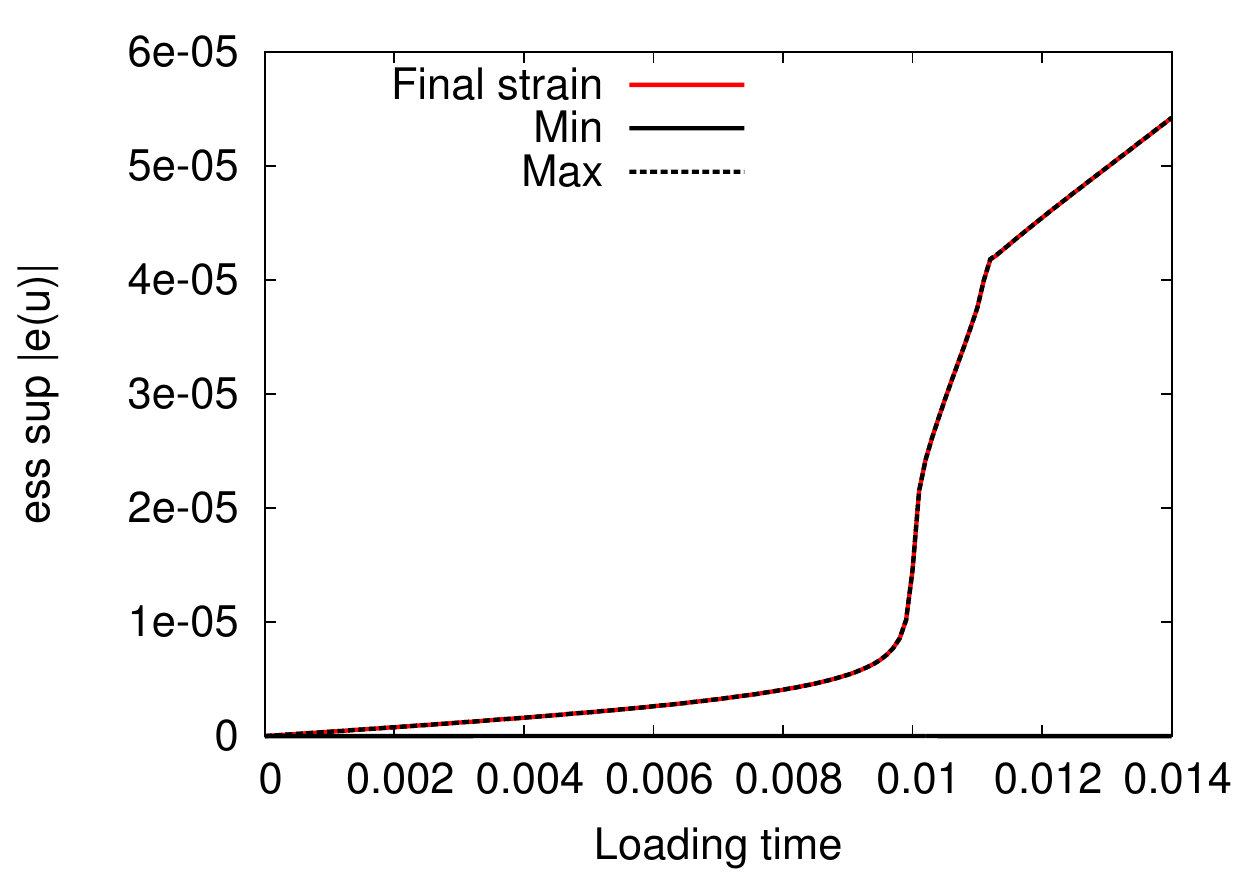}} 
{\includegraphics[width=8cm]{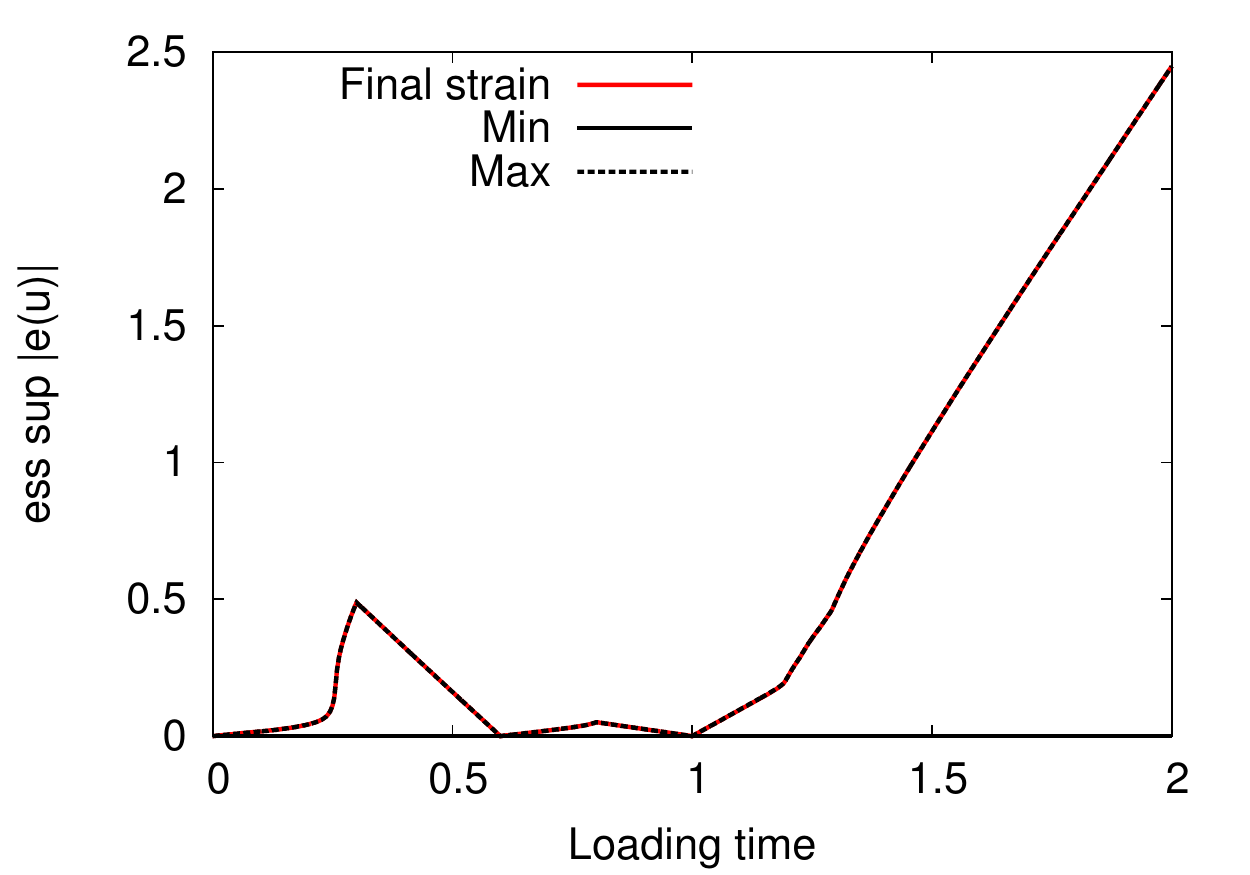}} 
\caption{Comparison of $\ess \sup_{x\in B} |\be(u^{n}(x))|$ 
and the minimal/maximal $\ess \sup_{x\in B} |\be(u^{n}(x))|$ 
per loading time step.}
\label{pic_comp_strains}
\end{figure}

\section{Conclusions}\label{sec:conclusions}
We have proposed a novel staggered iterative algorithm for brittle fracture
phase field models. This algorithm is employing stabilization and
linearization techniques known in the literature as the `$L$-scheme', 
which is a generalization of the Fixed Stress Splitting algorithm coming 
from poroelasticity. 
Through theory and numerical examples we have investigated the performance of our proposed variants of the $L$-scheme for brittle fracture phase field problems.


Under natural constraints that the elastic mechanical energy remains bounded, and that the model parameter $\epsilon$ is sufficiently
large (i.e., that the diffusive zone around crack surfaces must be sufficiently
thick), we have shown that a contraction of successive difference functions in energy norms can
be obtained from the proposed scheme. This result implies the algorithm is
converging monotonically with a linear convergence rate. However, in the convergence analysis there appears some unknown constants which makes the precise convergence rate, as well as the precise lower bound on $\epsilon$ unknown. 

We provide detailed numerical tests where our proposed scheme is employed on
several phase field brittle fracture bench-mark problems. 
For each numerical example we provide findings for different values of 
stabilization parameters. For most cases we let $L_u = L_\phi > 0$, but for comparison we include also 
for the stabilization configurations $L_u = 0$ with $L_\phi >
0$, and $L_u = L_\phi = 0$. For the test cases 
presented here, there is only Example 1 where $L_u = 0$ does not
work. This might be due to 
the very rapid crack growth, which sets Example 1 apart from Examples 2 and
3. In this regard, we 
conclude that further work is needed to find an optimal configuration of $L_u$
and $L_\phi$. For all numerical test we also provide computational justification for the assumption of bounded elastic mechanical energy. Furthermore, a slight dependency on $h$ in the iteration counts is observed in
the numerical tests, 
but this is expected since we use $\epsilon = 2h$, and as our analysis demonstrates, the convergence rate is
dependent on $\epsilon$. 
The variation in iteration numbers with mesh refinement is in any case
sufficiently small enough that we 
conclude our algorithm is robust with respect to mesh refinement.

Moreover, due to the iteration spikes at the critical loading steps, we have
included, for comparison, several results in which the iteration has been
truncated (labeled $LFI$ in Examples 1-3). 
Due to the monotonic convergence of the scheme, this strategy still produces
acceptable results, 
while effectively avoiding the iteration spikes. We therefore conclude, at
least for the 
particular examples presented here, that a truncation of the $L$-scheme can be
employed for 
greatly improved efficiency with only negligible 
(depending on the situation at hand, of course) loss of accuracy.

\section*{Acknowledgements}
This work forms part of Research Council of Norway project 250223. The authors
also acknowledges the support from the University of Bergen.
The first author, MKB, thanks the group `Wissenschaftliches Rechnen' of 
the Institute of Applied Mathematics of the Leibniz University Hannover 
for the hospitality during his research stay from Oct - Dec 2018.
The second author, TW, has been supported by the German Research Foundation, Priority Program 1748 (DFG SPP 1748) named
\textit{Reliable Simulation Techniques in Solid Mechanics. Development of
Non-standard Discretization Methods, Mechanical and Mathematical Analysis}. 
The subproject within the SPP1748 reads 
\textit{Structure Preserving Adaptive Enriched Galerkin Methods for 
Pressure-Driven 3D Fracture Phase-Field Models} (WI 4367/2-1).

\enlargethispage{0.2cm}

\bibliographystyle{siamplain}
\bibliography{ref}

\end{document}